\documentclass[10pt]{article}
\usepackage[utf8]{inputenc}
\usepackage{epsfig}
\usepackage{color}
\usepackage[british,english]{babel}
\usepackage{amsthm}
\usepackage{amsmath}
\usepackage{amsfonts}
\usepackage{amssymb}
\usepackage{graphicx}
\usepackage{subcaption} 
\usepackage{array}
\usepackage{rotating}
\usepackage{multirow}

\usepackage{fancyhdr}

\newtheorem{corollary}{Corollary}[section]
\newtheorem{definition}[corollary]{Definition}

\newtheorem{lemma}[corollary]{Lemma}
\newtheorem{proposition}[corollary]{Proposition}
\newtheorem{remark}[corollary]{Remark}
\newtheorem{theorem}[corollary]{Theorem}

\newcommand{\R}{\mathbb{R}}

\newcommand{\vps}{\varepsilon}
\newcommand{\eps}{\varepsilon}

\newcommand{\etainf}{\eta_\infty}

\numberwithin{equation}{section}

\date{}
\begin{document}
\title{Influence of the Reynolds number on non-Newtonian flow in thin porous media}

\maketitle

\vskip-30pt
  \centerline{Mar\'ia ANGUIANO\footnote{Departamento de An\'alisis Matem\'atico. Facultad de Matem\'aticas.
Universidad de Sevilla, 41012 Sevilla (Spain)
anguiano@us.es},
Matthieu BONNIVARD\footnote{Ecole Centrale de Lyon, CNRS, INSA Lyon, Universite Claude Bernard Lyon 1, Université Jean Monnet, ICJ UMR5208,
69130 Ecully, France,
matthieu.bonnivard@ec-lyon.fr}, and  Francisco J. SU\'AREZ-GRAU\footnote{Departamento de Ecuaciones Diferenciales y An\'alisis Num\'erico. Facultad de Matem\'aticas. Universidad de Sevilla, 41012 Sevilla (Spain) fjsgrau@us.es}}

 \renewcommand{\abstractname} {\bf Abstract}
\begin{abstract} 


We study the effect of the Reynolds number on the flow of a generalized Newtonian fluid through a thin porous medium in $\R^3$. This medium is a domain of thickness $\eps\ll 1$, which is perforated by periodically distributed solid cylinders of size $\eps$. We consider the nonlinear stationary Navier–Stokes system with viscosity following the Carreau law. Using the tools of homogenisation theory and assuming that the Reynolds number scales as $\eps^{-\gamma}$, where $\gamma$ is a real constant, we prove the existence of a critical Reynolds number of order $1/\eps$, in the sense that the inertial term in the Navier–Stokes system has no influence in the limit if the Reynolds number is of order smaller than or equal to $1/\eps$ (i.e., $\gamma=1$). In this case, we derive linear or nonlinear Darcy's laws connecting velocity to pressure gradient. Conversely, we expect a contribution from the inertial term in the homogenised problem if the Reynolds number is greater than $1/\eps$. Finally, we propose a numerical method to solve nonlinear Darcy's law describing effective flow in the critical case and demonstrate its practical applicability by applying it to several examples.

\end{abstract}

 {\small \bf AMS classification numbers: } 76A05, 76M50, 76A20, 35B27.
 
  \section {Introduction}
  
In many industrial processes it is necessary to model the behavior of generalized Newtonian fluids through thin porous media, for instance in injection moulding of melted polymers, flow of oils, muds, etc. (see Prat and Aga${\rm \ddot{e}}$sse \cite{Prat} or Yeghiazarian {\it et al.} \cite{Rosati} for more details). An important question that arises is to know for what values of the Reynolds number the inertial terms of the Navier-Stokes system are negligible, and, in this way, to be able to model such flows by means of Darcy's laws, either linear or nonlinear (see for instance Dybbs and Edwards \cite{Dybbs}). This question is addressed in this article for a generalized Newtonian fluid flow governed by the nonlinear stationary Navier-Stokes system with viscosity following the Carreau law through thin porous media $\Omega_\varepsilon\subset \mathbb{R}^3$, which is a domain of thickness $\varepsilon$ perforated by periodically distributed solid cylinders of size $\varepsilon$, with $\varepsilon\ll 1$. Our analysis points a critical Reynolds number of order $1/\varepsilon$ above which the inertial terms must be taken into account in the modeling and, as a consequence, such fluid flow cannot be modeled by means of Darcy's laws. Furthermore, for the rest of the Reynolds number values, we obtain that the inertial terms are negligible and we derive different models.  
 
The incompressible generalized Newtonian fluids is a type of non-Newtonian fluids which are characterized by the viscosity depending on the principal invariants of the  symmetric stretching tensor $\mathbb{D}[u]$. If $u=u(x',x_3)$ is the velocity, $p$ the pressure and $Du$ the gradient velocity tensor, $\mathbb{D}[u]=(Du+D^tu)/2$  denotes the  symmetric stretching tensor  and $\sigma$ the stress tensor given by $\sigma=-pI+2\eta_{r}\mathbb{D}[u]$. The viscosity $\eta_r$ is constant for a Newtonian fluid but dependent of the shear rate, i.e., $\eta_r=\eta_{r}(\mathbb D[u])$, for viscous generalized Newtonian fluids. The deviatoric stress tensor $\tau$, i.e., the part of the total stress tensor that is zero at equilibrium, is then a nonlinear function of the shear rate $\mathbb D[u]$, i.e., $
\tau=\eta_r(\mathbb D[u]) \mathbb D[u]$  (see Barnes {\it et al.} \cite{Barnes}, Bird {\it et al.} \cite{Bird} and Mikeli\'c \cite{Mikelic1} for more details).

An important model is the well-known  {\it Carreau law}, which will be considered in this article and is defined by
\begin{equation}\label{Carreau}\eta_r(\mathbb{D}[u])=(\eta_0-\eta_\infty)(1+\lambda|\mathbb{D}[u]|^2)^{{r\over 2}-1}+\eta_\infty,\quad 1<r<+\infty,\quad \eta_0> \eta_\infty>0,\quad \lambda>0,
\end{equation} 
where $\eta_\infty$ is the high-shear-rate limit of the viscosity,   the parameter $\lambda$ is a time constant and $r-1$ is  describes the slope in the power law region. The matrix norm $|\cdot|$ is defined by $|\xi|^2=Tr(\xi \xi^t)$ with $\xi \in \mathbb{R}^3$.  

Observe that for $r=2$, the Carreau law (\ref{Carreau}) is linear, i.e., $\eta_2(\mathbb{D}[u])=\eta_0$, which represents {\it Newtonian fluids}.
We recall that the generalized Newtonian fluids are classified in two main categories (see Saramito \cite[Chapter 2]{Saramito}): 
\begin{itemize}
\item[--]  {\it Pseudoplastic} or {\it shear thinning} fluids, where the viscosity decreases with the shear rate, which correspond to the case of flow index  $1<r<2$. 

\item[--]  {\it Dilatant} or {\it shear thickening} fluids, where the viscosity increases with the shear rate, which correspond to the case of flow index  $r>2$.
\end{itemize}
In this article, we consider a thin porous medium $\Omega_\varepsilon =\omega_\varepsilon\times (0,\varepsilon)\subset \R^3$ of small height $\varepsilon$, perforated by periodically distributed solid cylinders of diameter $\varepsilon$. We assume that a non-Newtonian fluid, whose viscosity follows the {\it Carreau law}~\eqref{Carreau}, flows through this porous medium. We model this flow by the following stationary Navier-Stokes system, completed with Dirichlet boundary conditions on the exterior boundary and the cylinders:
\begin{equation}\label{system_intro}
\left\{\begin{array}{rl}\displaystyle
\medskip
-{\color{black} \mu \varepsilon^\gamma}{\rm div}\left(\eta_r(\mathbb{D}[u])\mathbb{D}[u]\right)+(  u\cdot \nabla)u+ \nabla p = f & \hbox{in }\Omega_\varepsilon,\\
\medskip
\displaystyle
{\rm div}\,u=0& \hbox{in }\Omega_\varepsilon.
\end{array}\right.
\end{equation}
In the above system, $f$ is an external body force and the Reynolds number is proportional to $\mu^{-1} \varepsilon^{-\gamma}$, where the exponent $\gamma\in \R$ is fixed.

In \cite{Tapiero2} Boughanim and Tapi\'ero consider the system (\ref{system_intro}) where $\Omega_\varepsilon$ is a thin slab of thickness $\eps$. By using dimension reduction and homogenization techniques, they show that there is a critical Reynolds number of order $1/\varepsilon^{3\over 2}$ (i.e., $\gamma=3/2$) for pseudoplastic fluids and of order $1/\varepsilon^{1+r\over 2}$ (i.e., $\gamma={1+r\over 2}$) for dilatant fluids. The inertial term of the Navier-Stokes system (\ref{system_intro}) has no influence at the limit for the Reynolds number of order smaller than critical Reynolds number and they studied the limit when the thickness tends to zero. They prove that the averaged limit velocity satisfies according to the order of Reynolds number and the power $r$ of the Carreau law, a linear or nonlinear 2D Reynold's law of power or Carreau type.

On the other hand, in \cite{Bourgeat1} Bourgeat and  A. Mikeli\'c consider the system (\ref{system_intro}) where $\Omega_\varepsilon$ is a porous medium obtained by the periodic repetition of an elementary cell of size $\varepsilon$. Using the tools of homogenization theory, they show that there is a critical Reynolds number of order $1/\varepsilon^{3\over 2}$ (i.e., $\gamma=3/2$) for pseudoplastic or Newtonian fluids and of order $1/\varepsilon^{1+{r\over 4}}$ (i.e., $\gamma=1+{r\over 4}$) for dilatant fluids. The inertial term of the Navier-Stokes system (\ref{system_intro}) has no influence at the limit for the Reynolds number of order smaller than critical Reynolds number and depending on the Reynolds number they derive several averaged momentum equations, linear or nonlinear and nonlocal connecting the velocity to the pressure gradient.

However, to our knowledge, there does not seem to be in the literature any study of the asymptotic behavior of a solution of (\ref{system_intro}) where $\Omega_\varepsilon$ is the thin porous medium described above. To perform such analysis, we will combine the relevant methods used to handle porous media and thin films. For this purpose, we first perform a change of variables which consists in stretching in the $x_3$-direction by a factor $1/\varepsilon$. As in \cite{Anguiano_SuarezGrau}, we use the dilatation in the variable $x_3$, i.e., $y_3=x_3/\varepsilon$, in order to have the functions $\tilde u_\varepsilon$ defined in the open set fixed height $\widetilde{\Omega}_{\varepsilon}$. In this sense, we have that the inertial term is transformed into
\begin{eqnarray}\label{inertial_term_intro}
-\int_{\widetilde \Omega_\varepsilon} \tilde u_\varepsilon\tilde \otimes \tilde u_\varepsilon:D_{x'} v\,dx'dy_3
\displaystyle +{1\over \varepsilon}\left(\int_{\widetilde \Omega_\varepsilon}\partial_{y_3} \tilde u_{\varepsilon,3}\tilde  u_\varepsilon  v\,dx'dy_3+ \int_{\widetilde \Omega_\varepsilon}\tilde u_{\varepsilon,3}\partial_{y_3} \tilde u_\varepsilon\,v\,dx'dy_3\right)\,,
\end{eqnarray}
where $(u \tilde \otimes w)_{ij}=u_i w_j$, $i=1,2$, $j=1,2,3$ and $v$ is a test function.

Later, due to the perforations of the domain, we have to use the restricted operator $\tilde R_q^{\varepsilon}$ in order to extend the pressure to a fixed domain $\Omega$ without perforations (see Lemma \ref{restriction_operator} for more details). Then, in order to find a critical Reynolds number, we have to prove some a priori estimates for the extended pressure and we have to analyze the term (\ref{inertial_term_intro}) with $v=\tilde R_q^{\varepsilon}$.  We would like to highlight that analyzing this inertial term involves additional (nontrivial) difficulties which does not allow us to carry out a complete study of (\ref{system_intro}) for all values $r$. We can successfully solve this problem using advanced interpolation techniques which allow us to find a critical Reynolds number.
In particular, we get the following progress:
\begin{itemize}
\item[(i)]  For {\it Pseudoplastic} or {\it Newtonain} fluids: if we argue in a standard way to obtain estimates of the extension of the pressure, we get estimates in $L^2(\Omega)$ if $\gamma\leq 3/4$. But we have improved this result by obtaining estimates in $L^{3\over 2\gamma}(\Omega)$ if $3/4<\gamma\leq 1$ (see Proposition \ref{Estimates_extended_lemma} for more details). In this sense, the inertial term of the Navier-Stokes system has no influence at the limit for the Reynolds number of order smaller or equal than $1/\varepsilon$ (i.e., $\gamma\leq 1$) and we derive the limit laws. In particular, for pseudoplastic fluids, if $\gamma<1$ we obtain a linear 2D Darcy'law with viscosity $\eta_0$ and if $\gamma=1$ we obtain a nonlinear 2D Darcy's law of Carreau type. For Newtonian fluids, if $\gamma\leq 1$, we get a linear 2D Darcy's law with viscosity $\eta_0$ (see Theorem \ref{mainthmPseudo} for more details).
\item[(ii)]  For {\it Dilatant} fluids: if we argue in a standard way to obtain estimates of the extension of the pressure, we get estimates in $L^{r'}(\Omega)$, with $1/r+1/r'=1$, for $r\ge 9/4$ and $\gamma\leq 1$. But we have got a complete study for all $r$ because we have improved this result by obtaining estimates in $L^{3r\over 6-r}(\Omega)$ for $2<r<9/4$ and $\gamma\leq 1$ (see Proposition \ref{Estimates_extended_lemma2} for more details). In this sense, the inertial term of the Navier-Stokes system has no influence at the limit for the Reynolds number of order smaller or equal than $1/\varepsilon$ (i.e., $\gamma\leq 1$) and we derive the limit laws. In particular, if $\gamma<1$ we obtain a linear 2D Darcy'law with viscosity $\eta_0$ and if $\gamma=1$ we obtain a nonlinear 2D Darcy's law of Carreau type (see Theorem \ref{mainthmDilatant} for more details).

\end{itemize}

Taking into account that $1/\varepsilon^{\gamma}$ is the Reynolds number, in Table \ref{table_Reynold}, we summarize the critical Reynolds number of the Carreau fluid  governing by (\ref{system_intro}) depending on the type of domain $\Omega_\varepsilon$ and the type of fluid:\\

\begin{table}[h!]\centering
\begin{tabular}{|c||c|c|}
\hline
 &   $1<r\leq 2$  &  $r>2$\\
 \hline \hline
 {\small Thin slab} & $1/\varepsilon^{{3\over 2}}$ & $1/\varepsilon^{{1+r\over 2}}$\\  \cline{1-3}
  {\small Porous media} & $1/\varepsilon^{{3\over 2}}$ & $1/\varepsilon^{1+{r\over 4}}$\\  \cline{1-3}
  {\small Thin porous media} &  $1/\varepsilon$& $1/\varepsilon$\\ \cline{1-3}
  \hline
\end{tabular} 
\caption{Critical Reynolds number of the Carreau fluid depending on the type of domain and the type of fluid.}
\label{table_Reynold}
\end{table}

\begin{remark}
If the inertial term did not exist in (\ref{system_intro}), i.e., if we consider the Stokes system, then we do not have the restriction $\gamma\leq 1$ and we have a complete study for all $1<r<+\infty$ and for all $\gamma\in \R$. In addition to the laws already obtained for $\gamma\leq 1$, we obtain a linear 2D Darcy's law with viscosity $\eta_\infty$ for pseudoplastic fluids and $\gamma>1$, a linear 2D Darcy's law with viscosity $\eta_0$ for Newtonian fluids and $\gamma>1$ and a nonlinear 2D Darcy's law of power law type for dilatant fluids and $\gamma>1$ (see the recent articles Anguiano {\it et al.} \cite{Carreau_Ang_Bonn_SG,Carreau_Ang_Bonn_SG2} for more details).
\end{remark}

We finish the introduction with a list of references of other recent studies concerning thin porous media. Some stationary models for different fluids are obtained in Anguiano and Su\'arez-Grau \cite{Anguiano_SuarezGrau_Derivation, Anguiano_SuarezGrau_CMS, Anguiano_SG_Net}, Fabricius {\it et al.} \cite{Fabricius, Fabricius2}, Zhengan and Hongxing \cite{ZZ}, and some non-stationary models are developed in Anguiano \cite{Anguiano_MMAS, Anguiano_ZAMM, Anguiano_Derivation, Anguiano_European, Anguiano_BMMS}. The case of a Bingham flow is considered in Anguiano and Bunoiu \cite{Ang-Bun,Ang-Bun2} and the case of micropolar fluids in Su\'arez-Grau \cite{SG_micropolar, SG_MN}. The two-phase flow problem in thin porous media domains of Brinkman type has been considered in Armiti-Juber \cite{Armiti} and an approach for effective heat transport in thin porous media has been derived by  Scholz and Bringedal \cite{Scholz}.

The structure of the paper is as follows. In Section \ref{sec:main} we introduce the domain, make the statement of the problem and give the main results (Theorems \ref{mainthmPseudo} and \ref{mainthmDilatant}). The  proofs of the main results are provided in Section \ref{sec:proofs}. Section~\ref{Sect:Numerics} is devoted to the numerical resolution of the nonlinear Darcy's law~\eqref{thm:system_gamma1}. We finish the paper with a list of references.

\section{Setting of the problem and main result}\label{sec:main}
\paragraph{Geometrical setting.} The periodic porous medium is defined by a domain $\omega$ and an associated microstructure, or periodic cell $Y^{\prime}=(-1/2,1/2)^2$, which is made of two complementary parts: the fluid part $Y^{\prime}_{f}$, and the solid part $T^{\prime}$ ($Y^{\prime}_f  \bigcup T^{\prime}=Y^\prime$ and $Y^{\prime}_f  \bigcap T^{\prime}=\emptyset$). More precisely, we assume that $\omega$ is a smooth, bounded and connected subset of $\R^2$, and that $T^{\prime}$ is an open connected subset of $Y^\prime$ with a smooth boundary $\partial T^\prime$, such that $\overline{T^\prime}$ is strictly included  in $Y^\prime$.

The microscale of the porous medium is a small positive number ${\varepsilon}$. The domain $\omega$ is covered by a regular square mesh of size ${\varepsilon}$: for $k^{\prime}\in \mathbb{Z}^2$, each cell $Y^{\prime}_{k^{\prime},{\varepsilon}}={\varepsilon}k^{\prime}+{\varepsilon}Y^{\prime}$ is divided in a fluid part $Y^{\prime}_{f_{k^{\prime}},{\varepsilon}}$ and a solid part $T^{\prime}_{k^{\prime},{\varepsilon}}$, i.e., is similar to the unit cell $Y^{\prime}$ rescaled to size ${\varepsilon}$. We define $Y=Y^{\prime}\times (0,1)\subset \R^3$ and divide it in a fluid part $Y_{f}=Y'_f\times (0,1)$ and a solid part $T=T'\times(0,1)$. Similarly, we set $Y_{k^{\prime},{\varepsilon}}=Y^{\prime}_{k^{\prime},{\varepsilon}}\times (0,1)\subset \R^3$ and divide it in a fluid part $Y_{f_{k^{\prime}},{\varepsilon}}$ and a solid part $T_{{k^{\prime}},{\varepsilon}}$.

We denote by $\tau(\overline T'_{k',\varepsilon})$ the set of all translated images of $\overline T'_{k',\varepsilon}$. The set $\tau(\overline T'_{k',\varepsilon})$ represents the obstacles in $\R^2$.

The fluid part of the bottom $\omega_{\varepsilon}\subset \R^2$ of the porous medium is defined by $\omega_{\varepsilon}=\omega\backslash\bigcup_{k^{\prime}\in \mathcal{K}_{\varepsilon}} \overline{T^{\prime}_{{k^{\prime}},{\varepsilon}}},$ where $\mathcal{K}_{\varepsilon}=\{k^{\prime}\in \mathbb{Z}^2: Y^{\prime}_{k^{\prime}, {\varepsilon}} \cap \omega \neq \emptyset \}$.  The whole fluid part $\Omega_{\varepsilon}\subset \R^3$ of the thin porous medium is then given by 
\begin{equation*}
\Omega_{\varepsilon}=\{  (x_1,x_2,x_3)\in \omega_{\varepsilon}\times \R: 0<x_3<\varepsilon \}.
\end{equation*}
We assume that the obstacles $\tau(\overline T'_{k',\varepsilon})$ do not intersect the boundary $\partial\omega$ and we denote by $S_\varepsilon$ the set of the solid cylinders contained in $\Omega_\varepsilon$, i.e., $S_\varepsilon=\bigcup_{k^{\prime}\in \mathcal{K}_{\varepsilon}} T'_{k^\prime, \varepsilon}\times (0,\varepsilon)$.
\\

We define 
\begin{equation*}
\widetilde{\Omega}_{\varepsilon}=\omega_{\varepsilon}\times (0,1), \quad \Omega=\omega\times (0,1), \quad Q_\varepsilon=\omega\times (0,\varepsilon).
\end{equation*}
We observe that $\widetilde{\Omega}_{\varepsilon}=\Omega\backslash \bigcup_{k^{\prime}\in \mathcal{K}_{\varepsilon}} \overline {T_{{k^{\prime}}, {\varepsilon}}},$ and we define $T_\varepsilon=\bigcup_{k^{\prime}\in \mathcal{K}_{\varepsilon}} T_{k^\prime, \varepsilon}$ as the set of solid cylinders contained in $\widetilde \Omega_\varepsilon$.

\paragraph{Extra notation} The points $x\in\R^3$ will be decomposed as $x=(x^{\prime},x_3)$ with $x^{\prime}=(x_1,x_2)\in \R^2$, $x_3\in \R$. We also use the notation $x^{\prime}$ to denote a generic vector in $\R^2$. $L^{q}_0$  is the space of functions in $L^q$  with zero mean value. For $1<q<\infty$, we denote by $C^\infty_{\#}(Y)$ the space of infinitely differentiable functions in $\R^2\times (0,1)$, which are $Y'$-periodic. Then, $L^q_{\#}(Y)$ (resp. $W^{1,q}_{\#}(Y)$) stands for the completion of $C^\infty_{\#}(Y)$ in the $L^q(Y)$ norm (resp., in the $W^{1,q}(Y)$ norm). $L^q_{0,\#}(Y)$ is the space of functions in $L^q_{\#}(Y)$ with zero mean value.

\paragraph{Statement of the problem.}  We consider the stationary Navier-Stokes system in the thin porous medium $\Omega_\varepsilon$,  with a nonlinear viscosity following the {\it Carreau law}~\eqref{Carreau}, and scaled by a factor $\mu \eps^\gamma$, where $\mu$ is a positive constant and $\gamma$ is a fixed real exponent. We assume that the flow is driven by an external force $f$ and impose Dirichlet boundary conditions on both the exterior boundary $\partial Q_\varepsilon$ and the cylinders' boundary $\partial S_\varepsilon$. Hence, the model reads
\begin{equation}\label{1}
\left\{\begin{array}{rl}\displaystyle
\medskip
-{\color{black} \mu  \varepsilon^\gamma}{\rm div}\left(\eta_r(\mathbb{D}[u_\varepsilon])\mathbb{D}[u_\varepsilon]\right)+(  u_\varepsilon\cdot \nabla)u_\varepsilon+ \nabla p_\varepsilon = f & \hbox{in }\Omega_\varepsilon,\\
\medskip
\displaystyle
{\rm div}\,u_\varepsilon=0& \hbox{in }\Omega_\varepsilon,\\
\medskip
\displaystyle
u_\varepsilon=0& \hbox{on }\partial Q_\varepsilon\cup \partial S_\varepsilon,
\end{array}\right.
\end{equation}
As is usual when dealing with thin domains, we neglect the vertical component of $f$ and its dependence on the third variable $x_3$ (see~\cite{Tapiero2} for more details). In other words, we assume that $f$ is of the form
\begin{equation*}
f (x)=(f'(x'),0)\quad \hbox{with }f'\in L^\infty(\omega)^2.
\end{equation*}

\begin{remark}Under the previous assumptions, the classical theory (see for instance~\cite{Tapiero2,Bourgeat1, Lions2}) ensures the existence of at least one weak solution $(u_\varepsilon, p_\varepsilon)\in H^1_0(\Omega_\varepsilon)^3\times L^{2}_0(\Omega_\varepsilon)$, for $1< r\leq 2$, and $(u_\varepsilon, p_\varepsilon)\in W^{1,r}_0(\Omega_\varepsilon)^3\times L^{r'}_0(\Omega_\varepsilon)$ with $1/r+1/r'=1$, for $r> 2$.
\end{remark}

Our goal is to study the asymptotic behavior of $(u_{\varepsilon}, p_{\varepsilon})$ when $\varepsilon$  tends to zero. For this purpose, we use the dilatation in the variable $x_3$ as follows
\begin{equation}\label{dilatacion}
y_3=\frac{x_3}{\varepsilon},
\end{equation}
in order to redefine the velocity and pressure on the set $\widetilde\Omega_{\varepsilon}$ with fixed height $1$. Namely, we define $\tilde{u}_{\varepsilon}$ and $\tilde{p}_{\varepsilon }$ by $$\tilde{u}_{\varepsilon }(x^{\prime},y_3)=u_{\varepsilon }(x^{\prime}, \varepsilon y_3),\text{\ \ }\tilde{p}_{\varepsilon }(x^{\prime},y_3)=p_{\varepsilon }(x^{\prime}, \varepsilon y_3), \text{\ \ } a.e.\text{\ } (x^{\prime},y_3)\in \widetilde{\Omega}_{\varepsilon }.$$
In order to write the equations satisfied by these rescaled functions, we introduce the operators $\mathbb{D}_{\varepsilon}$, $D_{\varepsilon}$, ${\rm div}_{\varepsilon}$ and $\nabla_{\varepsilon}$, defined for any vectorial function $v=(v',v_3)$ and any scalar function $w$ by
\begin{equation*}
\mathbb{D}_{\varepsilon}\left[v\right]=\frac{1}{2}\left(D_{\varepsilon} v+D^t_{\varepsilon} v \right),
\end{equation*}
\begin{equation*}
(D_{\varepsilon}v)_{i,j}=\partial_{x_j}v_i\text{\ for \ }i=1,2,3,\ j=1,2,\quad 
(D_{\varepsilon}v)_{i,3}=\varepsilon^{-1}\partial_{y_3}v_i\text{\ for \ }i=1,2,3,
\end{equation*}
\begin{equation*}
{\rm div}_{ \varepsilon}v={\rm div}_{x^{\prime}}v^{\prime}+  \varepsilon^{-1}\partial_{y_3}v_3,\quad
\nabla_{ \varepsilon}w=(\nabla_{x^{\prime}}w,  \varepsilon^{-1}\partial_{y_3}w)^t.
\end{equation*}

Using the above operators, the pair $(\tilde{u}_{\varepsilon }, \tilde{p}_{\varepsilon })$ is a solution of the following system:
\begin{equation}
\left\{
\begin{array}
[c]{r@{\;}c@{\;}ll}%
\medskip
\displaystyle -{\color{black} \mu\varepsilon^\gamma} {\rm div}_{ \varepsilon} \left(  \eta_r(\mathbb{D}_\varepsilon[\tilde u_\varepsilon])\mathbb{D}_{ \varepsilon}\left[\tilde{u}_{\varepsilon}\right] \right)+( \tilde  u_\varepsilon\cdot \nabla_\varepsilon)\tilde u_\varepsilon+ \nabla_{ \varepsilon} \tilde{p}_{\varepsilon } &
= &
f \text{\ \ in  }\widetilde{\Omega}_{\varepsilon },\\
\medskip
{\rm div}_{\varepsilon} \tilde{u}_{\varepsilon } & = & 0 \text{\ \ in  }\widetilde{\Omega}_{\varepsilon },\\
\medskip
\tilde{u}_{\varepsilon }& = &0 \text{\ \ on  } \partial {\Omega}\cup \partial T_\varepsilon.
\end{array}
\right. \label{2}%
\end{equation}
Our goal then is to describe the asymptotic behavior of   $(\tilde{u}_{\varepsilon}, \tilde{p}_{\varepsilon})$. Since these functions are defined on a varying set $\widetilde{\Omega}_{\varepsilon}$, the first step is to extend them to the whole domain $\Omega$. We denote these extensions by $(\tilde{u}_{\varepsilon}, \tilde{P}_{\varepsilon})$ (in particular, we use the same notation for the velocity in $\widetilde\Omega_\varepsilon$ and its continuation in $\Omega$).

Our main results are given by the following theorems. 

\begin{theorem}[Pseudoplastic or Newtonian fluids]\label{mainthmPseudo}
Let $1< r\leq 2$,  $\gamma\leq1$ and $C(r)$ be defined as follows:
\begin{equation}\label{Cr_case_pseudo}
C(r)=\left\{\begin{array}{rcl}
2&\hbox{if}& \gamma\leq \displaystyle{3\over 4},
\\
\\
\displaystyle {3\over 2\gamma}&\hbox{if}& \displaystyle{3\over 4}<\gamma\leq 1.
\end{array}\right.
\end{equation}
Then, there exist $\tilde u\in H^1_0(0,1;L^2(\omega)^3)$, with $\tilde u=0$ on $\omega\times \{0,1\}$ and $\tilde u_3\equiv 0$, and $\tilde P\in L^{C(r)}_0(\omega)$, such that the extension $(\tilde u_{\varepsilon},\tilde P_{\varepsilon})$ of a solution of (\ref{2}) satisfies the following convergences:
$$\varepsilon^{\gamma-2}\tilde u_{\varepsilon }\rightharpoonup \tilde  u\quad\hbox{weakly in } H^1(0,1;L^2(\omega)^3),\quad \tilde   P_{\varepsilon }\to \tilde P\quad\hbox{strongly in }L^{C(r)}(\Omega).$$
Besides, define $V:\omega\rightarrow \R^3$ by $V(x')=\int_0^1\tilde u(x',y_3)\,dy_3$. Then, $V_3\equiv 0$ and depending on the values of $r$ and $\gamma$, the filtration velocity $V'$ is linked to $\widetilde P$ through a different two-dimensional Darcy law.
\begin{itemize}
\item[--] For pseudoplastic fluids ($1< r<2$):
\begin{itemize}
\item[(i)] If $\gamma<1 $, $(V', \tilde P)\in L^2(\omega)^2\times (L^{2}_0(\omega)\cap H^1(\omega))$ is the unique solution of the linear effective 2D Darcy's law 
\begin{equation}\label{thm:system}
\left\{\begin{array}{l}
\medskip
\displaystyle
V'(x')={1\over {\color{black}\mu}\eta_0}\mathcal{A}\left(f'(x')-\nabla_{x'}\tilde P(x')\right)\quad  \hbox{in }\omega.\\
\medskip
\displaystyle
{\rm div}_{x'} V'(x')=0\ \hbox{ in }\omega,\quad V'(x')\cdot n=0\ \hbox{ on }\partial\omega,
\end{array}
\right.
\end{equation}
The permeability tensor $\mathcal{A}\in\R^{2\times 2}$ is defined by its entries
\begin{equation}\label{permfuncNew}
\mathcal{A}_{ij}=\int_{Y_f}w^i_j(y)\,dy,\quad i,j=1,2,
\end{equation}
where for $i=1,2$, $(w^i, \pi^i)\in H^1_{0,\#}(Y_f)^3\times L^2_{0,\#}(Y_f)$ is the unique solution of the local Stokes system
\begin{equation}\label{LocalProblemNewtonian}
\left\{\begin{array}{rl}
\medskip
\displaystyle
-\Delta_yw^i +\nabla_{y}\pi^i=e_i &\hbox{in }Y_f,
\\
\medskip
\displaystyle
{\rm div}_y w^i=0&\hbox{in }Y_f,
\\
\medskip
\displaystyle
w^i=0&\hbox{on }\partial T\cup (Y_f'\times \{0,1\}),
\\
\medskip
\displaystyle y\to  w^i, \pi^i & Y-\hbox{periodic},
\end{array}\right.
\end{equation}
and $\{e_i\}_{i=1,2,3}$ is the canonical basis of $\R^3$. 

\item[(ii)] If $\gamma=1$, $(V', \tilde P)\in L^2(\omega)^2\times (L^{2}_0(\omega)\cap H^1(\omega))$  is the unique solution of the nonlinear 2D Darcy's law of Carreau type
\begin{equation}\label{thm:system_gamma1}
\left\{\begin{array}{l}
\medskip
\displaystyle
V'(x')= \mathcal{U}\left(f'(x')-\nabla_{x'}\tilde P(x')\right) \quad \hbox{in }\omega,\\
\medskip
\displaystyle
{\rm div}_{x'} V'(x')=0\ \hbox{ in }\omega,\quad V'(x')\cdot n=0\ \hbox{ on }\partial\omega.
\end{array}
\right.
\end{equation}
The permeability function $\mathcal{U}:\R^2\to \R^2$ is defined by
\begin{equation}\label{permfunc123}
\mathcal{U}(\xi')=\int_{Y_f}w'_{\xi'}(y)\,dy,\quad\forall\,\xi'\in\R^2,
\end{equation}
where for every $\xi'\in\R^2$, $(w_{\xi'}, \pi_{\xi'})\in H^1_{0,\#}(Y_f)^3\times L^{2}_{0,\#}(Y_f)$ is the unique solution of the local Stokes system \begin{equation}\label{LocalProblemNonNewtonian}
\left\{\begin{array}{rl}
\medskip
\displaystyle
-{\color{black} \mu}\,{\rm div}_y(\eta_r(\mathbb{D}_y[w_{\xi'}])\mathbb{D}_y[w_{\xi'}]) +\nabla_{y}\pi_{\xi'}=\xi' &\hbox{in }Y_f,
\\
\medskip
\displaystyle
{\rm div}_yw_{\xi'}=0&\hbox{in }Y_f,
\\
\medskip
\displaystyle
w_{\xi'}=0&\hbox{on }\partial T\cup (Y_f'\times \{0,1\}),
\end{array}\right.
\end{equation}
and the nonlinear viscosity $\eta_r$ is given by the  Carreau law (\ref{Carreau}).
\end{itemize}
\item[--] For Newtonian fluids ($r=2$) and $\gamma\leq 1$, $(V', \tilde P)\in L^2(\omega)^3\times (L^{2}_0(\omega)\cap H^1(\omega))$ is the unique solution of the linear effective 2D Darcy's law \eqref{thm:system} and $\mathcal{A}$ defined by \eqref{permfuncNew}.
\end{itemize}
\end{theorem}

\begin{remark}
 According to \cite[Theorem 1.1]{Allaire0}, the permeability tensor $\mathcal{A}$ is  symmetric and positive definite. 
\end{remark}
\begin{remark}
Note that in the case $r=2$ (i.e., for Newtonian fluids), the linear Darcy law obtained (\ref{thm:system}) and the permeability tensor (\ref{permfuncNew}) are exactly those obtained in~\cite[Theorem 17]{AnguianoSG_pressure}.
\end{remark}

\begin{theorem}[Dilatant fluids]\label{mainthmDilatant}
Let $r> 2$, $\gamma\le 1$ and $C(r)$ be defined as follows:
\begin{equation}\label{Cr_case_dilatant}
C(r)=\left\{\begin{array}{ccl}
\displaystyle {3r\over 6-r}&\hbox{for}&  2<r<\displaystyle {9\over 4},
\\
\\
r'&\hbox{for}&  \displaystyle r\ge {9\over 4}.
\end{array}\right.
\end{equation}
The result of the theorem depends on the value of  $\gamma$.
\begin{itemize}
\item[(i)]  If $\gamma<1$, there exist $\tilde u\in H^1(0,1;L^2(\omega)^3)$, with $\tilde u=0$ on $\omega \times \{0, 1\}$ and $\tilde u_3\equiv 0$,  and $\tilde P\in L^{C(r)}_0(\omega)$, such that the extension $(\tilde u_{\eps},\tilde P_{\eps})$ of the solution of (\ref{2}) satisfies the following convergences:
$$\eps^{\gamma-2}\tilde u_{\eps }\rightharpoonup \tilde  u\quad\hbox{weakly in } H^1(0,1;L^2(\omega)^3),\quad \tilde   P_{\eps }\to \tilde P\quad\hbox{strongly in }L^{C(r)}(\Omega).$$
Besides, define $V:\omega\rightarrow \R^3$ by $V(x')=\int_0^1\tilde u(x',y_3)\,dy_3$. Then, $V_3\equiv 0$ and the pair $( V', \tilde P)\in L^2(\omega)^2\times (L^{2}_0(\omega)\cap H^1(\omega))$ is the unique solution of the linear 2D Darcy's law \eqref{thm:system}-\eqref{LocalProblemNewtonian}.

\item[(ii)] If $\gamma=1$, then there exist $\tilde u\in W^{1,r}_0(0,1;L^r(\omega)^3)$, with $\tilde u=0$ on $\omega \times \{0,1\}$ and $\tilde u_3\equiv 0$, and $\tilde P\in L^{C(r)}_0(\omega)$, such that the extension $(\tilde u_{\varepsilon},\tilde P_{\varepsilon})$ of a solution of (\ref{2}) satisfies the following convergences
$$\varepsilon^{-1}\tilde u_{\varepsilon }\rightharpoonup \tilde  u\quad\hbox{weakly in } W^{1,r}(0,1;L^r(\omega)^3),\quad \tilde   P_{\varepsilon }\to \tilde P\quad\hbox{strongly in }L^{C(r)}(\Omega).$$
Besides, define $V:\omega\rightarrow \R^3$ by $V(x')=\int_0^1\tilde u(x',y_3)\,dy_3$. Then, $V_3\equiv 0$ and $(V', \tilde P)\in L^r(\omega)^2\times (L^{r'}_0(\omega)\cap W^{1,r'}(\omega))$ is the unique solution of the lower-dimensional effective nonlinear 2D Darcy's law of Carreau type \eqref{thm:system_gamma1}. The permeability function $\mathcal{U}:\R^2\to \R^2$ is defined by \eqref{permfunc123}, where for every $\xi'\in\R^2$, $(w_{\xi'}, \pi_{\xi'})\in W^{1,r}_{0,\#}(Y_f)^3\times L^{r'}_{0,\#}(Y_f)$ is the unique solution of the local Stokes system \eqref{LocalProblemNonNewtonian} with a nonlinear viscosity given by the Carreau law \eqref{Carreau}.

\end{itemize}

\end{theorem}

\begin{remark}
 According to \cite[Lemma 2]{Bourgeat2}, the permeability function $\mathcal{U}$ is  coercive and strictly monotone. 
\end{remark}

\section{Proof of the main results}\label{sec:proofs}
In this section, we detail the proof of the main results (Theorems \ref{mainthmPseudo} and \ref{mainthmDilatant}). The first step is to derive a priori estimates on the solution of \eqref{2}, and to extend it to the domain $\Omega$. Then, we implement the unfolding method and establish compactness properties which are the key to conclude the proof.

\subsection{A priori estimates on the velocity}

In this subsection, we establish sharp a priori estimates on the rescaled velocity $\tilde u_\eps$. To this aim, let us first recall Poincar\'e and Korn inequalities in $\widetilde\Omega_\varepsilon$, which can be found in \cite{Anguiano_SuarezGrau}.
\begin{lemma}[Remark 4.3-(i) in \cite{Anguiano_SuarezGrau}] \label{Lemma_Poincare} The following inequalities hold.
\begin{itemize}
\item[(i)] Let $1\leq q<+\infty$. There exists a positive constant $C$ such that, for every $\eps$ and every $\tilde  v\in W^{1,q}_0(\widetilde \Omega_\varepsilon)^3$,
\begin{equation*}
\|\tilde v\|_{L^q(\widetilde \Omega_\varepsilon)^3}\leq C\varepsilon\|D_\varepsilon \tilde  v\|_{L^q(\widetilde \Omega_\varepsilon)^{3\times 3}}\quad (\hbox{Poincar\'e's inequality}).
\end{equation*}
\item[(ii)] Let $1< q<+\infty$. There exists a positive constant $C$ such that, for every $\eps$ and every $\tilde  v\in W^{1,q}_0(\widetilde \Omega_\varepsilon)^3$, 
\begin{equation*}
\|D_\varepsilon \tilde  v\|_{L^q(\widetilde \Omega_\varepsilon)^{3\times 3}}\leq C\|\mathbb{D}_\varepsilon[\tilde  v]\|_{L^q(\widetilde \Omega_\varepsilon)^{3\times 3}} \quad (\hbox{Korn's inequality}).
\end{equation*}
\end{itemize}
\end{lemma}
Using the previous inequalities, we give a priori estimates for the rescaled velocity $\tilde u_\varepsilon$ in $\widetilde \Omega_\varepsilon$.

\begin{lemma} \label{Estimates_lemma}  Let $\tilde u_\varepsilon$ be a weak solution of \eqref{2}. Then, the following estimates hold.

\begin{itemize}
\item[(i)] {\it (Pseudoplastic or Newtonian fluids)} Consider $1< r\leq 2$. There exists a positive constant $C$, independent of $\varepsilon$, such that for every value of $\gamma$,
\begin{equation}\label{estimates_u_tilde}
\|\tilde u_\varepsilon\|_{L^2(\widetilde \Omega_\varepsilon)^{3}}\leq C\varepsilon^{2-\gamma} ,\quad \|D_\varepsilon \tilde u_\varepsilon\|_{L^2(\widetilde \Omega_\varepsilon)^{3\times 3}}\leq C\varepsilon^{1-\gamma} ,\quad \|\mathbb{D}_\varepsilon[\tilde u_\varepsilon]\|_{L^2(\widetilde \Omega_\varepsilon)^{3\times 3}}\leq C\varepsilon^{1-\gamma} \,.
\end{equation}
\item[(ii)] {\it (Dilatant fluids)} Consider $r>2$. There exists a positive constant $C$, independent of $\varepsilon$, such that estimates (\ref{estimates_u_tilde}) hold. Depending on the value of $\gamma$, we also have the following extra estimates:
\begin{itemize}
\item If $\gamma<1$,
\begin{equation}\label{estimates_u_tilde2less1}
\|\tilde u_\varepsilon\|_{L^r(\widetilde \Omega_\varepsilon)^{3}}\leq C\varepsilon^{-{2\over r}(\gamma-1)+1} ,\  \|D_\varepsilon \tilde u_\varepsilon\|_{L^r(\widetilde \Omega_\varepsilon)^{3\times 3}}\leq C\varepsilon^{-{2\over r}(\gamma-1)} ,\  \|\mathbb{D}_\varepsilon[\tilde u_\varepsilon]\|_{L^r(\widetilde \Omega_\varepsilon)^{3\times 3}}\leq C\varepsilon^{-{2\over r}(\gamma-1)} \,,
\end{equation}
\item If $\gamma>1$,
\begin{equation}\label{estimates_u_tilde2greater1}
\|\tilde u_\varepsilon\|_{L^r(\widetilde \Omega_\varepsilon)^{3}}\leq C\varepsilon^{{-{\gamma-1\over r-1}}+1} ,\quad \|D_\varepsilon \tilde u_\varepsilon\|_{L^r(\widetilde \Omega_\varepsilon)^{3\times 3}}\leq C\varepsilon^{{-{\gamma-1\over r-1}}} ,\quad \|\mathbb{D}_\varepsilon[\tilde u_\varepsilon]\|_{L^r(\widetilde \Omega_\varepsilon)^{3\times 3}}\leq C\varepsilon^{{-{\gamma-1\over r-1}}} \,.
\end{equation}
\item If $\gamma=1$,
\begin{equation}\label{estimates_u_tilde2equal1}
\|\tilde u_\varepsilon\|_{L^r(\widetilde \Omega_\varepsilon)^{3}}\leq C\varepsilon ,\quad \|D_\varepsilon \tilde u_\varepsilon\|_{L^r(\widetilde \Omega_\varepsilon)^{3\times 3}}\leq C ,\quad \|\mathbb{D}_\varepsilon[\tilde u_\varepsilon]\|_{L^r(\widetilde \Omega_\varepsilon)^{3\times 3}}\leq C \,.
\end{equation}
\end{itemize}
\end{itemize}
\end{lemma}
\begin{proof}
 Multiplying (\ref{2}) by $\tilde u_\varepsilon$, integrating over $\widetilde\Omega_\varepsilon$, and taking into account that
 $$\int_{\widetilde\Omega_\varepsilon}( \tilde  u_\varepsilon\cdot \nabla_\varepsilon)\tilde u_\varepsilon \tilde u_\varepsilon dx'dy_3=0,$$
 we get
\begin{equation*}{\color{black} \mu\varepsilon^\gamma} (\eta_0-\eta_\infty)\int_{\widetilde\Omega_\varepsilon}\left(1+\lambda|\mathbb{D}_{\varepsilon}[\tilde u_\varepsilon]|^2\right)^{{r\over 2}-1}|\mathbb{D}_\varepsilon[\tilde u_\varepsilon]|^2dx'dy_3+{\color{black} \mu}\varepsilon^\gamma \eta_\infty\int_{\widetilde\Omega_\varepsilon}|\mathbb{D}_{\varepsilon}[\tilde u_\varepsilon]|^2dx'dy_3= \int_{\widetilde\Omega_\varepsilon}f'\cdot \tilde u_\varepsilon'\,dx'dy_3.
\end{equation*}
Then, we argue as in~\cite[Lemma 3.2]{Carreau_Ang_Bonn_SG2} to obtain the estimates \eqref{estimates_u_tilde}--\eqref{estimates_u_tilde2equal1}.
\end{proof}

\begin{remark}
We extend the velocity $\tilde u_\varepsilon$ by zero in $\Omega\setminus\widetilde \Omega_\varepsilon$, which is compatible with the homogeneous boundary condition on $\partial \Omega\cup \partial T_\varepsilon$, and denote the extension by the same symbol. Obviously, the estimates given in Lemma~\ref{Estimates_lemma} remain valid and the extension $\tilde u_\varepsilon$ is divergence free too.
\end{remark}

\subsection{A priori estimates for the pressure}

 In order to extend the pressure $\tilde p_\varepsilon$ to the whole domain $\Omega$ and obtain a priori estimates, we recall a result from~\cite{Anguiano_SuarezGrau} related to the existence of a restriction operator from $W^{1,q}_0(Q_\varepsilon)^3$ into $W^{1,q}_0(\Omega_\varepsilon)^3$, and its main properties, where $Q_\eps$ is the thin layer $Q_\eps=\omega\times (0,\eps)$. 

\begin{lemma}[Lemma 4.5-(i) in \cite{Anguiano_SuarezGrau}] \label{restriction_operator}
Let $1<q<+\infty$ be fixed. There exists  a restriction operator $R^\varepsilon_q$ acting from $W^{1,q}_0(Q_\varepsilon)^3$ into $W^{1,q}_0(\Omega_\varepsilon)^3$,  such that
\begin{enumerate}
\item $R^\varepsilon_q v=v$, if $v \in W^{1,q}_0(\Omega_\varepsilon)^3$ (elements of $W^{1,q}_0(\Omega_\varepsilon)$ are extended by $0$ to $Q_\varepsilon$).
\item ${\rm div}R^\varepsilon_q v=0\hbox{  in }\Omega_\varepsilon$, if ${\rm div}\,v=0\hbox{  in }Q_\varepsilon$.
\item There exists a positive constant $C$, such that for every $\eps$ and every $v\in W^{1,q}_0(Q_\varepsilon)^3$, 
\begin{equation}\label{estim_restricted}
\begin{array}{l}
\|R^\varepsilon_q v\|_{L^q(\Omega_\varepsilon)^{3}}+ \varepsilon\|D R^\varepsilon_q v\|_{L^q(\Omega_\varepsilon)^{3\times 3}} \leq C\left(\|v\|_{L^q(Q_\varepsilon)^3}+\varepsilon \|D v\|_{L^q(Q_\varepsilon)^{3\times 3}}\right)\,.
\end{array}
\end{equation}
\end{enumerate}
\end{lemma}

\subsubsection{Pseudoplastic or Newtonian fluids: case $1< r\leq 2$}
In the next result, by using the restriction operator defined in Lemma \ref{restriction_operator},  we first extend the gradient of the pressure by duality in $W^{-1,C(r)}(Q_\varepsilon)^3$, where $C(r)$ is given by (\ref{Cr_case_pseudo}), and then, by means of the dilatation, we extend $\tilde p_\varepsilon$ to $\Omega$, and finally, we derive estimates of the extension of the pressure.

We have the following results for the pressure  $\tilde p_\varepsilon$.
\begin{proposition}\label{Estimates_extended_lemma}  Consider $1< r\leq 2$. Let the constant $C(r)$ be defined by (\ref{Cr_case_pseudo}). Then, there exist an extension  $\tilde P_\varepsilon \in L^{C(r)}_0(\Omega)$ of the pressure $\tilde p_\varepsilon$ and a positive constant $C$ independent of $\varepsilon$, such that 
\begin{equation}\label{esti_P}
\|\tilde P_\varepsilon\|_{L^{C(r)}(\Omega)}\leq C\,, \quad \|\nabla_\varepsilon \tilde P_\varepsilon\|_{W^{-1,C(r)}(\Omega)^3}\leq C.
\end{equation}
\end{proposition}
 
\begin{proof}
 We divide the proof in two steps. First, we extend the pressure, and next we obtain the estimates in the case of pseudoplastic or Newtonian fluids.\\

{\it Step 1.} {\it Extension of the pressure}.  Using the restriction operator with $q\geq 2$, where $q$ to be determined in {\it Step 2}, given in  Lemma \ref{restriction_operator}, i.e., $R^\varepsilon_q$, we  introduce $F_\varepsilon$ in $W^{-1,C(r)}(Q_\varepsilon)^3$, where $C(r)$ is the conjugate of $q$,  in the following way
\begin{equation}\label{F}\langle F_\varepsilon, v\rangle_{W^{-1,C(r)}(Q_\varepsilon)^3, W^{1,q}_0(Q_\varepsilon)^3}=\langle \nabla p_\varepsilon, R^\varepsilon_q v\rangle_{{W^{-1,C(r)}(\Omega_\varepsilon)^3, W^{1,q}_0(\Omega_\varepsilon)^3}}\,,\quad \hbox{for any }v\in W^{1,q}_0(Q_\varepsilon)^3\,,
\end{equation}
and calculate the right hand side of (\ref{F}) by using the variational formulation of problem (\ref{1}), which  gives
\begin{equation}\label{equality_duality}
\begin{array}{rl}
\medskip
\displaystyle
\left\langle F_{\varepsilon},v\right\rangle_{W^{-1,C(r)}(Q_\varepsilon)^3, W^{1,q}_0(Q_\varepsilon)^3}=&\displaystyle
-{\color{black} \mu\varepsilon^\gamma}(\eta_0-\eta_\infty)\int_{\Omega_\varepsilon} (1+\lambda|\mathbb{D}[u_\varepsilon]|^2)^{{r\over 2}-1}\mathbb{D}[u_\varepsilon]: DR^{\varepsilon}_qv\,dx \\
\medskip
&\displaystyle
- {\color{black} \mu\varepsilon^\gamma}\eta_\infty  \int_{\Omega_\varepsilon}   \mathbb{D}[u_\varepsilon]: DR^{\varepsilon}_qv\,dx+ \int_{\Omega_\varepsilon} f'\cdot (R^{\varepsilon}_qv)'\,dx \,\\
\medskip
&\displaystyle 
-\int_{\Omega_\varepsilon} (u_\varepsilon\cdot \nabla)u_\varepsilon\,R^{\varepsilon}_qv\,dx.

\end{array}\end{equation}
Using Lemma \ref{Estimates_lemma} for fixed $\varepsilon$, we see that it is a bounded functional on $W^{1,q}_0(Q_\varepsilon)$ (see {\it Step} 2 below), and in fact $F_\varepsilon\in W^{-1,C(r)}(Q_\varepsilon)^3$. Moreover, ${\rm div}\, v=0$ implies $\left\langle F_{\varepsilon},v\right\rangle=0\,,$ and the DeRham theorem gives the existence of $P_\varepsilon$ in $L^{C(r)}_0(Q_\varepsilon)$ with $F_\varepsilon=\nabla P_\varepsilon$.

Next, we get for every $\tilde v\in W^{1,q}_0(\Omega)^3$ where $\tilde v(x', y_3)=v(x',\varepsilon y_3)$, using the change of variables (\ref{dilatacion}), that
$$\begin{array}{rl}\displaystyle\langle \nabla_{\varepsilon}\tilde P_\varepsilon, \tilde v\rangle_{W^{-1,C(r)}(\Omega)^3, W^{1,q}_0(\Omega)^3}&\displaystyle
=-\int_{\Omega}\tilde P_\varepsilon\,{\rm div}_{\varepsilon}\,\tilde v\,dx'dy_3
=-\varepsilon^{-1}\int_{Q_\varepsilon}P_\varepsilon\,{\rm div}\,v\,dx=\varepsilon^{-1}\langle \nabla P_\varepsilon, v\rangle_{W^{-1,C(r)}(Q_\varepsilon)^3, W^{1,q}_0(Q_\varepsilon)^3}\,.
\end{array}$$
Using the identification (\ref{equality_duality}) of $F_\varepsilon$, we have
$$\begin{array}{rl}\displaystyle\langle \nabla_{\varepsilon}\tilde P_\varepsilon, \tilde v\rangle_{W^{-1,C(r)}(\Omega)^3, W^{1,q}_0(\Omega)^3}\displaystyle
\medskip
 =&\displaystyle \varepsilon^{-1}\left(
-{\color{black} \mu\varepsilon^\gamma}(\eta_0-\eta_\infty) \int_{\Omega_\varepsilon} (1+\lambda|\mathbb{D}[u_\varepsilon]|^2)^{{r\over 2}-1}\mathbb{D}[u_\varepsilon]: DR^{\varepsilon}_qv\,dx\right. \\
\medskip
& \displaystyle
\quad\quad - {\color{black} \mu\varepsilon^\gamma}\eta_\infty  \int_{\Omega_\varepsilon}   \mathbb{D}[u_\varepsilon]: DR^{\varepsilon}_qv\,dx+ \int_{\Omega_\varepsilon} f'\cdot (R^{\varepsilon}_qv)'\,dx\,\\
\medskip
&\displaystyle\left.\quad\quad -\int_{\Omega_\varepsilon} (u_\varepsilon\cdot \nabla)u_\varepsilon\,R^{\varepsilon}_qv\,dx\right),
\end{array}$$
and applying the change of variables (\ref{dilatacion}), we obtain 
\begin{equation}\label{extension_1}
\begin{array}{rl}\medskip
\displaystyle\langle \nabla_{\varepsilon}\tilde P_\varepsilon, \tilde v\rangle_{W^{-1,C(r)}(\Omega)^3, W^{1,q}_0(\Omega)^3}
=&\displaystyle- {\color{black} \mu\varepsilon^\gamma}(\eta_0-\eta_\infty)\int_{\widetilde \Omega_\varepsilon} (1+\lambda|\mathbb{D}_\varepsilon[\tilde u_\varepsilon]|^2)^{{r\over 2}-1}\mathbb{D}_\varepsilon[\tilde u_\varepsilon] : D_{\varepsilon}\tilde R^{\varepsilon}_q\tilde v\,dx'dy_3\\
\medskip
&\displaystyle - {\color{black} \mu\varepsilon^\gamma}\eta_\infty  \int_{\widetilde\Omega_\varepsilon}   \mathbb{D}_\varepsilon[\tilde u_\varepsilon]: D_\varepsilon\tilde R^{\varepsilon}_q\tilde v\,dx'dy_3+\int_{\widetilde \Omega_\varepsilon} f'(x')\cdot (\tilde R^{\varepsilon}_q \tilde v)'\,dx'dy_3
\\
\medskip
&\displaystyle
-\int_{\widetilde \Omega_\varepsilon} (\tilde u_\varepsilon\cdot \nabla_\varepsilon )\tilde u_\varepsilon\,\tilde R^{\varepsilon}_q\tilde v\,dx'dy_3\,,
\end{array}
\end{equation}
where $\tilde R^\varepsilon_q\tilde v=R^\varepsilon_q v$ for any $\tilde v \in W^{1,q}_0(\Omega)^3$.\\

{\it Step 2}. {\it  Estimates of the extended pressure for pseudoplastic or Newtonian fluids}.  We consider an exponent $q\geq 2$ and its conjugate $C(r)\in (1,2]$. Taking into account that the a priori estimates for the velocity, given in (\ref{estimates_u_tilde}), is given in $L^2$, the value of $q$ has to be as near as possible to $2$ and it will be derived later when we estimate the inertial term.

The estimate  (\ref{estim_restricted}) implies that  $\tilde R^\varepsilon_q\tilde v$ satisfies
\begin{equation*}\begin{array}{l}
\medskip\displaystyle
\|\tilde R^\varepsilon_q\tilde v\|_{L^q(\widetilde\Omega_\varepsilon)^3}+ \varepsilon \|D_{\varepsilon}\tilde R^\varepsilon_q\tilde v\|_{L^q(\widetilde\Omega_\varepsilon)^{3\times 3}}\leq  C\left(\|\tilde v\|_{L^q(\Omega)^3} 
+ \varepsilon\|D_{\varepsilon}\tilde v\|_{L^q(\Omega)^{3\times 3}}\right),
\end{array} 
\end{equation*}
and since $\varepsilon\ll 1$, we have 
\begin{equation}\label{ext_2}
\|\tilde R^\varepsilon_q \tilde v\|_{L^q(\widetilde\Omega_\varepsilon)^3}\leq  C \|\tilde v\|_{W^{1,q}_0(\Omega)^3},\quad \|D_\varepsilon \tilde R^\varepsilon_q\tilde v\|_{L^q(\widetilde\Omega_\varepsilon)^{3\times 3}}\leq {C\over \varepsilon}\|\tilde v\|_{W^{1,q}_0(\Omega)^3}.
\end{equation}
From the Sobolev embedding  $L^2 \hookrightarrow L^{C(r)}$, using the last estimate in (\ref{estimates_u_tilde}) and the estimates of the dilated restricted operator given in (\ref{ext_2}), we get the estimates for the three first terms of the right-hand side of (\ref{extension_1})
\begin{equation}\label{estim_caso2_1}
\begin{array}{rl}
\medskip
\displaystyle
\left|{\color{black} \mu\varepsilon^\gamma}\int_{\widetilde \Omega_\varepsilon}(1+\lambda|\mathbb{D}_\varepsilon[\tilde u_\varepsilon]|^2)^{{r\over 2}-1}\mathbb{D}_\varepsilon[\tilde u_\varepsilon] : D_{\varepsilon}\tilde R^{\varepsilon}_q \tilde v\,dx'dy_3\right| \leq &\displaystyle {\color{black} \mu\varepsilon^\gamma}\int_{\widetilde\Omega_\varepsilon}(1+\lambda|\mathbb{D}_\varepsilon[\tilde u_\varepsilon]|^2)^{{r\over 2}-1}|\mathbb{D}_\varepsilon[\tilde u_\varepsilon] : D_{\varepsilon}\tilde R^{\varepsilon}_q \tilde v|\,dx'dy_3\\
\medskip
 \leq &\displaystyle
 {\color{black} \mu\varepsilon^\gamma}\int_{\widetilde \Omega_\varepsilon}|\mathbb{D}_\varepsilon[\tilde u_\varepsilon] : D_{\varepsilon}\tilde R^{\varepsilon}_q \tilde v|\,dx'dy_3
 \\
\medskip
 \leq &\displaystyle
\displaystyle{\color{black} \mu \varepsilon^\gamma}\|\mathbb{D}_\varepsilon[\tilde u_\varepsilon]\|_{L^{C(r)}(\widetilde\Omega_\varepsilon)^{3\times 3}}\|D_{\varepsilon}\tilde R^{\varepsilon}_q \tilde v\|_{L^{q}(\widetilde\Omega_\varepsilon)^{3\times 3}}\\
 \medskip \leq &\displaystyle C{\color{black} \mu\varepsilon^\gamma} \|\mathbb{D}_\varepsilon[\tilde u_\varepsilon]\|_{L^{2}(\widetilde\Omega_\varepsilon)^{3\times 3}}\|D_{\varepsilon}\tilde R^{\varepsilon}_q \tilde v\|_{L^{q}(\widetilde\Omega_\varepsilon)^{3\times 3}}\\
\medskip
 \leq &\displaystyle  C\|\tilde v\|_{W^{1,q}_0(\Omega)^3},
\end{array}
\end{equation}
\begin{equation}\label{estim_caso2_2}
\begin{array}{rl}
\medskip
\displaystyle
\left|{\color{black} \mu\varepsilon^\gamma}  \int_{\widetilde\Omega_\varepsilon}   \mathbb{D}_\varepsilon[\tilde u_\varepsilon]: D_\varepsilon\tilde R^{\varepsilon}_q\tilde v\,dx'dy_3\right|  \leq &\displaystyle C{\color{black} \mu\varepsilon^\gamma} \|\mathbb{D}_{\varepsilon}[\tilde u_\varepsilon]\|_{L^{C(r)}(\widetilde\Omega_\varepsilon)^{3\times 3}}\|D_{\varepsilon}\tilde R^\varepsilon_q\tilde v\|_{L^q(\widetilde \Omega_\varepsilon)^{3\times 3}}\\
\medskip
\leq &\displaystyle C{\color{black} \mu\varepsilon^\gamma} \|\mathbb{D}_{\varepsilon}[\tilde u_\varepsilon]\|_{L^{2}(\widetilde\Omega_\varepsilon)^{3\times 3}}\|D_{\varepsilon}\tilde R^\varepsilon_q\tilde v\|_{L^q(\widetilde \Omega_\varepsilon)^{3\times 3}}
\\
\medskip
 \leq &\displaystyle  C \|\tilde v\|_{W^{1,q}_0(\Omega)^3},
\end{array}
\end{equation}
\begin{equation}\label{estim_caso2_3}
\begin{array}{rl}
\medskip
\displaystyle
 \left|\int_{\widetilde\Omega_\varepsilon}f'\cdot (\tilde R^\varepsilon_q \tilde v)' \,dx'dy_3\right|&\leq C\|\tilde R^\varepsilon_q \tilde v \|_{L^q(\widetilde\Omega_\varepsilon)^3}\leq C\|\tilde v\|_{W^{1,q}_0(\Omega)^3}\,.
\end{array}
\end{equation}
Hence, we just need to estimate the inertial term, which can be written as 
\begin{eqnarray}\label{inertial_case_2}
\displaystyle \int_{\widetilde \Omega_\varepsilon} (\tilde u_\varepsilon\cdot \nabla_{\varepsilon}) \tilde u_\varepsilon\,\tilde R^\varepsilon_q \tilde v\, dx'dy_3=-\int_{\widetilde \Omega_\varepsilon} \tilde u_\varepsilon\tilde \otimes \tilde u_\varepsilon:D_{x'} \tilde R^\varepsilon_q \tilde v\,dx'dy_3\\
\displaystyle +{1\over \varepsilon}\left(\int_{\widetilde \Omega_\varepsilon}\partial_{y_3} \tilde u_{\varepsilon,3}\tilde  u_\varepsilon \tilde R^\varepsilon_q \tilde v \,dx'dy_3+ \int_{\widetilde \Omega_\varepsilon}\tilde u_{\varepsilon,3}\partial_{y_3} \tilde u_\varepsilon\, \tilde R^\varepsilon_q \tilde v\,dx'dy_3\right)\,,\nonumber
\end{eqnarray}
where $(u \tilde \otimes w)_{ij}=u_i w_j$, $i=1,2$, $j=1,2,3$.  In the following, we will show the estimate  with the values of $C(r)$ according to $\gamma$ given in (\ref{Cr_case_pseudo}), where $C(r)$ is the conjugate of $q$. For that, we will consider each of the three terms of the sum in the second member of (\ref{inertial_case_2}) in order to get an estimate as follows:
\begin{equation}\label{inertial_alpha}\left|\int_{\widetilde\Omega_\varepsilon}(\tilde u_\varepsilon\cdot \nabla_\varepsilon)\tilde u_\varepsilon \tilde R^\varepsilon_q \tilde v\,dx'dy_3\right|\leq C\varepsilon^{\alpha}\|\tilde v\|_{W^{1,q}_0(\Omega)^3},\quad \hbox{with }\alpha\geq0.
\end{equation}
\begin{itemize}
\item[{\bf a)}] {\bf Estimate for the first term in the second member of (\ref{inertial_case_2})}: by using Hölder inequality and  the second estimate of the restriction operator (\ref{ext_2}), we have
\begin{equation}\label{interpola_1_caso2}
\begin{array}{l}
\displaystyle \left|\int_{\widetilde \Omega_\varepsilon} \tilde u_\varepsilon\tilde \otimes \tilde u_\varepsilon:D_{x'} \tilde R^\varepsilon_q \tilde v\,dx'dy_3\right| 
\leq \|\tilde u_\varepsilon\|_{L^{q'}(\widetilde\Omega_\varepsilon)^3}^2\|D_{x'}\tilde R^\varepsilon_q \tilde v\|_{L^q(\widetilde\Omega_\varepsilon)^{3\times 2}}\leq C\varepsilon^{-1}  \|\tilde u_\varepsilon\|_{L^{q'}(\widetilde\Omega_\varepsilon)^3}^2\|\tilde v\|_{W^{1,q}_0(\Omega)^3},
\end{array}
\end{equation}
where 
\begin{equation}\label{ineq1}
{2\over q'}+{1\over q}\leq 1.
\end{equation}
 It remains to estimate $ \|\tilde u_\varepsilon\|_{L^{q'}(\widetilde\Omega_\varepsilon)^3}$ using (\ref{estimates_u_tilde}). So, our goal is to have $q$ the closer possible to $2$. Therefore, by interpolating between $L^2(\widetilde \Omega_\varepsilon)$ and $W_0^{1,2}(\widetilde \Omega_\varepsilon)$ with the parameter $\theta\in [0,1]$ such that 
\begin{equation}\label{inter2_6}{1\over q'}={\theta\over 2}+{1\over 6}(1-\theta),
\end{equation}
we have
\begin{equation*} \|\tilde u_\varepsilon\|_{L^{q'}(\widetilde\Omega_\varepsilon)^3}\leq \|\tilde u_\varepsilon\|^{\theta}_{L^2(\widetilde\Omega_\varepsilon)^3}\|\tilde u_\varepsilon\|^{1-\theta}_{L^{6}(\widetilde\Omega_\varepsilon)^3},
\end{equation*}
and from the Sobolev embedding $W_0^{1,2}(\widetilde \Omega_\varepsilon) \hookrightarrow L^{6}(\widetilde\Omega_\varepsilon)$ and  estimates (\ref{estimates_u_tilde}), we obtain
\begin{equation}\label{interpola_2_caso2} \|\tilde u_\varepsilon\|_{L^{q'}(\widetilde\Omega_\varepsilon)^3}\leq \|\tilde u_\varepsilon\|^{\theta}_{L^2(\widetilde\Omega_\varepsilon)^3}\| D \tilde u_\varepsilon\|^{1-\theta}_{L^2(\widetilde\Omega_\varepsilon)^{3\times 3}}\leq \|\tilde u_\varepsilon\|^{\theta}_{L^2(\widetilde\Omega_\varepsilon)^3}\| D_\varepsilon \tilde u_\varepsilon\|^{1-\theta}_{L^2(\widetilde\Omega_\varepsilon)^{3\times 3}}
\leq C\varepsilon^{\theta(2-\gamma)+(1-\theta)(1-\gamma)}.
\end{equation}
Taking into account (\ref{interpola_2_caso2}) in (\ref{interpola_1_caso2}), we can deduce 
\begin{equation*}
\begin{array}{l}
\displaystyle \left|\int_{\widetilde \Omega_\varepsilon} \tilde u_\varepsilon\tilde \otimes \tilde u_\varepsilon:D_{x'} \tilde R^\varepsilon_q \tilde v\,dx'dy_3\right| 
\leq  C\varepsilon^{2\left(\theta-\gamma+{1\over 2}\right)} \|\tilde v\|_{W_0^{1,q}( \Omega )^{3}}.
\end{array}
\end{equation*}
Thus, we must choose $\theta$ such that $\theta-\gamma+1/2\geq 0$, that is
\begin{equation}\label{intertial_case1_5}
\theta\geq\theta_0= \max\{0,\gamma-1/2\}.
\end{equation}
The inequality (\ref{ineq1}) and the equation (\ref{inter2_6}) imply that 
\begin{equation}\label{q_bounded}
q\ge {3\over 2(1-\theta)}.
\end{equation}
Taking into account (\ref{intertial_case1_5}) in (\ref{q_bounded}), we obtain 
\begin{equation}\label{intertial_case1_6}
q\ge {3\over 3-2\gamma}.
\end{equation}
If we consider $q=2$ in (\ref{intertial_case1_6}), we obtain that $\gamma\leq 3/4$. For $\gamma>3/4$ we consider (\ref{intertial_case1_6}) taking into account that $\gamma-1/2\leq \theta \leq 1$, namely we consider (\ref{intertial_case1_6}) for $3/4<\gamma\leq 3/2$. And then we obtain
\begin{equation}\label{case_1_q_term1} \left\{\begin{array}{lcl}
q=2& \hbox{for}& \gamma \leq \displaystyle{3\over 4},\\
\\
\displaystyle q\geq{3\over 3-2\gamma}& \hbox{for}&\displaystyle {3\over 4}<\gamma \leq {3\over 2}.
\end{array}\right.
\end{equation}

\item[{\bf b)}] {\bf Estimate for the second and third term in the second member of (\ref{inertial_case_2})}: by using Hölder inequality, the second and third term are estimated by 
\begin{eqnarray}\label{estimate_(b)}
&&\left|\varepsilon^{-1}\left(\int_{\widetilde \Omega_\varepsilon}\partial_{y_3} \tilde u_{\varepsilon,3}\tilde  u_\varepsilon \tilde R^\varepsilon_q \tilde v \,dx'dy_3+ \int_{\widetilde \Omega_\varepsilon}\tilde u_{\varepsilon,3}\partial_{y_3} \tilde u_\varepsilon\, \tilde R^\varepsilon_q \tilde v\,dx'dy_3\right)\right|\\
&&\leq C \varepsilon^{-1}\|\partial_{y_3}\tilde u_\varepsilon\|_{L^2(\widetilde\Omega_\varepsilon)^3}\|\tilde u_\varepsilon\|_{L^{q'}(\widetilde\Omega_\varepsilon)^3}\| R^\varepsilon_q \tilde v\|_{L^{q_1}(\widetilde\Omega_\varepsilon)^3},\nonumber
\end{eqnarray}
with
\begin{equation}\label{ineq2}
{1\over 2}+{1\over q'}+{1\over q_1}\leq 1.
\end{equation}
Interpolating between $L^2(\widetilde \Omega_\varepsilon)$ and $W_0^{1,2}(\widetilde \Omega_\varepsilon)$ with the parameter $\theta\in [0,1]$, we have again (\ref{inter2_6}) and (\ref{interpola_2_caso2}), and taking into account the second estimate in (\ref{estimates_u_tilde}), then (\ref{estimate_(b)}) can be written by
\begin{eqnarray}\label{estimate_(b)2}
\left|\varepsilon^{-1}\left(\int_{\widetilde \Omega_\varepsilon}\partial_{y_3} \tilde u_{\varepsilon,3}\tilde  u_\varepsilon \tilde R^\varepsilon_q \tilde v \,dx'dy_3+ \int_{\widetilde \Omega_\varepsilon}\tilde u_{\varepsilon,3}\partial_{y_3} \tilde u_\varepsilon\, \tilde R^\varepsilon_q \tilde v\,dx'dy_3\right)\right|\leq C \varepsilon^{\theta-2\gamma+2}\| R^\varepsilon_q \tilde v\|_{L^{q_1}(\widetilde\Omega_\varepsilon)^3}.
\end{eqnarray}
By Sobolev-Gagliardo-Nirenberg theorem, if $1\leq q<3$, then we have the continuity of embedding $W^{1,q} \hookrightarrow L^{q^*}$ where $1/q^*=1/q-1/3$. As our goal is to have $\tilde v\in W^{1,q}(\Omega)^3$ with $q\ge 2$ as close as possible to $2$, we choose $q_1={3q\over 3-q}$ which satisfies (\ref{ineq2}). Then, the Sobolev-Gagliardo-Nirenberg theorem and the second estimate of the restriction operator (\ref{ext_2}) imply that
$$\| R^\varepsilon_q \tilde v\|_{L^{q_1}(\widetilde\Omega_\varepsilon)^3}\leq C \|D R^\varepsilon_q \tilde v\|_{L^{q}(\widetilde\Omega_\varepsilon)^{3\times 3}}\leq C \|D_\varepsilon R^\varepsilon_q \tilde v\|_{L^{q}(\widetilde\Omega_\varepsilon)^{3\times 3}}\leq C\varepsilon^{-1}\|\tilde v\|_{W^{1,q}_0(\Omega)^3},$$
and (\ref{estimate_(b)2}) can be written by
\begin{eqnarray*}
\left|\varepsilon^{-1}\left(\int_{\widetilde \Omega_\varepsilon}\partial_{y_3} \tilde u_{\varepsilon,3}\tilde  u_\varepsilon \tilde R^\varepsilon_q \tilde v \,dx'dy_3+ \int_{\widetilde \Omega_\varepsilon}\tilde u_{\varepsilon,3}\partial_{y_3} \tilde u_\varepsilon\, \tilde R^\varepsilon_q \tilde v\,dx'dy_3\right)\right|\leq C \varepsilon^{\theta-2\gamma+1}\|\tilde v\|_{W^{1,q}_0(\Omega)^3}.
\end{eqnarray*}
Thus, we must choose $\theta$ such that $\theta-2\gamma+1\geq 0$, that is
\begin{equation*}
\theta\geq \max\{0,2\gamma-1\}.
\end{equation*}
The inequality (\ref{ineq2}) with $q_1={3q\over 3-q}$, and the equation (\ref{inter2_6}) imply that 
\begin{equation*}
\theta\leq \min\{1,2-{3\over q}\}.
\end{equation*}
 Then,  since $2\gamma-1\leq \theta\leq 2-{3\over q}$, then  we get (\ref{intertial_case1_6}). If we consider $q=2$ in (\ref{intertial_case1_6}), we obtain that $\gamma\leq 3/4$. For $\gamma>3/4$ we consider (\ref{intertial_case1_6}) taking into account that $2\gamma-1\leq \theta \leq 1$, namely we consider (\ref{intertial_case1_6}) for $3/4<\gamma\leq 1$. And then we obtain
\begin{equation}\label{case_1_q_term2} \left\{\begin{array}{lcl}
q=2& \hbox{for}& \gamma \leq \displaystyle{3\over 4},\\
\\
\displaystyle q\geq{3\over 3-2\gamma}& \hbox{for}&\displaystyle {3\over 4}<\gamma \leq 1.
\end{array}\right.
\end{equation}
\end{itemize}
As condition (\ref{case_1_q_term2}) is weaker than (\ref{case_1_q_term1}), it is the definition (\ref{Cr_case_pseudo}) which is the good one for $C(r)$, where $C(r)$ is the conjugate of $q$. Therefore, for the choice of $q$ given by (\ref{case_1_q_term2}), the estimate of the intertial term satisfies (\ref{inertial_alpha}), as we wanted.\\

Finally, considering $C(r)$ the conjugate of $q$, defined by (\ref{Cr_case_pseudo}), and taking into account (\ref{estim_caso2_1})-(\ref{estim_caso2_3}) and (\ref{inertial_alpha}) in (\ref{extension_1}), we get the second estimate in (\ref{esti_P}) and finally, using the Ne${\check{\rm c}}$as inequality, there exists a representative $\tilde P_\varepsilon\in L^{C(r)}_0(\Omega)$  such that
$$\|\tilde P_\varepsilon\|_{L^{C(r)}(\Omega)}\leq C\|\nabla\tilde P_\varepsilon\|_{W^{-1,C(r)}(\Omega)^3}\leq C\|\nabla_{\varepsilon}\tilde P_\varepsilon\|_{W^{-1,C(r)}(\Omega)^3},$$
which implies the first estimate in (\ref{esti_P}).

\end{proof}

\subsubsection{Dilatant fluids: case $r> 2$}
In the next result, by using the restriction operator defined in Lemma \ref{restriction_operator},  we first extend the gradient of the pressure by duality in $W^{-1,C(r)}(Q_\varepsilon)^3$, where $C(r)$ is given by (\ref{Cr_case_dilatant}), and then, by means of the dilatation, we extend $\tilde p_\varepsilon$ to $\Omega$, and finally, we derive estimates of the extension of the pressure.

We have the following results for the pressure  $\tilde p_\varepsilon$.
\begin{proposition}\label{Estimates_extended_lemma2}  Consider $r>2$ and $\gamma\leq 1$. Let the constant $C(r)$ be defined by (\ref{Cr_case_dilatant}). Then, there exist an extension  $\tilde P_\varepsilon \in L^{C(r)}_0(\Omega)$ of the pressure $\tilde p_\varepsilon$ and a positive constant $C$ independent of $\varepsilon$, such that 
\begin{equation}\label{esti_P2}
\|\tilde P_\varepsilon\|_{L^{C(r)}(\Omega)}\leq C\,, \quad \|\nabla_\varepsilon \tilde P_\varepsilon\|_{W^{-1,C(r)}(\Omega)^3}\leq C.
\end{equation}
\end{proposition}

\begin{proof}
 We divide the proof in two steps. First, we extend the pressure, and next we obtain the estimates in the case of dilatant fluids.\\

{\it Step 1.} {\it Extension of the pressure}. Using the restriction operator with $q\geq r$, where $q$ to be determined in {\it Step 2}, given in  Lemma \ref{restriction_operator}, i.e., $R^\varepsilon_q$, we  introduce $F_\varepsilon$ in $W^{-1,C(r)}(Q_\varepsilon)^3$, where $C(r)$ is the conjugate of $q$. Arguing as step 1 in Proposition \ref{Estimates_extended_lemma}, we have (\ref{extension_1}).

{\it Step 2}. {\it  Estimates of the extended pressure for dilatant fluids}. We consider an exponent $q\geq r$ and its conjugate $C(r)\in (1,{r\over r-1}]$. Taking into account that the a priori estimates for the velocity, given in (\ref{estimates_u_tilde2less1})-(\ref{estimates_u_tilde2equal1}), is given in $L^r$, the value of $q$ has to be as near as possible to $r$ and it will be derived later when we estimate the intertial term.

Taking into account that $r>2$, we have that $C(r)\leq {r\over r-1}<2$. Then, from the Sobolev embedding  $L^2\hookrightarrow L^{r\over r-1} \hookrightarrow L^{C(r)}$, using the last estimate in (\ref{estimates_u_tilde}) and the estimates of the dilated restricted operator given in (\ref{ext_2}), we get the estimates for the three first terms of the right-hand side of (\ref{extension_1})
\begin{equation}\label{estim_caso2_1_dilatant}
\begin{array}{rl}
\medskip
\displaystyle
\left|{\color{black} \mu\varepsilon^\gamma}\int_{\widetilde \Omega_\varepsilon}(1+\lambda|\mathbb{D}_\varepsilon[\tilde u_\varepsilon]|^2)^{{r\over 2}-1}\mathbb{D}_\varepsilon[\tilde u_\varepsilon] : D_{\varepsilon}\tilde R^{\varepsilon}_q \tilde v\,dx'dy_3\right| \leq &\displaystyle {\color{black} \mu\varepsilon^\gamma}\int_{\widetilde\Omega_\varepsilon}(1+\lambda|\mathbb{D}_\varepsilon[\tilde u_\varepsilon]|^2)^{{r\over 2}-1}|\mathbb{D}_\varepsilon[\tilde u_\varepsilon] : D_{\varepsilon}\tilde R^{\varepsilon}_q \tilde v|\,dx'dy_3\\
\medskip
 \leq &\displaystyle
{\color{black} \mu \varepsilon^\gamma}\int_{\widetilde \Omega_\varepsilon}|\mathbb{D}_\varepsilon[\tilde u_\varepsilon] : D_{\varepsilon}\tilde R^{\varepsilon}_q \tilde v|\,dx'dy_3
 \\
\medskip
 \leq &\displaystyle
\displaystyle {\color{black} \mu\varepsilon^\gamma}\|\mathbb{D}_\varepsilon[\tilde u_\varepsilon]\|_{L^{C(r)}(\widetilde\Omega_\varepsilon)^{3\times 3}}\|D_{\varepsilon}\tilde R^{\varepsilon}_q \tilde v\|_{L^{q}(\widetilde\Omega_\varepsilon)^{3\times 3}}\\
 \medskip \leq &\displaystyle C{\color{black} \mu\varepsilon^\gamma} \|\mathbb{D}_\varepsilon[\tilde u_\varepsilon]\|_{L^{2}(\widetilde\Omega_\varepsilon)^{3\times 3}}\|D_{\varepsilon}\tilde R^{\varepsilon}_q \tilde v\|_{L^{q}(\widetilde\Omega_\varepsilon)^{3\times 3}}\\
\medskip
 \leq &\displaystyle  C\|\tilde v\|_{W^{1,q}_0(\Omega)^3},
\end{array}
\end{equation}
\begin{equation}\label{estim_caso2_2_dilatant}
\begin{array}{rl}
\medskip
\displaystyle
\left|{\color{black} \mu\varepsilon^\gamma}  \int_{\widetilde\Omega_\varepsilon}   \mathbb{D}_\varepsilon[\tilde u_\varepsilon]: D_\varepsilon\tilde R^{\varepsilon}_q\tilde v\,dx'dy_3\right|  \leq &\displaystyle C{\color{black} \mu\varepsilon^\gamma} \|\mathbb{D}_{\varepsilon}[\tilde u_\varepsilon]\|_{L^{C(r)}(\widetilde\Omega_\varepsilon)^{3\times 3}}\|D_{\varepsilon}\tilde R^\varepsilon_q\tilde v\|_{L^q(\widetilde \Omega_\varepsilon)^{3\times 3}}\\
\medskip
\leq &\displaystyle C{\color{black} \mu\varepsilon^\gamma} \|\mathbb{D}_{\varepsilon}[\tilde u_\varepsilon]\|_{L^{2}(\widetilde\Omega_\varepsilon)^{3\times 3}}\|D_{\varepsilon}\tilde R^\varepsilon_q\tilde v\|_{L^q(\widetilde \Omega_\varepsilon)^{3\times 3}}
\\
\medskip
 \leq &\displaystyle  C \|\tilde v\|_{W^{1,q}_0(\Omega)^3},
\end{array}
\end{equation}
\begin{equation}\label{estim_caso2_3_dilatant}
\begin{array}{rl}
\medskip
\displaystyle
 \left|\int_{\widetilde\Omega_\varepsilon}f'\cdot (\tilde R^\varepsilon_q \tilde v)' \,dx'dy_3\right|&\leq C\|\tilde R^\varepsilon_q \tilde v \|_{L^q(\widetilde\Omega_\varepsilon)^3}\leq C\|\tilde v\|_{W^{1,q}_0(\Omega)^3}\,.
\end{array}
\end{equation}
Finally, we estimate the inertial term, which can be written as (\ref{inertial_case_2}) with the values of $C(r)$ according to $\gamma$ given in (\ref{Cr_case_dilatant}), where $C(r)$ is the conjugate of $q$. For that, we will consider each of the three terms of the sum in the second member of (\ref{inertial_case_2}) in order to get an estimate as follows:
\begin{equation}\label{inertial_alpha_dilatant}\left|\int_{\widetilde\Omega_\varepsilon}(\tilde u_\varepsilon\cdot \nabla_\varepsilon)\tilde u_\varepsilon \tilde R^\varepsilon_q \tilde v\,dx'dy_3\right|\leq C\varepsilon^{\alpha}\|\tilde v\|_{W^{1,q}_0(\Omega)^3},\quad \hbox{with }\alpha\geq0.
\end{equation}

\begin{itemize}
\item[{\bf a)}] {\bf Estimate for the first term in the second member of (\ref{inertial_case_2})}: considering (\ref{interpola_1_caso2}) and (\ref{ineq1}), it remains to estimate $ \|\tilde u_\varepsilon\|_{L^{q'}(\widetilde\Omega_\varepsilon)^3}$ using (\ref{estimates_u_tilde2less1})-(\ref{estimates_u_tilde2equal1}). So, our goal is to have $q$ the closer possible to $r$. Therefore, by interpolating between $L^r(\widetilde \Omega_\varepsilon)$ and $W_0^{1,r}(\widetilde \Omega_\varepsilon)$ with the parameter $\theta\in [0,1]$ such that 
\begin{equation}\label{inter2_q1}{1\over q'}={\theta\over r}+{1\over q_2}(1-\theta),
\end{equation}
we have
\begin{equation} \label{interpolation_r}
\|\tilde u_\varepsilon\|_{L^{q'}(\widetilde\Omega_\varepsilon)^3}\leq \|\tilde u_\varepsilon\|^{\theta}_{L^r(\widetilde\Omega_\varepsilon)^3}\|\tilde u_\varepsilon\|^{1-\theta}_{L^{q_2}(\widetilde\Omega_\varepsilon)^3}.
\end{equation}
By Sobolev-Gagliardo-Nirenberg theorem, if $1\leq r<3$, then we have the continuity of embedding $W^{1,r} \hookrightarrow L^{q^*}$ where $1/q^*=1/r-1/3$. Moreover, if $r=3$,  then we have the continuity of embedding $W^{1,r} \hookrightarrow L^{q^*}$ where $q^*\in[3,+\infty)$ and by Morrey theorem if $r>3$, then we have the continuity of embedding $W^{1,r} \hookrightarrow L^{\infty}$. Then, we choose
\begin{equation}\label{def_q2_dilatant}
q_2= \left\{\begin{array}{rcl}
\displaystyle {3r\over 3-r} & \hbox{for}& 2<r<3,\\
\\
\hbox{any positive value}& \hbox{for}& r=3,\\
\\
\infty& \hbox{for}& r>3.
 \end{array}\right.
\end{equation}
Then, by (\ref{interpolation_r}), the Sobolev-Gagliardo-Nirenberg theorem and Morray theorem, we obtain
\begin{equation*} \|\tilde u_\varepsilon\|_{L^{q'}(\widetilde\Omega_\varepsilon)^3}\leq \|\tilde u_\varepsilon\|^{\theta}_{L^r(\widetilde\Omega_\varepsilon)^3}\| D \tilde u_\varepsilon\|^{1-\theta}_{L^r(\widetilde\Omega_\varepsilon)^{3\times 3}}\leq \|\tilde u_\varepsilon\|^{\theta}_{L^r(\widetilde\Omega_\varepsilon)^3}\| D_\varepsilon \tilde u_\varepsilon\|^{1-\theta}_{L^r(\widetilde\Omega_\varepsilon)^{3\times 3}},
\end{equation*}
and by estimates (\ref{estimates_u_tilde2less1})-(\ref{estimates_u_tilde2equal1}), we have
\begin{equation}\label{interpola_2_caso2_dilatant}
\|\tilde u_\varepsilon\|_{L^{q'}(\widetilde\Omega_\varepsilon)^3}\leq \left\{\begin{array}{rcl}
C\varepsilon^{\theta-{2\over r}(\gamma-1)} & \hbox{if}& \gamma<1,\\
\\
C\varepsilon^{\theta-{\gamma-1\over r-1}} & \hbox{if}& \gamma>1,\\
\\
C\varepsilon^{\theta} & \hbox{if}& \gamma=1.
 \end{array}\right.
\end{equation}
Taking into account (\ref{interpola_2_caso2_dilatant}) in (\ref{interpola_1_caso2}), we can deduce 
\begin{equation}\label{interpola_3_caso2_dilatant}
\displaystyle \left|\int_{\widetilde \Omega_\varepsilon} \tilde u_\varepsilon\tilde \otimes \tilde u_\varepsilon:D_{x'} \tilde R^\varepsilon_q \tilde v\,dx'dy_3\right| \leq \left\{\begin{array}{rcl}
C\varepsilon^{2(\theta-{2\over r}(\gamma-1)-{1\over 2})}\|\tilde v\|_{W_0^{1,q}( \Omega )^{3}} & \hbox{if}& \gamma<1,\\
\\
C\varepsilon^{2(\theta-{\gamma-1\over r-1}-{1\over 2})}\|\tilde v\|_{W_0^{1,q}( \Omega )^{3}} & \hbox{if}& \gamma>1,\\
\\
C\varepsilon^{2(\theta-{1\over 2})}\|\tilde v\|_{W_0^{1,q}( \Omega )^{3}} & \hbox{if}& \gamma=1.
 \end{array}\right.
\end{equation}
To have a positive power of $\varepsilon$, we must choose $\theta\in [0,1]$ such that
\begin{equation}\label{theta_dilatant}
\theta\geq \left\{\begin{array}{rcl}
\displaystyle{1\over 2}& \hbox{for}& \gamma\leq 1,\\
\\
\displaystyle  {\gamma-1\over r-1}+{1\over 2}& \hbox{for}& 1<\gamma\leq \displaystyle{r+1\over 2}.
 \end{array}\right.
\end{equation}
We have $\gamma\leq \displaystyle{r+1\over 2}$ because $\theta$ has to be $\leq 1$.
\begin{itemize}
\item[{\bf $a_1)$}] {\bf For $r\ge 3$}: we can choose ${1\over q_2}>0$ as close as zero as we like and $\theta=1$. Then, from (\ref{inter2_q1}) we have ${1\over q'}<{1\over r}$ for $\gamma \leq {r+1\over 2}$, and if we consider (\ref{interpola_1_caso2}) with $q=r$, we have
\begin{equation*}
\begin{array}{l}
\displaystyle \left|\int_{\widetilde \Omega_\varepsilon} \tilde u_\varepsilon\tilde \otimes \tilde u_\varepsilon:D_{x'} \tilde R^\varepsilon_r \tilde v\,dx'dy_3\right| 
\leq \|\tilde u_\varepsilon\|_{L^{r}(\widetilde\Omega_\varepsilon)^3}^2\|D_{x'}\tilde R^\varepsilon_r \tilde v\|_{L^r(\widetilde\Omega_\varepsilon)^{3\times 2}}\leq C\varepsilon^{-1}  \|\tilde u_\varepsilon\|_{L^{r}(\widetilde\Omega_\varepsilon)^3}^2\|\tilde v\|_{W^{1,r}_0(\Omega)^3},
\end{array}
\end{equation*}
and using (\ref{estimates_u_tilde2less1})-(\ref{estimates_u_tilde2equal1}), we have
\begin{equation*}
\displaystyle \left|\int_{\widetilde \Omega_\varepsilon} \tilde u_\varepsilon\tilde \otimes \tilde u_\varepsilon:D_{x'} \tilde R^\varepsilon_r \tilde v\,dx'dy_3\right| \leq \left\{\begin{array}{rcl}
C\varepsilon^{1-{4\over r}(\gamma-1)}\|\tilde v\|_{W_0^{1,r}( \Omega )^{3}} & \hbox{if}& \gamma<1,\\
\\
C\varepsilon\|\tilde v\|_{W_0^{1,r}( \Omega )^{3}} & \hbox{if}& \gamma=1,\\
\\
C\varepsilon^{1-2{\gamma-1\over r-1}}\|\tilde v\|_{W_0^{1,r}( \Omega )^{3}} & \hbox{if}& 1<\gamma\leq \displaystyle{r+1\over 2}.
\end{array}\right.
\end{equation*}
Observe that if $\gamma<1$ then $1-{4\over r}(\gamma-1)>0$ and if $1<\gamma\leq \displaystyle{r+1\over 2}$, then $1-2\displaystyle{\gamma-1\over r-1}\ge0$. So we can choose $q=r$.
\item[{\bf $a_2)$}] {\bf For $2<r<3$}: in this case, we have $q_2={3r\over 3-r}$ and from (\ref{inter2_q1}), we obtain
$${1\over q'}={\theta \over 3}+{3-r\over 3r}.$$
For $\gamma\leq 1$, we can choose $\theta=\displaystyle {1\over 2}$ and condition (\ref{ineq1}) is satisfied for $q=r$ if we consider $\displaystyle {9\over 4}\leq r<3$. For $2<r<\displaystyle {9\over 4}$, by (\ref{theta_dilatant}) we have that $\theta\ge \displaystyle{1\over 2}$, and by condition (\ref{ineq1}) we have
$$q\ge \displaystyle {3r\over 4r-6}.$$\\
\\
For $1<\gamma\leq \displaystyle{r+1\over 2}$, by (\ref{theta_dilatant}) we have that $\theta\ge \displaystyle{\gamma-1\over r-1}+{1\over 2}$, and condition (\ref{ineq1}) is satisfied for $q=r$ only if
$$\gamma\leq -{11\over 2}+2r+{9\over 2r}.$$\\
\\
For $\displaystyle-{11\over 2}+2r+{9\over 2r}<\gamma\leq \displaystyle {r+1\over 2}$, by (\ref{theta_dilatant}) we have that $\theta\ge \displaystyle{\gamma-1\over r-1}+{1\over 2}$, and by condition (\ref{ineq1}) we have
$$q\ge {3r(r-1)\over 2(r-1)(2r-3)-2r(\gamma-1)}.$$
\end{itemize}
Finally, to have a positive power of $\varepsilon$ in estimate of the first term in the second member of of (\ref{inertial_case_2}), we have to choose
\begin{equation}\label{q_dilantant_term1}
q=\left\{\begin{array}{ccl}
r&\hbox{for}& r\ge 3 \hbox{ and } \gamma\leq \displaystyle {r+1\over 2},
\\
\\
\displaystyle {3r\over 4r-6}&\hbox{for}&  2<r<\displaystyle{9\over 4}\hbox{ and }\gamma\leq 1,\\
\\
r&\hbox{for}&  \displaystyle {9\over 4}\leq r<3\hbox{ and }\gamma\leq 1,\\
\\
r&\hbox{for}& 2<r<3\hbox{ and } 1<\gamma\leq \displaystyle -{11\over 2}+2r+{9\over 2r},\\
\\
\displaystyle {3r(r-1)\over 4r^2-8r+6-2\gamma r}&\hbox{for}& 2<r<3\hbox{ and }\displaystyle-{11\over 2}+2r+{9\over 2r}<\gamma \leq {r+1\over 2}.
\end{array}\right.
\end{equation}

\item[{\bf b)}] {\bf Estimate for the second and third term in the second member of (\ref{inertial_case_2})}: by using Hölder inequality, the second and third term are estimated by 
\begin{eqnarray}\label{estimate_(b)_dilatant}
&&\left|\varepsilon^{-1}\left(\int_{\widetilde \Omega_\varepsilon}\partial_{y_3} \tilde u_{\varepsilon,3}\tilde  u_\varepsilon \tilde R^\varepsilon_q \tilde v \,dx'dy_3+ \int_{\widetilde \Omega_\varepsilon}\tilde u_{\varepsilon,3}\partial_{y_3} \tilde u_\varepsilon\, \tilde R^\varepsilon_q \tilde v\,dx'dy_3\right)\right|\\
&&\leq C \varepsilon^{-1}\|\partial_{y_3}\tilde u_\varepsilon\|_{L^r(\widetilde\Omega_\varepsilon)^3}\|\tilde u_\varepsilon\|_{L^{q'}(\widetilde\Omega_\varepsilon)^3}\| R^\varepsilon_q \tilde v\|_{L^{q_1}(\widetilde\Omega_\varepsilon)^3},\nonumber
\end{eqnarray}
with
\begin{equation}\label{ineq2_dilatant}
{1\over r}+{1\over q'}+{1\over q_1}\leq 1.
\end{equation}
By Sobolev-Gagliardo-Nirenberg theorem, if $1\leq q<3$, then we have the continuity of embedding $W^{1,q} \hookrightarrow L^{q^*}$ where $1/q^*=1/q-1/3$. Moreover, if $q=3$,  then we have the continuity of embedding $W^{1,q} \hookrightarrow L^{q^*}$ where $q^*\in[3,+\infty)$ and by Morrey theorem if $q>3$, then we have the continuity of embedding $W^{1,q} \hookrightarrow L^{\infty}$. Then, we choose
\begin{equation}\label{def_q1_dilatant_term2}
q_1= \left\{\begin{array}{rcl}
\displaystyle {3q\over 3-q} & \hbox{for}& 2<q<3,\\
\\
\hbox{any positive value}& \hbox{for}& q=3,\\
\\
\infty& \hbox{for}& q>3.
 \end{array}\right.
\end{equation}
Then, the Sobolev-Gagliardo-Nirenberg theorem, the Morrey theorem and the second estimate of the restriction operator (\ref{ext_2}) imply that
$$\| R^\varepsilon_q \tilde v\|_{L^{q_1}(\widetilde\Omega_\varepsilon)^3}\leq C \|D R^\varepsilon_q \tilde v\|_{L^{q}(\widetilde\Omega_\varepsilon)^{3\times 3}}\leq C \|D_\varepsilon R^\varepsilon_q \tilde v\|_{L^{q}(\widetilde\Omega_\varepsilon)^{3\times 3}}\leq C\varepsilon^{-1}\|\tilde v\|_{W^{1,q}_0(\Omega)^3},$$
and (\ref{estimate_(b)_dilatant}) can be written by
\begin{eqnarray}\label{estimate_(b)_dilatant2}
&&\left|\varepsilon^{-1}\left(\int_{\widetilde \Omega_\varepsilon}\partial_{y_3} \tilde u_{\varepsilon,3}\tilde  u_\varepsilon \tilde R^\varepsilon_q \tilde v \,dx'dy_3+ \int_{\widetilde \Omega_\varepsilon}\tilde u_{\varepsilon,3}\partial_{y_3} \tilde u_\varepsilon\, \tilde R^\varepsilon_q \tilde v\,dx'dy_3\right)\right|\\
&&\leq C \varepsilon^{-2}\|\partial_{y_3}\tilde u_\varepsilon\|_{L^r(\widetilde\Omega_\varepsilon)^3}\|\tilde u_\varepsilon\|_{L^{q'}(\widetilde\Omega_\varepsilon)^3}\|\tilde v\|_{W^{1,q}_0(\Omega)^3}.\nonumber
\end{eqnarray}
It remains to estimate $ \|\tilde u_\varepsilon\|_{L^{q'}(\widetilde\Omega_\varepsilon)^3}$ using (\ref{estimates_u_tilde2less1})-(\ref{estimates_u_tilde2equal1}). So, our goal is to have $q$ the closer possible to $r$. Therefore, by interpolating between $L^r(\widetilde \Omega_\varepsilon)$ and $W_0^{1,r}(\widetilde \Omega_\varepsilon)$ with the parameter $\theta\in [0,1]$ such that 
\begin{equation}\label{inter2_q1_2}{1\over q'}={\theta\over r}+{1\over q_2}(1-\theta),
\end{equation}
where $q_2$ is given by (\ref{def_q2_dilatant}), we have (\ref{interpola_2_caso2_dilatant}).

Taking into account (\ref{estimates_u_tilde2less1})-(\ref{estimates_u_tilde2equal1}) and (\ref{interpola_2_caso2_dilatant}), (\ref{estimate_(b)_dilatant2}) can be written by
\begin{equation*}
\displaystyle \left|\varepsilon^{-1}\left(\int_{\widetilde \Omega_\varepsilon}\partial_{y_3} \tilde u_{\varepsilon,3}\tilde  u_\varepsilon \tilde R^\varepsilon_q \tilde v \,dx'dy_3+ \int_{\widetilde \Omega_\varepsilon}\tilde u_{\varepsilon,3}\partial_{y_3} \tilde u_\varepsilon\, \tilde R^\varepsilon_q \tilde v\,dx'dy_3\right)\right| \leq \left\{\begin{array}{rcl}
C\varepsilon^{\theta-{4\over r}(\gamma-1)-1}\|\tilde v\|_{W^{1,q}_0(\Omega)^3} & \hbox{if}& \gamma<1,\\
\\
C\varepsilon^{\theta-2{\gamma-1\over r-1}-1}\|\tilde v\|_{W^{1,q}_0(\Omega)^3} & \hbox{if}& \gamma>1,\\
\\
C\varepsilon^{\theta-1}\|\tilde v\|_{W^{1,q}_0(\Omega)^3} & \hbox{if}& \gamma=1.
 \end{array}\right.
\end{equation*}
To have a positive power of $\varepsilon$, we must choose $\theta=1$ for $\gamma\leq 1$. For $\gamma>1$, we are not able to find a bound proportional to a positive power of $\varepsilon$. However then
$$\int_{\widetilde \Omega_\varepsilon}\partial_{y_3} \tilde u_{\varepsilon,3}\tilde  u_\varepsilon \tilde R^\varepsilon_q \tilde v \,dx'dy_3+ \int_{\widetilde \Omega_\varepsilon}\tilde u_{\varepsilon,3}\partial_{y_3} \tilde u_\varepsilon\, \tilde R^\varepsilon_q \tilde v\,dx'dy_3$$
is just bounded, for $\gamma>1$, and may have an influence in the limit. We do not consider this case and limit us to $\gamma\leq 1$. 

\begin{itemize}
\item[{\bf $b_1)$}] {\bf For $r\ge 3$}: We can choose $q=r$ because $\theta=1$ and, by (\ref{inter2_q1_2}), $q'=r$. As $q\ge r$, we have that $q\ge 3$ and by definition (\ref{def_q1_dilatant_term2}), we have that condition (\ref{ineq2_dilatant}) is satisfied. Then, (\ref{estimate_(b)_dilatant2}) can be written by
\begin{eqnarray*}
&&\left|\varepsilon^{-1}\left(\int_{\widetilde \Omega_\varepsilon}\partial_{y_3} \tilde u_{\varepsilon,3}\tilde  u_\varepsilon \tilde R^\varepsilon_r \tilde v \,dx'dy_3+ \int_{\widetilde \Omega_\varepsilon}\tilde u_{\varepsilon,3}\partial_{y_3} \tilde u_\varepsilon\, \tilde R^\varepsilon_r \tilde v\,dx'dy_3\right)\right|\\
&&\leq C \varepsilon^{-2}\|\partial_{y_3}\tilde u_\varepsilon\|_{L^r(\widetilde\Omega_\varepsilon)^3}\|\tilde u_\varepsilon\|_{L^{r}(\widetilde\Omega_\varepsilon)^3}\|\tilde v\|_{W^{1,r}_0(\Omega)^3},\nonumber
\end{eqnarray*}
and taking into account (\ref{estimates_u_tilde2less1}) and (\ref{estimates_u_tilde2equal1}), we obtain
\begin{equation*}
\displaystyle \left|\varepsilon^{-1}\left(\int_{\widetilde \Omega_\varepsilon}\partial_{y_3} \tilde u_{\varepsilon,3}\tilde  u_\varepsilon \tilde R^\varepsilon_r \tilde v \,dx'dy_3+ \int_{\widetilde \Omega_\varepsilon}\tilde u_{\varepsilon,3}\partial_{y_3} \tilde u_\varepsilon\, \tilde R^\varepsilon_r \tilde v\,dx'dy_3\right)\right| \leq \left\{\begin{array}{rcl}
C\varepsilon^{-{4\over r}(\gamma-1)}\|\tilde v\|_{W^{1,r}_0(\Omega)^3} & \hbox{if}& \gamma<1,\\
\\
C\|\tilde v\|_{W^{1,r}_0(\Omega)^3} & \hbox{if}& \gamma=1.
 \end{array}\right.
\end{equation*}
Observe that if $\gamma<1$, then $-{4\over r}(\gamma-1)>0$.

\item[{\bf $b_2)$}] {\bf For $2<r<3$}: As $\theta=1$, by (\ref{inter2_q1_2}), we have $q'=r$. In this case, we have $q_1={3q\over 3-q}$ and by condition (\ref{ineq2_dilatant}), we obtain
\begin{equation}\label{relation_q_r_dilatant}
q\ge \displaystyle {3r\over 4r-6}.
\end{equation}
If we consider $q=r$ in (\ref{relation_q_r_dilatant}), we obtain that $ \displaystyle {9\over 4}\leq r<3$. For $2<r<\displaystyle {9\over 4}$ we consider $q=\displaystyle {3r\over 4r-6}$.
\end{itemize}

Finally, to have a positive power of $\varepsilon$ in estimate of the second and third term in the second member of (\ref{inertial_case_2}), we have to choose
\begin{equation}\label{q_dilantant_term2}
q=\left\{\begin{array}{ccl}
r&\hbox{for}& r\ge 3 \hbox{ and } \gamma\leq 1,
\\
\\
\displaystyle {3r\over 4r-6}&\hbox{for}&  2<r<\displaystyle {9\over 4}\hbox{ and }\gamma\leq 1,
\\
\\
r&\hbox{for}&  \displaystyle {9\over 4}\leq r<3\hbox{ and }\gamma\leq 1.
\end{array}\right.
\end{equation}
\end{itemize}

Putting together (\ref{q_dilantant_term2}) and (\ref{q_dilantant_term1}), it is the definition (\ref{Cr_case_dilatant}) which is the good one for $C(r)$, where $C(r)$ is the conjugate of $q$. Therefore, for the choice of $q$ given by (\ref{q_dilantant_term2}), the estimate of the intertial term satisfies (\ref{inertial_alpha_dilatant}), as we wanted.\\

Finally, considering $C(r)$ the conjugate of $q$, defined by (\ref{Cr_case_dilatant}), and taking into account (\ref{estim_caso2_1_dilatant})-(\ref{inertial_alpha_dilatant}) in (\ref{extension_1}), we get the second estimate in (\ref{esti_P2}) and finally, using the Ne${\check{\rm c}}$as inequality, there exists a representative $\tilde P_\varepsilon\in L^{C(r)}_0(\Omega)$  such that
$$\|\tilde P_\varepsilon\|_{L^{C(r)}(\Omega)}\leq C\|\nabla\tilde P_\varepsilon\|_{W^{-1,C(r)}(\Omega)^3}\leq C\|\nabla_{\varepsilon}\tilde P_\varepsilon\|_{W^{-1,C(r)}(\Omega)^3},$$
which implies the first estimate in (\ref{esti_P2}).
\end{proof}

\subsection{Adaptation of the unfolding method}\label{sec:unfolding}
The change of variables (\ref{dilatacion}) does not provide the information we need about the behavior of $\tilde u_\varepsilon$ in the microstructure associated to $\widetilde\Omega_\varepsilon$. To solve this difficulty, we use an adaptation introduced in \cite{Anguiano_SuarezGrau} of the unfolding method from \cite{CDG}.

Let us recall that this adaptation of the unfolding method divides the domain $\widetilde\Omega_\varepsilon$ in cubes of lateral length $\varepsilon$ and vertical length $1$. Thus, given $(\tilde{\varphi}_{\varepsilon},  \tilde \psi_\varepsilon) \in W^{1,q}_0(\Omega)^3\times L^{q'}_0(\Omega)$, $1<q<+\infty$ and $1/q+1/q'=1$, we define $(\hat{\varphi}_{\varepsilon},  \hat \psi_\varepsilon)$ by
\begin{equation}\label{phihat}
\hat{\varphi}_{\varepsilon}(x^{\prime},y)=\tilde{\varphi}_{\varepsilon}\left( {\varepsilon}\kappa\left(\frac{x^{\prime}}{{\varepsilon}} \right)+{\varepsilon}y^{\prime},y_3 \right),\quad 
\hat{\psi}_{\varepsilon}(x^{\prime},y)=\tilde{\psi}_{\varepsilon}\left( {\varepsilon}\kappa\left(\frac{x^{\prime}}{{\varepsilon}} \right)+{\varepsilon}y^{\prime},y_3 \right),\quad \hbox{ a.e. }(x^{\prime},y)\in \omega\times Y,
\end{equation}
assuming $\tilde \varphi_\varepsilon$ and $\tilde \psi_\varepsilon$ are extended by zero outside $\omega$, where the function $\kappa:\R^2\to \mathbb{Z}^2$ is defined by 
$$\kappa(x')=k'\Longleftrightarrow x'\in Y'_{k',1},\quad\forall\,k'\in\mathbb{Z}^2.$$

\begin{remark}\label{remarkCV_1}
The function $\kappa$ is well defined up to a set of zero measure in $\R^2$ (the set $\cup_{k'\in \mathbb{Z}^2}\partial Y'_{k',1}$). Moreover, for every $\varepsilon>0$, we have
$$\kappa\left({x'\over \varepsilon}\right)=k'\Longleftrightarrow x'\in Y'_{k',\varepsilon}.$$
\end{remark}
Following the proof of \cite[Lemma 4.9]{Anguiano_SuarezGrau}, we have the following estimates relating  $(\hat \varphi_\varepsilon,\hat \psi_\varepsilon)$ and $(\tilde \varphi_\varepsilon, \tilde \psi_\varepsilon)$.
\begin{lemma}\label{estimates_relation} The sequence $(\hat \varphi_\varepsilon, \hat \psi_\varepsilon)$ defined by (\ref{phihat}) satisfies the following estimates:
\begin{equation*}
\begin{array}{c}
\displaystyle
\medskip\|\hat \varphi_\varepsilon\|_{L^q(\omega\times Y)^3}\leq \|\tilde \varphi_\varepsilon\|_{L^q(\Omega)^3},\\
\medskip
\displaystyle \|D_{y'} \hat \varphi_\varepsilon\|_{L^q(\omega\times Y)^{3\times 2}}\leq \varepsilon \|D_{x'}\tilde \varphi_\varepsilon\|_{L^q(\Omega)^{3\times 2}},\quad \|\partial_{y_3} \hat \varphi_\varepsilon\|_{L^q(\omega\times Y)^{3 }}\leq   \|\partial_{y_3}\tilde \varphi_\varepsilon\|_{L^q(\Omega)^{3}},\\
\medskip
\displaystyle \|\mathbb{D}_{y'}[\hat \varphi_\varepsilon]\|_{L^q(\omega\times Y)^{3\times 2}}\leq \varepsilon \|\mathbb{D}_{x'}[\tilde \varphi_\varepsilon]\|_{L^q(\Omega)^{3\times 2}},\quad \|\partial_{y_3}[\hat u_\varepsilon]\|_{L^q(\omega\times Y)^{3 }}\leq   \|\partial_{y_3}[\tilde \varphi_\varepsilon]\|_{L^q(\Omega)^{3}},\\
\medskip
\|\hat \psi_\varepsilon\|_{L^{q'}(\omega\times Y)}\leq \|\tilde \psi_\varepsilon\|_{L^{q'}(\Omega)}.
\end{array}
\end{equation*}
\end{lemma}
\begin{definition}\label{Definition_unfolded}[Unfolded velocity and pressure] Let us define the unfolded velocity and pressure $(\hat u_\varepsilon, \hat P_\varepsilon)$ from  $(\tilde u_\varepsilon, \tilde P_\varepsilon)$ depending on   $r$:
\begin{itemize}
\item[--] ({\it Pseudoplastic of Newtonian fluids}) From $(\tilde u_\varepsilon, \tilde P_\varepsilon)\in H^1_0(\Omega)^3\times L^{2}_0(\Omega)$, we define $(\hat u_\varepsilon, \hat P_\varepsilon)$ by  using (\ref{phihat}) with $\tilde \varphi_\varepsilon=\tilde u_\varepsilon$ and $\tilde \psi_\varepsilon=\tilde P_\varepsilon$.\\

\item[--] ({\it Dilatant fluids}) From $(\tilde u_\varepsilon, \tilde P_\varepsilon)\in W^{1,r}_0(\Omega)^3\times L^{r'}_0(\Omega)$, we define $(\hat u_\varepsilon, \hat P_\varepsilon)$ by using (\ref{phihat}) with $\tilde \varphi_\varepsilon=\tilde u_\varepsilon$ and  $\tilde \psi_\varepsilon=\tilde P_\varepsilon$.
\end{itemize}
\end{definition}
\begin{remark}\label{remarkCV2}
For $k^{\prime}\in \mathcal{K}_{\varepsilon}$, the restrictions of $(\hat{u}_{\varepsilon},   \hat P_\varepsilon)$  to $Y^{\prime}_{k^{\prime},{\varepsilon}}\times Y$ does not depend on $x^{\prime}$, whereas as a function of $y$ it is obtained from $(\tilde{u}_{\varepsilon},  \tilde{P}_{\varepsilon})$ by using the change of variables $\displaystyle y^{\prime}=\frac{x^{\prime}- {\varepsilon}k^{\prime}}{{\varepsilon}}$,
which transforms $Y_{k^{\prime}, {\varepsilon}}$ into $Y$.
\end{remark}
Now, from  estimates of the extended velocity (\ref{estimates_u_tilde}), (\ref{estimates_u_tilde2less1}) and (\ref{estimates_u_tilde2equal1}), and from estimates of the pressure (\ref{esti_P}) and (\ref{esti_P2}) together with  Lemma  \ref{estimates_relation},  we have the following estimates for $(\hat u_\varepsilon,\hat P_\varepsilon)$.

\begin{lemma}\label{estimates_hat} We have the following estimates for the unfolded functions $(\hat u_\varepsilon,\hat P_\varepsilon)$ depending on the type of fluid:
 \begin{itemize}
 \item[(i)] {\it (Pseudoplastic or Newtonian fluids)}  Consider $1< r\leq 2$ and $\gamma \leq 1$. There exists a constant $C>0$ independent of $\varepsilon$, such that
 \begin{eqnarray}\medskip
 \|\hat u_\varepsilon\|_{L^2(\omega\times Y)^3}\leq C\varepsilon^{2-\gamma},& 
 \|D_{y}\hat u_\varepsilon\|_{L^2(\omega\times Y)^{3\times 3}}\leq C\varepsilon^{2-\gamma},& 
  \|\mathbb{D}_{y}[\hat u_\varepsilon]\|_{L^2(\omega\times Y)^{3\times 3}}\leq C\varepsilon^{2-\gamma},\label{estim_u_hat}\\
 \medskip
& \|\hat P_\varepsilon\|_{L^{C(r)}(\omega\times Y)}\leq C,&\label{estim_P_hat}
    \end{eqnarray}
    where $C(r)$ is given by (\ref{Cr_case_pseudo}).
 \item[(ii)] ({\it Dilatant fluids}) Consider $r>2$ and $\gamma \leq 1$. There exists a constant $C>0$ independent of $\varepsilon$, such that  estimates (\ref{estim_u_hat}) hold and also, depending on the value of $\gamma$, we have
 \begin{itemize}
 \item[--] If $\gamma<1$, it holds
 \begin{equation}\label{estim_u_hat2}
 \begin{array}{c}
 \|\hat u_\varepsilon\|_{L^r(\omega\times Y)^3}\leq C\varepsilon^{-{2\over r}(\gamma-1)+1},\quad  
 \|D_{y}\hat u_\varepsilon\|_{L^r(\omega\times Y)^{3\times 3}}\leq C\varepsilon^{-{2\over r}(\gamma-1)+1},\\
\\
  \|\mathbb{D}_{y}[\hat u_\varepsilon]\|_{L^r(\omega\times Y)^{3\times 3}}\leq C\varepsilon^{-{2\over r}(\gamma-1)+1}.
 \end{array} \end{equation}
  \item[--] If $\gamma=1$, it holds
 \begin{equation}\label{estim_u_hat22_1}
 \begin{array}{c}
 \|\hat u_\varepsilon\|_{L^r(\omega\times Y)^3}\leq C\varepsilon,\quad 
 \|D_{y}\hat u_\varepsilon\|_{L^r(\omega\times Y)^{3\times 3}}\leq C\varepsilon,\\
 \\
  \|\mathbb{D}_{y}[\hat u_\varepsilon]\|_{L^r(\omega\times Y)^{3\times 3}}\leq C\varepsilon.
  \end{array}
  \end{equation}

 \end{itemize}
 Moreover, we have the following estimate for the pressure
  \begin{eqnarray}\medskip
  \|\hat P_\varepsilon\|_{L^{C(r)}(\omega\times Y)}\leq C,&\label{estim_P_hat2}
    \end{eqnarray}
    where $C(r)$ is given by (\ref{Cr_case_dilatant}).
 \end{itemize}
 \end{lemma}

\subsection{Variational formulation of the unfolded functions}

In this section, we derive the variational formulation of the unfolded functions, defined in Section \ref{sec:unfolding}, which will be useful in the proofs of Theorems \ref{mainthmPseudo} and \ref{mainthmDilatant}.\\

First, we choose a test function $v(x',y)\in \mathcal{D}(\omega; C^\infty_{\#}(Y)^3)$ with $v(x',y)=0$ in $\omega\times T$. Multiplying (\ref{2}) by $v(x',x'/\varepsilon,y_3)$, integrating by parts, and taking into account the extension of $\tilde u_\varepsilon$ and $\tilde P_\varepsilon$, we have
\begin{equation}\label{variational_formulation_tilde}
\begin{array}{l}
\displaystyle
\medskip
{\color{black} \mu\varepsilon^\gamma }(\eta_0-\eta_\infty)\int_{ \Omega }(1+\lambda|\mathbb{D}_\varepsilon[\tilde u_\varepsilon]|^2)^{{r\over 2}-1}\mathbb{D}_\varepsilon[\tilde u_\varepsilon] :\left( \mathbb{D}_{x'}[ v]+ \varepsilon^{-1}\mathbb{D}_{y}[ v]\right)\,dx'dy_3\\
\medskip
\displaystyle
+{\color{black} \mu\varepsilon^\gamma}\eta_\infty\int_{ \Omega}\mathbb{D}_\varepsilon[\tilde u_\varepsilon] :\left( \mathbb{D}_{x'}[ v]+ \varepsilon^{-1}\mathbb{D}_{y}[ v]\right)\,dx'dy_3+\int_{\Omega}\left(\tilde u_\varepsilon\cdot \nabla_{\varepsilon}\right)\tilde u_\varepsilon\,v\,dx'dy_3\\
\medskip
\displaystyle
-\int_{\Omega}\tilde P_\varepsilon \left({\rm div}_{x'}v'+\varepsilon^{-1}{\rm div}_{y}v\right)\,dx'dy_3=\int_{\Omega} f'\cdot v'\,dx'dy_3+O_\varepsilon,
\end{array}
\end{equation}
where $O_\varepsilon$ is a generic real sequence which tends to zero with $\varepsilon$ and can change from line to line.\\

In this section, we are interested in giving the variational formulation of the unfolded functions using (\ref{variational_formulation_tilde}). We study in details the inertial term as it is the main novelty in this paper. The inertial term, which appears in (\ref{variational_formulation_tilde}), can be written as 
\begin{eqnarray}\label{inertial_variational_formulation}
\displaystyle \int_{ \Omega} \left(\tilde u_\varepsilon\cdot \nabla_{\varepsilon}\right)\tilde u_\varepsilon\,v\,dx'dy_3=-\int_{ \Omega} \tilde u_\varepsilon\tilde \otimes \tilde u_\varepsilon:D_{x'}  v\,dx'dy_3
\displaystyle +{1\over \varepsilon}\left(\int_{ \Omega}\partial_{y_3} \tilde u_{\varepsilon,3}\tilde  u_\varepsilon  v \,dx'dy_3+ \int_{ \Omega}\tilde u_{\varepsilon,3}\partial_{y_3} \tilde u_\varepsilon\,  v\,dx'dy_3\right)\,,
\end{eqnarray}
where $(u \tilde \otimes w)_{ij}=u_i w_j$, $i=1,2$, $j=1,2,3$.

Taking into account the adaptation of the unfolding method given in Section \ref{sec:unfolding} (see \cite[Section 4.2, Section 6]{Anguiano_SuarezGrau} for more details), by the change of variables given in Remark \ref{remarkCV2} we can deduce
\begin{eqnarray}\label{unfolded_inertial1} 
\int_{\omega\times Y} v dx^{\prime}dy
=  \int_{\Omega} v dx^{\prime}dy_3+O_\varepsilon,\quad \int_{\omega\times Y} {D}_{y^{\prime}}v dx^{\prime}dy
= {\varepsilon} \int_{\Omega} {D}_{x^{\prime}}v dx^{\prime}dy_3+O_\varepsilon,
\end{eqnarray}
\begin{eqnarray}\label{unfolded_inertial2}
\int_{\omega\times Y} \partial_{y_3}\hat{u}_{\varepsilon} dx^{\prime}dy=\int_{\Omega} \partial_{y_3}\tilde{u}_{\varepsilon} dx^{\prime}dy_3,
\end{eqnarray}
and 
\begin{eqnarray}\label{unfolded_inertial3}
\int_{\omega\times Y} \hat{u}_{\varepsilon} dx^{\prime}dy=\int_{\Omega}\tilde{u}_{\varepsilon} dx^{\prime}dy_3.
\end{eqnarray}
Then, by the change of variables given in Remark \ref{remarkCV2} and taking into account (\ref{unfolded_inertial1})--(\ref{unfolded_inertial3}) in (\ref{inertial_variational_formulation}), we obtain 
\begin{eqnarray}\label{inertial_variational_formulation_unfolded}\nonumber
&&\displaystyle \int_{ \Omega} \left(\tilde u_\varepsilon\cdot \nabla_{\varepsilon}\right)\tilde u_\varepsilon\,v\,dx'dy_3\\ \nonumber
&=& \varepsilon^{-1}\left(-\int_{ \omega\times Y} \hat u_\varepsilon\tilde \otimes \hat u_\varepsilon:D_{y'}  v\,dx'dy
\displaystyle +\int_{ \omega\times Y}\partial_{y_3} \hat u_{\varepsilon,3}\hat  u_\varepsilon  v \,dx'dy+ \int_{ \omega\times Y}\hat u_{\varepsilon,3}\partial_{y_3} \hat u_\varepsilon\,  v\,dx'dy\right)+O_\varepsilon\\
&=&\varepsilon^{-1}\int_{\omega\times Y}(\hat u_\varepsilon\cdot \nabla_y)\hat u_\varepsilon\,v\,dx'dy+O_\varepsilon.
\end{eqnarray}

Finally, we apply the change of variables given in Remark \ref{remarkCV2} to the other terms in (\ref{variational_formulation_tilde}) (see proof of Theorem 2.1 in \cite{Carreau_Ang_Bonn_SG} for more details) and taking into account (\ref{inertial_variational_formulation_unfolded}) for the inertial term, we can deduce
\begin{equation}\label{form_var_hat}\begin{array}{l}
\displaystyle
\medskip
{\color{black} \mu}\varepsilon^{\gamma-1}(\eta_0-\eta_\infty)\int_{ \omega\times Y }(1+\lambda |\varepsilon^{-1}\mathbb{D}_{y}[\hat u_\varepsilon]|^2)^{{r\over 2}-1}\left(\varepsilon^{-1} \mathbb{D}_{y}[\hat u_\varepsilon] \right): \mathbb{D}_{y}[  v]\,dx'dy\\
\medskip
\displaystyle
+{\color{black} \mu}\varepsilon^{\gamma-1}\eta_\infty\int_{ \omega\times Y}\varepsilon^{-1}\mathbb{D}_{y}[\hat u_\varepsilon] : \mathbb{D}_{y}[ v]\,dx'dy+\varepsilon^{-1}\int_{\omega\times Y}(\hat u_\varepsilon\cdot \nabla_y)\hat u_\varepsilon\,v\,dx'dy\\
\medskip
\displaystyle
-\int_{\omega\times Y}\hat P_\varepsilon\, {\rm div}_{x'}v'\,dx'dy-\varepsilon^{-1}\int_{\omega\times Y}\hat P_\varepsilon\,{\rm div}_{y}v\,dx'dy=\int_{\omega\times Y} f'\cdot v'\,dx'dy+O_\varepsilon.
\end{array}
\end{equation}
Now, let us define the functional $J_r$ by
$$
J_r(v)={\eta_0-\eta_\infty\over r\lambda}\int_{\omega\times Y}(1+\lambda|\mathbb{D}_y[v]|^2)^{r\over 2}dx'dy+{\eta_\infty\over 2}\int_{\omega\times Y}|\mathbb{D}_y[v]|^2dx'dy.
$$
Observe that $J_r$ is convex and Gateaux differentiable on $L^q(\omega; W^{1,q}_{\#}(Y)^3)$ with  $q=\max\{2,r\}$, (see \cite[Proposition 2.1 and Section 3]{Baranger} for more details) and $A_r=J'_r$ is given by
$$
(A_r(w),v)=(\eta_0-\eta_\infty)\int_{\omega\times Y}(1+\lambda|\mathbb{D}_y[w]|^2)^{{r\over 2}-1}\mathbb{D}_y[w]:\mathbb{D}_y[v]dx'dy+\eta_\infty\int_{\omega\times Y}\mathbb{D}_y[w]:\mathbb{D}_y[v]dx'dy.
$$
Applying \cite[Proposition 1.1., p.158]{Lions2}, in particular, we have that $A_r$ is monotone, i.e.,
\begin{equation}\label{monotonicity}
(A_r(w)-A_r(v),w-v)\ge0,\quad \forall w,v\in L^q(\omega; W^{1,q}_{\#}(Y)^3).
\end{equation}
On the other hand, for all $\varphi\in \mathcal{D}(\omega; C^\infty_{\#}(Y)^3)$ satisfying the divergence conditions ${\rm div}_{x'}\int_Y \varphi'\,dy=0$ in $\omega$ and ${\rm div}_{y}\varphi=0$ in $\omega\times Y$, we choose $v_\varepsilon$ defined by
$$
v_\varepsilon=\varphi-\varepsilon^{-1}\hat u_\varepsilon,
$$
as a test function in (\ref{form_var_hat}). Taking into account that ${\rm div}_{\varepsilon}\tilde u_\varepsilon=0$, we get that $\varepsilon^{-1}{\rm div}_y \hat u_\varepsilon=0$, and then we obtain
\begin{equation*}\begin{array}{l}
\displaystyle
\medskip
{\color{black} \mu}\varepsilon^{\gamma-1}(A_r(\varepsilon^{-1}\hat u_\varepsilon),v_\varepsilon)+\varepsilon^{-1}\int_{\Omega}(\hat u_\varepsilon\cdot \nabla_y)\hat u_\varepsilon\,v_\varepsilon\,dx'dy-\int_{\omega\times Y}\hat P_\varepsilon\, {\rm div}_{x'}v'_\varepsilon\,dx'dy=\int_{\omega\times Y} f'\cdot v'_\varepsilon\,dx'dy+O_\varepsilon,
\end{array}
\end{equation*}
which is equivalent to
\begin{equation*}\begin{array}{l}
\displaystyle
\medskip
{\color{black} \mu}\varepsilon^{\gamma-1}(A_r(\varphi)-A_r(\varepsilon^{-1}\hat u_\varepsilon),v_\varepsilon)-{\color{black} \mu}\varepsilon^{\gamma-1}(A_r(\varphi),v_\varepsilon)-\varepsilon^{-1}\int_{\omega\times Y}(\hat u_\varepsilon\cdot \nabla_y)\hat u_\varepsilon\,v_\varepsilon\,dx'dy\\
\medskip
\displaystyle 
+\int_{\omega\times Y}\hat P_\varepsilon\, {\rm div}_{x'}v'_\varepsilon\,dx'dy=-\int_{\omega\times Y} f'\cdot v'_\varepsilon\,dx'dy+O_\varepsilon.
\end{array}
\end{equation*}
Due to (\ref{monotonicity}), we can deduce
\begin{equation*}\begin{array}{l}
\displaystyle
\medskip
{\color{black} \mu}\varepsilon^{\gamma-1}(A_r(\varphi),v_\varepsilon)+\varepsilon^{-1}\int_{\Omega}(\hat u_\varepsilon\cdot \nabla_y)\hat u_\varepsilon\,v_\varepsilon\,dx'dy-\int_{\omega\times Y}\hat P_\varepsilon\, {\rm div}_{x'}v'_\varepsilon\,dx'dy\ge\int_{\omega\times Y} f'\cdot v'_\varepsilon\,dx'dy+O_\varepsilon,
\end{array}
\end{equation*}
i.e.,
\begin{equation}\label{v_ineq_carreau}\begin{array}{l}
\displaystyle
\medskip
{\color{black} \mu}\varepsilon^{\gamma-1}(\eta_0-\eta_\infty)\int_{\omega\times Y}(1+\lambda|\mathbb{D}_y[\varphi]|^2)^{{r\over 2}-1}\mathbb{D}_y[\varphi]:\mathbb{D}_y[v_\varepsilon]dx'dy+{\color{black} \mu}\varepsilon^{\gamma-1}\eta_\infty\int_{\omega\times Y}\mathbb{D}_y[\varphi]:\mathbb{D}_y[v_\varepsilon]dx'dy\\
\medskip
\displaystyle+\varepsilon^{-1}\int_{\omega\times Y}(\hat u_\varepsilon\cdot \nabla_y)\hat u_\varepsilon\,v_\varepsilon\,dx'dy-\int_{\omega\times Y}\hat P_\varepsilon\, {\rm div}_{x'}v'_\varepsilon\,dx'dy\ge\int_{\omega\times Y} f'\cdot v'_\varepsilon\,dx'dy+O_\varepsilon.
\end{array}
\end{equation}

\begin{remark} In the case of Newtonian fluids ($r=2$), due to the linearity, the variational formulation is the following one:
\begin{equation}\label{v_ineq_Newtonian}
\begin{array}{l}
\displaystyle
\medskip
{\color{black} \mu}\varepsilon^{\gamma-2}\eta_0\int_{\omega\times Y} \mathbb{D}_{y}[\hat u_\varepsilon]: \mathbb{D}_{y}[\varphi]dx'dy+\varepsilon^{-1}\int_{\omega\times Y}(\hat u_\varepsilon\cdot \nabla_y)\hat u_\varepsilon\,\varphi\,dx'dy\\
\medskip
\displaystyle-\int_{\omega\times Y}\hat P_\varepsilon\, {\rm div}_{x'}\varphi' \,dx'dy=\int_{\omega\times Y} f'\cdot \varphi' \,dx'dy+O_\varepsilon,
\end{array}
\end{equation}
for $\varphi(x',y)\in \mathcal{D}(\omega; C^\infty_{\#}(Y)^3)$ with $\varphi(x',y)=0$ in $\omega\times T$.
\end{remark}

Finally, we have the following result on the inertial term, which we will use in the next sections.
\begin{proposition}
Consider $1< r<+\infty$ and $\gamma\leq1$. Then, the inertial term satisfies
\begin{equation}\label{bounded_term_inertial_phi}
\left|\int_{\omega\times Y}(\hat u_\varepsilon\cdot \nabla_y)\hat u_\varepsilon\,\varphi\,dx'dy\right|\leq C\varepsilon^{4-2\gamma},
\end{equation}
where $\varphi\in \mathcal{D}(\omega; C^\infty_{\#}(Y)^3)$.
\end{proposition}
\begin{proof}
The inertial term can be written as 
\begin{eqnarray}\label{inertial_case_limit_pseudo}
\displaystyle \int_{\omega\times Y}(\hat u_\varepsilon\cdot \nabla_y)\hat u_\varepsilon\, \varphi\,dx'dy=-\int_{\omega \times Y} \hat u_\varepsilon \otimes \hat u_\varepsilon:D_{y} \varphi\,dx'dy,
\end{eqnarray}
where $(u \otimes w)_{ij}=u_i w_j$, $i,j=1,2,3$. 

By using Hölder inequality and the first estimate in (\ref{estim_u_hat}), we have
\begin{equation*}\label{interpola_1_caso2_limit_pseudo}
\begin{array}{l}
\displaystyle \left|\int_{\omega \times Y} \hat u_\varepsilon \otimes \hat u_\varepsilon:D_{y} \varphi\,dx'dy\right| 
\leq \|\hat u_\varepsilon\|_{L^{2}(\omega\times Y)^3}^2\|D_{y}\varphi \|_{L^\infty(\omega\times Y)^{3\times 3}}\leq  C\varepsilon^{4-2\gamma},
\end{array}
\end{equation*}
and by (\ref{inertial_case_limit_pseudo}), we obtain (\ref{bounded_term_inertial_phi}).

\end{proof}

\subsection{Proof of Theorem \ref{mainthmPseudo}. Pseudoplastic or Newtonian fluids: case $1< r\leq 2$}\label{sec:compactness_pseudo} 

In this section, we give some appropriate compactness results, and then, we give the proof of Theorem  \ref{mainthmPseudo}.

 \begin{lemma} \label{lemma_compactness_pseudo} Consider $1< r\leq 2$, $\gamma\leq1$ and $C(r)$ given by (\ref{Cr_case_pseudo}). Then, there exist:
 \begin{itemize}
 \item[i)]  A subsequence, still denoted by $(\tilde u_\varepsilon, \tilde P_\varepsilon)$, chosen from a sequence of $(\tilde u_\varepsilon, \tilde P_\varepsilon)$ of solutions of (\ref{2}) (we also denote by $(\hat u_\varepsilon, \hat P_\varepsilon)$ the subsequence of the corresponding unfolded functions),
 
 \item[ii)]   $(\tilde u, \tilde P)\in  H^1_0(0,1;L^2(\omega)^3)\times L^{C(r)}_0(\omega)$, where  $\tilde u=0$ on $\omega \times \{0,1\}$ and $\tilde u_3\equiv 0$, 
 
 \item[iii)]  $\hat u\in L^2(\omega; H^1_{0,\#}(Y)^3)$, with $\hat u=0$ on $\omega\times Y'\times \{0,1\}$ and satisfying the relation 
 \begin{equation}\label{relation_mean}\int_{Y}\hat u(x',y)\,dy=\int_0^1\tilde u(x',y_3)\,dy_3\quad \hbox{with }\int_Y\hat u_3(x',y)\,dy=0,
 \end{equation}
 
 \end{itemize} 
such that, up to a subsequence, 
\begin{equation}\label{conv_vel_tilde}
\varepsilon^{\gamma-2}\tilde u_\varepsilon \rightharpoonup (\tilde u',0)\hbox{ weakly in } H^1(0,1;L^2(\omega)^3),
\end{equation}
\begin{equation}\label{conv_vel_gorro}
\varepsilon^{\gamma-2}\hat u_\varepsilon\rightharpoonup \hat u\hbox{ weakly in }L^2(\omega; H^1(Y)^3),\end{equation}
\begin{equation}\label{conv_pres_tilde}
\tilde P_\varepsilon\to \tilde P \hbox{ strongly in }L^{C(r)}(\Omega),
\end{equation}
\begin{equation}\label{conv_pres_hat}
\hat P_\varepsilon\to \tilde P \hbox{ strongly in }L^{C(r)}(\omega\times Y).
\end{equation}
Moreover, $\tilde u$ and $\hat u$ satisfy the following divergence conditions 
\begin{equation}\label{divxproperty}
{\rm div}_{x'}\left(\int_0^1\tilde u'(x',y_3)\,dy_3\right)=0\  \hbox{ in }\omega,\quad \left(\int_0^1\tilde u'(x',y_3)\,dy_3\right)\cdot n=0\  \hbox{ in }\partial \omega,
\end{equation}
\begin{equation}\label{divyproperty}
 {\rm div}_{y}\,\hat u(x',y)=0\  \hbox{ in }\omega\times Y_f,\quad   {\rm div}_{x'}\left(\int_{Y_f}\hat u'(x',y)\,dy\right)=0\  \hbox{ in }\omega, \quad\left(\int_{Y_f}\hat u'(x',y)\,dy\right)\cdot n=0\  \hbox{ on }\partial\omega.
\end{equation}
\end{lemma}
\begin{proof}Arguing as in \cite[Section 3.3]{Carreau_Ang_Bonn_SG2} for pseudoplastic or Newtonian fluids, we obtain the proof of this lemma.\\
\end{proof}

\begin{proof}[{\bf Proof of Theorem \ref{mainthmPseudo}}]

We recall that $1< r\leq 2$.    The proof will be divided in three steps. In the first step, we consider the case $1< r<2$ and $\gamma<1$ and we will obtain the homogenized behavior given by a coupled system, with a constant macroviscosity, and we will decouple it to obtain the macroscopic law. Proceeding similarly, in the second step  we will consider the case $1< r<2$ and $\gamma=1$ and in the last step, we will consider the case $r=2$ and $\gamma\leq 1$.\\

{\it Step 1.} Consider $1< r<2$ and $\gamma<1$. 

We follows the lines of the proof to obtain (\ref{v_ineq_carreau}) but choosing now $v_\varepsilon$ and $\varphi$ such that $v_\varepsilon=\varepsilon^{1-\gamma}\varphi-\varepsilon^{-1}\hat u_\varepsilon$  with $\varphi\in \mathcal{D}(\omega; C^\infty_{\#}(Y)^3)$ satisfying the divergence conditions ${\rm div}_{x'}\int_Y \varphi'\,dy=0$ in $\omega$ and ${\rm div}_{y}\varphi=0$ in $\omega\times Y$.  Then, we get
\begin{equation*}\begin{array}{l}
\displaystyle
\medskip
{\color{black} \mu}\varepsilon^{1-\gamma}(\eta_0-\eta_\infty)\int_{\omega\times Y}(1+\lambda\varepsilon^{2(1-\gamma)}|\mathbb{D}_y[\varphi]|^2)^{{r\over 2}-1}\mathbb{D}_y[\varphi]:\mathbb{D}_y[\varphi-\varepsilon^{\gamma-2}  \hat u_\varepsilon]\,dx'dy\\
\medskip
\displaystyle
+{\color{black} \mu}\varepsilon^{1-\gamma}\eta_\infty\int_{\omega\times Y}\mathbb{D}_y[\varphi]:\mathbb{D}_y[\varphi-\varepsilon^{\gamma-2}  \hat u_\varepsilon]dx'dy+\varepsilon^{-\gamma}\int_{\omega\times Y}(\hat u_\varepsilon\cdot \nabla_y)\hat u_\varepsilon\,(\varphi-\varepsilon^{\gamma-2}  \hat u_\varepsilon)\,dx'dy\\
\medskip
\displaystyle-\varepsilon^{1-\gamma}\int_{\omega\times Y}\hat P_\varepsilon\, {\rm div}_{x'}(\varphi'-\varepsilon^{\gamma-2}  \hat u'_\varepsilon)\,dx'dy\ge \varepsilon^{1-\gamma}\int_{\omega\times Y} f'\cdot (\varphi'-\varepsilon^{\gamma-2}  \hat u'_\varepsilon)\,dx'dy+O_\varepsilon.
\end{array}
\end{equation*}
Taking into account that
$$\int_{\omega\times Y}(\hat u_\varepsilon\cdot \nabla_y)\hat u_\varepsilon\, \hat u_\varepsilon\,dx'dy=0,$$
and dividing by $\varepsilon^{1-\gamma}$, we have
\begin{equation}\label{v_ineq_carreau_12_1}\begin{array}{l}
\displaystyle
\medskip
 {\color{black} \mu}(\eta_0-\eta_\infty)\int_{\omega\times Y}(1+\lambda\varepsilon^{2(1-\gamma)}|\mathbb{D}_y[\varphi]|^2)^{{r\over 2}-1}\mathbb{D}_y[\varphi]:\mathbb{D}_y[\varphi-\varepsilon^{\gamma-2}  \hat u_\varepsilon]\,dx'dy\\
\medskip
\displaystyle
+{\color{black} \mu} \eta_\infty\int_{\omega\times Y}\mathbb{D}_y[\varphi]:\mathbb{D}_y[\varphi-\varepsilon^{\gamma-2}  \hat u_\varepsilon]dx'dy+\varepsilon^{-1}\int_{\omega\times Y}(\hat u_\varepsilon\cdot \nabla_y)\hat u_\varepsilon\,\varphi\,dx'dy\\
\medskip
\displaystyle- \int_{\omega\times Y}\hat P_\varepsilon\, {\rm div}_{x'}(\varphi'-\varepsilon^{\gamma-2}  \hat u'_\varepsilon)\,dx'dy\ge  \int_{\omega\times Y} f'\cdot (\varphi'-\varepsilon^{\gamma-2}  \hat u'_\varepsilon)\,dx'dy+O_\varepsilon.
\end{array}
\end{equation}
Taking into account (\ref{bounded_term_inertial_phi}), we get that
$$\left|\varepsilon^{-1}\int_{\omega\times Y}(\hat u_\varepsilon\cdot \nabla_y)\hat u_\varepsilon\,\varphi\,dx'dy\right|\leq C\varepsilon^{3-2\gamma},$$
and since $\gamma< 1$, we deduce that the inertial term tends to zero, and we have
$$\varepsilon^{-1}\int_{\omega\times Y}(\hat u_\varepsilon\cdot \nabla_y)\hat u_\varepsilon\,\varphi\,dx'dy=O_\varepsilon.$$
Then, the variational formulation (\ref{v_ineq_carreau_12_1}) reads as follows
\begin{equation}\label{v_ineq_carreau_12}\begin{array}{l}
\displaystyle
\medskip
 {\color{black} \mu}(\eta_0-\eta_\infty)\int_{\omega\times Y}(1+\lambda\varepsilon^{2(1-\gamma)}|\mathbb{D}_y[\varphi]|^2)^{{r\over 2}-1}\mathbb{D}_y[\varphi]:\mathbb{D}_y[\varphi-\varepsilon^{\gamma-2}  \hat u_\varepsilon]\,dx'dy\\
\medskip
\displaystyle
+{\color{black} \mu} \eta_\infty\int_{\omega\times Y}\mathbb{D}_y[\varphi]:\mathbb{D}_y[\varphi-\varepsilon^{\gamma-2}  \hat u_\varepsilon]dx'dy\\
\medskip
\displaystyle- \int_{\omega\times Y}\hat P_\varepsilon\, {\rm div}_{x'}(\varphi'-\varepsilon^{\gamma-2}  \hat u'_\varepsilon)\,dx'dy\ge  \int_{\omega\times Y} f'\cdot (\varphi'-\varepsilon^{\gamma-2}  \hat u'_\varepsilon)\,dx'dy+O_\varepsilon.
\end{array}
\end{equation} 

Now, we can pass to the limit in (\ref{v_ineq_carreau_12}) when $\varepsilon\to 0$, and arguing as the proof of Theorem 2.1 in \cite{Carreau_Ang_Bonn_SG2} we can deduce (\ref{thm:system}).

{\it Step 2.} We consider $1< r<2$ and $\gamma=1$. 
 
Considering $\gamma=1$ in the variational formulation (\ref{v_ineq_carreau}), we get
\begin{equation}\label{v_ineq_carreau21}\begin{array}{l}
\displaystyle
\medskip
{\color{black} \mu}(\eta_0-\eta_\infty)\int_{\omega\times Y}(1+\lambda|\mathbb{D}_y[\varphi]|^2)^{{r\over 2}-1}\mathbb{D}_y[\varphi]:\mathbb{D}_y[\varphi-\varepsilon^{-1}\hat u_\varepsilon]dx'dy\\
\medskip
\displaystyle+{\color{black} \mu}\eta_\infty\int_{\omega\times Y}\mathbb{D}_y[\varphi]:\mathbb{D}_y[\varphi-\varepsilon^{-1}\hat u_\varepsilon]dx'dy
\displaystyle+\varepsilon^{-1}\int_{\omega\times Y}(\hat u_\varepsilon\cdot \nabla_y)\hat u_\varepsilon\,(\varphi-\varepsilon^{-1}\hat u_\varepsilon)\,dx'dy\\
\medskip
\displaystyle-\int_{\omega\times Y}\hat P_\varepsilon\, {\rm div}_{x'}(\varphi'-\varepsilon^{-1}\hat u'_\varepsilon)\,dx'dy\ge\int_{\omega\times Y} f'\cdot (\varphi'-\varepsilon^{-1}\hat u'_\varepsilon)\,dx'dy+O_\varepsilon,
\end{array}
\end{equation}
with $\varphi\in \mathcal{D}(\omega; C^\infty_{\#}(Y)^3)$ satisfying the divergence conditions ${\rm div}_{x'}\int_Y \varphi'\,dy=0$ in $\omega$ and ${\rm div}_{y}\varphi=0$ in $\omega\times Y$.

Taking into account that
$$\int_{\omega\times Y}(\hat u_\varepsilon\cdot \nabla_y)\hat u_\varepsilon\, \hat u_\varepsilon\,dx'dy=0,$$
and that using (\ref{bounded_term_inertial_phi}) with $\gamma=1$, we get that
$$\left|\varepsilon^{-1}\int_{\omega\times Y}(\hat u_\varepsilon\cdot \nabla_y)\hat u_\varepsilon\,\varphi\,dx'dy\right|\leq C\varepsilon,$$
then, we can deduce that the inertial term tends to zero, and we have
$$\varepsilon^{-1}\int_{\omega\times Y}(\hat u_\varepsilon\cdot \nabla_y)\hat u_\varepsilon\,(\varphi-\varepsilon^{-1}  \hat u_\varepsilon)\,dx'dy=O_\varepsilon.$$

Then, the variational formulation (\ref{v_ineq_carreau21}) reads  
\begin{equation*}\begin{array}{l}
\displaystyle
\medskip
{\color{black} \mu}(\eta_0-\eta_\infty)\int_{\omega\times Y}(1+\lambda|\mathbb{D}_y[\varphi]|^2)^{{r\over 2}-1}\mathbb{D}_y[\varphi]:\mathbb{D}_y[\varphi-\varepsilon^{-1}\hat u_\varepsilon]dx'dy+{\color{black} \mu}\eta_\infty\int_{\omega\times Y}\mathbb{D}_y[\varphi]:\mathbb{D}_y[\varphi-\varepsilon^{-1}\hat u_\varepsilon]dx'dy\\
\medskip
\displaystyle-\int_{\omega\times Y}\hat P_\varepsilon\, {\rm div}_{x'}(\varphi'-\varepsilon^{-1}\hat u'_\varepsilon)\,dx'dy\ge\int_{\omega\times Y} f'\cdot (\varphi'-\varepsilon^{-1}\hat u'_\varepsilon)\,dx'dy+O_\varepsilon.
\end{array}
\end{equation*}
Now, passing to the limit and arguing as in the proof of Theorem 2.1 in \cite{Carreau_Ang_Bonn_SG}, we obtain (\ref{thm:system_gamma1}).

 {\it Step 3.} We consider $r=2$ and $\gamma\leq 1$. 
 
 We consider the variational formulation (\ref{v_ineq_Newtonian}).  
 
 Taking into account (\ref{bounded_term_inertial_phi}), we get that
$$\left|\varepsilon^{-1}\int_{\omega\times Y}(\hat u_\varepsilon\cdot \nabla_y)\hat u_\varepsilon\,\varphi\,dx'dy\right|\leq C\varepsilon^{3-2\gamma},$$
and since $\gamma\leq 1$, we deduce that the inertial term tends to zero, and we have
$$\varepsilon^{-1}\int_{\omega\times Y}(\hat u_\varepsilon\cdot \nabla_y)\hat u_\varepsilon\,\varphi\,dx'dy=O_\varepsilon.$$
Then, the variational formulation (\ref{v_ineq_Newtonian}) reads as follows
$$\begin{array}{l}
\displaystyle
\medskip
{\color{black} \mu}\varepsilon^{\gamma-2}\eta_0\int_{\omega\times Y} \mathbb{D}_{y}[\hat u_\varepsilon]: \mathbb{D}_{y}[\varphi]dx'dy-\int_{\omega\times Y}\hat P_\varepsilon\, {\rm div}_{x'}\varphi' \,dx'dy=\int_{\omega\times Y} f'\cdot \varphi' \,dx'dy+O_\varepsilon,
\end{array}
$$
for  $\varphi(x',y)\in \mathcal{D}(\omega; C^\infty_{\#}(Y)^3)$ with $\varphi(x',y)=0$ in $\omega\times T$. Now, passing to the limit and arguing as in the proof of~\cite[Theorem 2.5]{Carreau_Ang_Bonn_SG2}, we get the linear effective 2D Darcy's law 
(\ref{thm:system}), which concludes the proof of Theorem \ref{mainthmPseudo}-$(ii)$.

\end{proof}

\subsection{Proof of Theorem \ref{mainthmDilatant}. Dilatant fluids: case $r>2$} \label{sec:compactness_dilatant}

In this section, we give some appropriate compactness results, and then, we give the proof of Theorem  \ref{mainthmDilatant}.

 \begin{lemma} \label{lemma_compactness_dilatant}Consider $r> 2$, $\gamma\le 1$ and and $C(r)$ defined by (\ref{Cr_case_dilatant}).  Then, there exist:
 \begin{itemize}
 \item[i)]  A subsequence, still denoted by $(\tilde u_\varepsilon, \tilde P_\varepsilon)$, chosen from a sequence of $(\tilde u_\varepsilon, \tilde P_\varepsilon)$ of solutions of (\ref{2}) (we also denote by $(\hat u_\varepsilon, \hat P_\varepsilon)$ the subsequence of the corresponding unfolded functions),
 
 \item[ii)]   $(\tilde u, \tilde P)\in   H^{1}(0,1;L^2(\omega)^3)\times L^{C(r)}_0(\omega)$ if $\gamma<1$ and $(\tilde u, \tilde P)\in   W^{1,r}_0(0,1;L^r(\omega)^3)\times L^{C(r)}_0(\omega)$ if $\gamma=1$, where $\tilde u=0$ on $\omega \times \{0,1\}$ and $\tilde u_3\equiv 0$, 
 
 \item[iii)]  $\hat u\in L^2(\omega; H^{1}_{0,\#}(Y)^3)$ if $\gamma<1$ and $\hat u\in L^r(\omega; W^{1,r}_{0,\#}(Y)^3)$ if $\gamma=1$, with $\hat u=0$ on $\omega\times Y'\times \{0,1\}$ and satisfying the relation 
 \begin{equation}\label{relation_mean2}\int_{Y}\hat u(x',y)\,dy=\int_0^1\tilde u(x',y_3)\,dy_3\quad \hbox{with }\int_Y\hat u_3(x',y)\,dy=0,
 \end{equation}
 \end{itemize} 
such that, up to a subsequence,
\begin{equation}\label{conv_vel_tilde2}
\varepsilon^{\gamma-2}\tilde u_\varepsilon \rightharpoonup (\tilde u',0)\hbox{ weakly in } H^{1}(0,1;L^2(\omega)^3)\hbox{ if }\gamma<1,
\end{equation}
\begin{equation}\label{conv_vel_tilde2_2}
\varepsilon^{-1}\tilde u_\varepsilon \rightharpoonup (\tilde u',0)\hbox{ weakly in } W^{1,r}(0,1;L^r(\omega)^3)\hbox{ if }\gamma=1,
\end{equation}
\begin{equation}\label{conv_vel_gorro2}
\varepsilon^{\gamma-2}\hat u_\varepsilon\rightharpoonup \hat u\hbox{ weakly in }L^2(\omega; H^{1}(Y)^3)\hbox{ if }\gamma<1,
\end{equation}
\begin{equation}\label{conv_vel_gorro2_2}
\varepsilon^{-1}\hat u_\varepsilon\rightharpoonup \hat u\hbox{ weakly in }L^r(\omega; W^{1,r}(Y)^3)\hbox{ if }\gamma=1,
\end{equation}
\begin{equation}\label{conv_pres_tilde2}
\tilde P_\varepsilon\to \tilde P \hbox{ strongly in }L^{C(r)}(\Omega),
\end{equation}
\begin{equation}\label{conv_pres_hat2}
\hat P_\varepsilon\to \tilde P \hbox{ strongly in }L^{C(r)}(\omega\times Y).
\end{equation}

Moreover, $\tilde u$ and $\hat u$ satisfy the following divergence conditions in both cases
\begin{equation}\label{divxproperty2}
{\rm div}_{x'}\left(\int_0^1\tilde u'(x',y_3)\,dy_3\right)=0\  \hbox{ in }\omega,\quad \left(\int_0^1\tilde u'(x',y_3)\,dy_3\right)\cdot n=0\  \hbox{ in }\partial \omega,
\end{equation}
\begin{equation}\label{divyproperty2}
 {\rm div}_{y}\,\hat u(x',y)=0\  \hbox{ in }\omega\times Y_f,\quad   {\rm div}_{x'}\left(\int_{Y_f}\hat u'(x',y)\,dy\right)=0\  \hbox{ in }\omega, \quad\left(\int_{Y_f}\hat u'(x',y)\,dy\right)\cdot n=0\  \hbox{ on }\partial\omega.
\end{equation}
\end{lemma}
\begin{proof}Arguing as in \cite[Section 3.3]{Carreau_Ang_Bonn_SG2} for dilatant fluids, we obtain the proof of this lemma.\\
\end{proof}

\begin{proof}[{\bf Proof of Theorem \ref{mainthmDilatant}}]
We recall that in this case $r>2$. The proof will be divided in two steps. In the first step, we obtain the homogenized behavior in the case $\gamma<1$ and in the second step, we consider the case $\gamma=1$.\\

{\it Step 1.}  Consider $r>2$ and $\gamma<1$. Arguing as in {\it Step 1} of the proof of Theorem \ref{mainthmPseudo}, we have the variational inequality (\ref{v_ineq_carreau_12}). We pass to the limit in every term of  (\ref{v_ineq_carreau_12}) arguing as in the proof of Theorem 2.3 in \cite{Carreau_Ang_Bonn_SG2}.

{\it Step 2}. Consider $r>2$ and $\gamma=1$. 

Arguing as in the proof of Theorem 2.3 in \cite{Carreau_Ang_Bonn_SG2}, we deduce the nonlinear 2D Darcy's law of Carreau type (\ref{thm:system_gamma1}),  where the permeability function $\mathcal{U}:\R^2\to \R^2$ is defined by (\ref{permfunc123})  with $(w_{\xi'}, \pi_{\xi'})\in W^{1,r}_{0,\#}(Y_f)^3\times L^{r'}_{0,\#}(Y_f)$, for every $\xi'\in\R^2$, the unique solution of the local Stokes system (\ref{LocalProblemNonNewtonian})  with nonlinear viscosity given by the Carreau law (\ref{Carreau}).

\end{proof}

\section{Numerical simulations}\label{Sect:Numerics}

In this section, we propose a numerical method of Newton's type to solve the nonlinear Darcy's law~\eqref{thm:system_gamma1}, and provide some numerical results to show its practical applicability.

\subsection{Numerical method}\label{Sec:numerical_method}

 To fix the ideas, assume that $1<r<2$ (the case $r>2$ will be treated similarly). In that case, dropping the tilde symbol for simplicity, we seek a pressure $P\in H^1(\omega)\cap L^2_0(\omega)$ such that for every test function $\psi\in H^1(\omega)$,
\begin{equation}\label{WeakFormDarcy}
	\int_{\omega} \mathcal{U}(f'(x')-\nabla_{x'}P(x'))\cdot \nabla_{x'}\psi(x')\,  dx' = 0.
\end{equation}
Defining $\mathcal F(P)$ in the dual space of $H^1(\omega)$ by
\[
\forall \psi\in H^1(\omega)\quad  \langle \mathcal F(P),\psi \rangle = 	\int_{\omega} \mathcal{U}(f'(x')-\nabla_{x'}P(x'))\cdot \nabla_{x'}\psi(x')\,  dx',
\]
solving Darcy's law is equivalent to solving the equation $\mathcal F(P)=0$. Since the permeability function $\mathcal{U}$ is nonlinear, this equation cannot be solved directly. Hence, we rely on Newton's method to construct a sequence of approximated pressure fields $P_k$ that converges to the target pressure $P$ satisfying~\eqref{WeakFormDarcy}.

This approach requires differentiating the permeability function $\mathcal U$, which in turn requires differentiating the resolvent mapping $\xi'\in \R^2\mapsto \mathcal R(\xi'):=w_{\xi'}$ associated with the cell problem~\eqref{LocalProblemNonNewtonian}. For a fixed $\xi'\in \R^2$, the solution $w_{\xi'}$ to the local problem~\eqref{LocalProblemNonNewtonian} satisfies the following variational formulation: for every $\phi\in H^1_{0,\#}(Y_f)^3$ such that ${\rm div}_y\phi = 0$ in $Y_f$, 
\begin{equation*}
	\mu \int_{Y_f} \eta_r(\mathbb{D}_y[w_{\xi'}])\mathbb{D}_y[w_{\xi'}]:\mathbb{D}_y[\phi] = \int_{Y_f}\xi'\cdot \phi'.
\end{equation*}
Differentiating the previous relation with respect to $\xi'$, in the direction $\delta\xi'\in \R^2$, we obtain that $h:=\langle \mathcal R(\xi'), \delta\xi'\rangle$ belongs to $H^1_{0,\#}(Y_f)^3$, satisfies the condition ${\rm div}_y h = 0$ in $Y_f$ and the variational formulation: for every $\phi\in H^1_{0,\#}(Y_f)^3$ such that ${\rm div}_y\phi = 0$ in $Y_f$, 
\begin{equation}\label{VarFormCellDerivative}
		 \mu\int_{Y_f} \left\lbrace\eta_r(\mathbb{D}_y[w_{\xi'}])\mathbb{D}_y[h]:\mathbb{D}_y[\phi] + \langle \eta_r'(\mathbb{D}_y[w_{\xi'}]),\mathbb{D}_y[h]  \rangle \mathbb{D}_y[w_{\xi'}] \right\rbrace:\mathbb{D}_y[\phi]\\
		 = \int_{Y_f}\delta\xi'\cdot \phi'.
\end{equation}
In the above expression, $\langle \eta_r'(\mathbb{D}_y[w_{\xi'}]),\mathbb{D}_y[h]\rangle$ is the differential of the Carreau law~\eqref{Carreau} evaluated at the deformation rate $\mathbb{D}_y[w_{\xi'}]$, in the direction of $\mathbb{D}_y[h]$, and is given by
\[
\langle \eta_r'(\mathbb{D}_y[w_{\xi'}]),\mathbb{D}_y[h]\rangle = \lambda(r-2)(\eta_0-\eta_{\infty})(1+\lambda|\mathbb{D}_y[w_{\xi'}]|^2)^{\frac{r}{2}-2}\mathbb{D}_y[w_{\xi'}]:\mathbb{D}_y[h].
\]

By~\cite[Corollary 2.2]{Saramito}, the solution $h$ of the variational formulation~\eqref{VarFormCellDerivative} exists and is unique.  We deduce that the differential of $\mathcal U$ at $\xi'$ in the direction $\delta\xi'$ is given by
\begin{equation}\label{InitialDefDU}
	\langle D\mathcal U(\xi'),\delta\xi' \rangle = \int_{Y_f} h'. 
\end{equation} 
Moreover, for a fixed vector $\xi'$, $h'$ depends linearly on $\delta\xi'$, therefore it can be represented for any $\delta\xi'=(\delta\xi_1,\delta\xi_2)$ using the solutions $h_1,h_2$ of problem~\eqref{VarFormCellDerivative}, respectively obtained by setting $\delta\xi' = (1,0)$ and $\delta\xi' = (0,1)$, as
\begin{equation*}\label{Computeh}
h' = \delta\xi_1\, h_1' + \delta\xi_2\, h_2'.
\end{equation*}
As a result, writing $\delta\xi'$ as a column vector $\delta\xi'=(\delta\xi_1,\delta\xi_2)^T$, the differential $\langle D\mathcal U(\xi'),\delta\xi' \rangle $ defined by~\eqref{InitialDefDU} admits a matricial representation of the form
\begin{equation}\label{MatrixRepDU}
	\langle D\mathcal U(\xi'),\delta\xi' \rangle = A_{\xi'}\delta\xi'
\end{equation}
where for $j=1,2$, column $j$ of the $2\times 2$ matrix $A_{\xi'}$ is the integral over $Y_f$ of the vector $h_j$ introduced above, written as a column vector as well.

We can now express the differential of $\mathcal F$ at a pressure $P$, in the direction $\delta P$, as the linear form $\mathcal L:=\langle D\mathcal F(P),\delta P\rangle$ defined by
\begin{equation*}
	\forall \psi\in H^1(\omega)\quad \mathcal L(\psi)=-\int_{\omega}\langle D\mathcal U(f'(x')-\nabla_{x'}P(x')),\nabla_{x'}(\delta P)(x') \rangle \cdot \nabla_{x'}\psi. 
\end{equation*}
Newton's algorithm consists then in constructing a sequence $P_k\in H^1(\omega)\cap L^2_0(\omega)$ as follows: $P_0$ is given, and if $P_k$ has been computed at step $k$, it is updated by solving the equation
\begin{equation}\label{NewtonFcal}
\langle D\mathcal F(P_k),\delta P\rangle = -\mathcal F(P_k)
\end{equation}
of unknown $\delta P\in H^1(\omega)$, and by setting $P_{k+1}=P_k + \delta P - \frac{1}{|\omega|}\int_{\omega}(P_k + \delta P)$, where $|\omega|$ is the area of the domain $\omega$. Equation~\eqref{NewtonFcal} means that for every $\psi\in H^1(\omega)$,
\begin{equation}\label{NewtonFcalDetail}
\int_{\omega}\langle D\mathcal U(f'(x')-\nabla_{x'}P_k(x')),\nabla_{x'}(\delta P)(x') \rangle \cdot \nabla_{x'}\psi = \int_{\omega} \mathcal{U}(f'(x')-\nabla_{x'}P_k(x'))\cdot \nabla_{x'}\psi(x')\,  dx'.
\end{equation}

To solve equation~\eqref{NewtonFcalDetail}, we rely on a finite element method. We consider a triangular mesh of $\omega$ and for all $k$, we discretize the pressure $P_k$, the test function $\psi$ and the unknown $\delta P$ using $P^2$ elements. Therefore, $\nabla_{x'}P_k$ is affine on each triangle. Approximating the external force $f'$ by $P^1$ elements, we get that on $f'-\nabla_{x'}P_k$ is determined by its values on the vertices. As a result, we can approximate the right hand side of~\eqref{NewtonFcalDetail}, which requires to solve a finite number of cell problems~\eqref{LocalProblemNonNewtonian}, each one being associated with a different vector $\xi'$ corresponding to the value of $f'-\nabla_{x'}P_k$ at a given vertex. The resolution of the cell problems can be performed by Newton's method as well, see for instance~\cite[Section 2.10]{Saramito}. Concerning the left hand side, we argue similarly and use the matricial formula~\eqref{MatrixRepDU} to replace on every triangle the quantity $\langle D\mathcal U(f'(x')-\nabla_{x'}P_k(x')),\nabla_{x'}(\delta P)(x')\rangle$ by
\[
A_{f'-\nabla_{x'}P_k}\nabla_{x'}(\delta P)(x')
\]
where $\nabla_{x'}(\delta P)(x')$ is written in column. This allows us to obtain a manageable bilinear form for the left hand side, and we solve the corresponding weak formulation using the open source software FreeFem++~\cite{FREEFEM}.

\subsection{Numerical results}

In this section, we present some numerical results obtained by applying the method described in Section~\ref{Sec:numerical_method}. We consider two different geometries for the base $T'$ of the cylindrical inclusions:
\begin{itemize}
	\item a disk of radius $R=0.25$,
	\item an ellipse of semi-major axis $a=0.35$, aligned with the $y_1$ coordinate axis, and of semi-minor axis $b=R^2/a$.
\end{itemize}
Note that this choice of parameters guarantees that the volumes of both inclusions are identical.

\subsubsection{Resolution of the cell problem~\eqref{LocalProblemNonNewtonian}}\label{Sect:NumCellProblem}

As previously mentioned, the resolution of Equation~\eqref{NewtonFcalDetail} requires to solve a family of cell problems of the form~\eqref{LocalProblemNonNewtonian} where the vector $\xi'$ is given by the value of $f'-\nabla_{x'}P_k$ at each vertex of the mesh of the two-dimensional domain $\omega$. In practice, we construct a finite element solver for the cell problem. This solver takes the vector $\xi'$ as input and produces the vector $V'=\int_{Y_f}w'_{\xi'}$, where $w'_{\xi'}$ is an approximation of the solution of~\eqref{LocalProblemNonNewtonian} obtained by Newton's method. Since the right hand side of the cell problem is a vector in $\R^2\times \{0\}$, it is easy to see that the third component of $w_{\xi'}$ is identically null, and that the solution is symmetric with respect to the horizontal cross section $Y_f'\times\{1/2\}$ of the cell $Y_f'$. This allows us to restrict the problem to the lower half cell $Y_f'\times \{0,1/2\}$ and impose symmetric boundary conditions on $y_3=1/2$.

The corresponding meshes are represented in Fig.~\ref{Fig:MeshCellDisk} in the case of the disk, and in Fig.~\ref{Fig:MeshCellEllipse} in the case of the ellipse. Tables~\ref{Table:DiskCell} and~\ref{Table:EllipseCell} summarize the outputs associated with $\xi'=(1,0)$ and $\xi'=(\sqrt{2}/2,\sqrt{2}/2)$, for different values of the parameters $\lambda$ and $\mu$. In each case, we have fixed the other parameters to $r=1.3$, $\eta_{0}=1$ and $\eta_{\infty}=0$. This choice is motivated by the fact that, in practice, $\eta_{\infty}$ is generally negligible with respect to $\eta_0$ (see, for instance, \cite{Bird}).

We can observe that for $\lambda=1$, passing from $\mu=10$ to $\mu=1$ results roughly in a multiplication of $V'$ by a factor $10$. This is not the case when one passes from $\mu=1$ to $\mu=0.1$ anymore, which suggests that the nonlinearity of the cell problem with respect to $\mu$ becomes more important as $\mu$ is reduced. Taking a much larger value of $\lambda$, namely $\lambda=1000$, it is clear that the Carreau viscosity has a very nonlinear behaviour with respect to the strain rate. This results in large variations of the norm of $V'$ when $\mu$ is divided by a factor $10$. Finally, let us mention that, in the case of the disk, the vector $V'$ is stricly aligned with $\xi'$. This is not the case anymore for the ellipse: for $\xi'=(\sqrt{2}/2,\sqrt{2}/2)$, the cosine of the angle between $V'$ and $\xi'$ is approximately equal to 0.89. This is coherent with the fact that the geometry of the obstacle is no longer invariant with respect to a rotation of the vector $\xi'$ in the horizontal plane.

\begin{figure}[h!]
	\centering
	\begin{minipage}{0.4\textwidth}
		\centering
		\includegraphics[width=\textwidth]{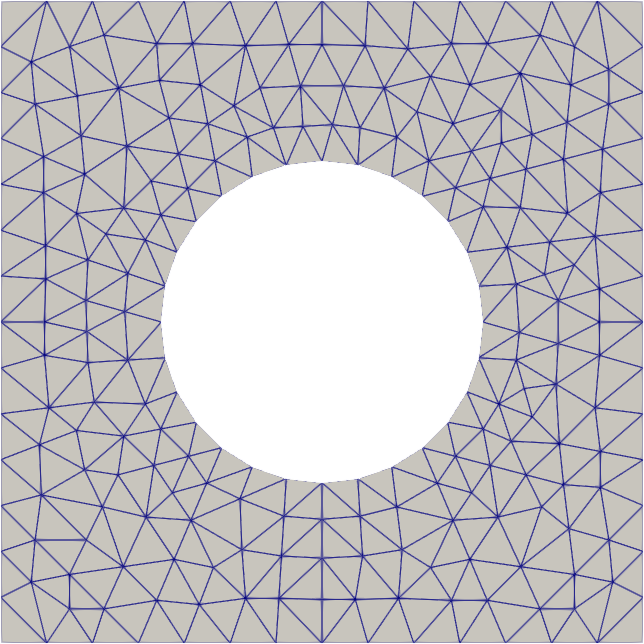}
	\end{minipage}
	\hspace{0.05\textwidth}
	\begin{minipage}{0.4\textwidth}
		\centering
		\includegraphics[width=\textwidth]{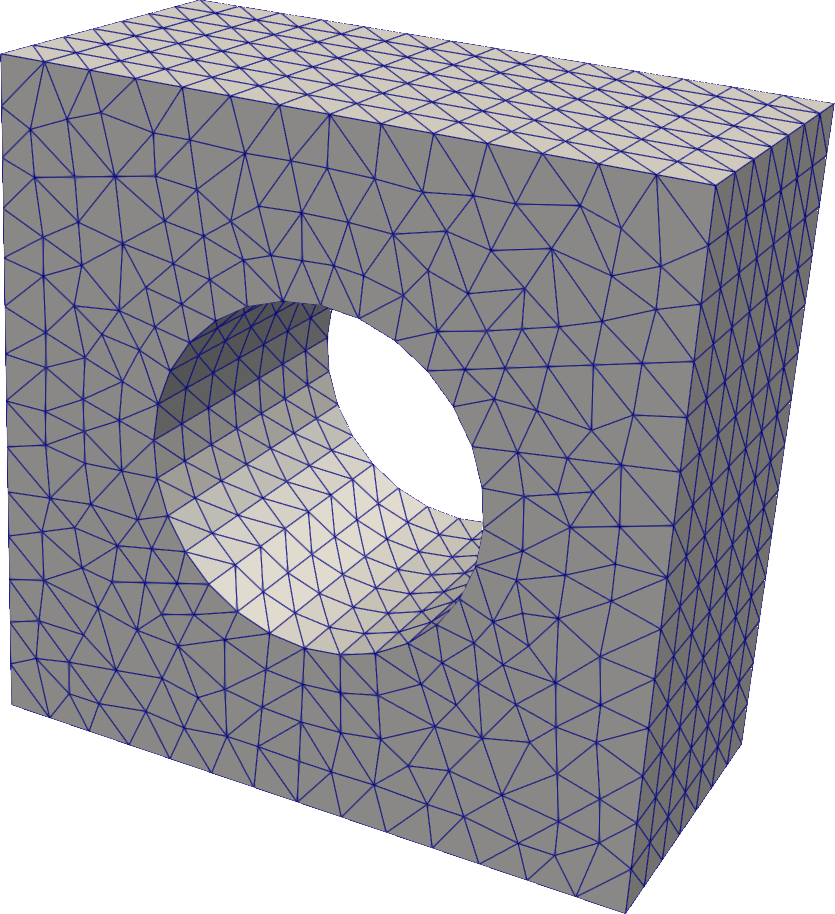}
	\end{minipage}
	\caption{Left: mesh of the 2D domain $Y_f'$ in the case of a circular-based cylindrical inclusion. Right: mesh of the lower half $Y_f'\times (0,1/2)$ of the 3D cell $Y_f$, obtained by extending the 2D mesh of $Y_f'$ in the $z$ direction using the FreeFem++ command \texttt{buildlayers}. The cell problem~\eqref{LocalProblemNonNewtonian} is solved in this half cell, using symmetric boundary condition on $y_3=1/2$. In this example, the 3D mesh contains $8736$ tetrahedra.}
	\label{Fig:MeshCellDisk}
\end{figure}


\begin{figure}[h!]
	\centering
	
	\begin{tabular}{c c}
		\begin{subfigure}{0.4\textwidth}
			\centering
			\includegraphics[width=\textwidth]{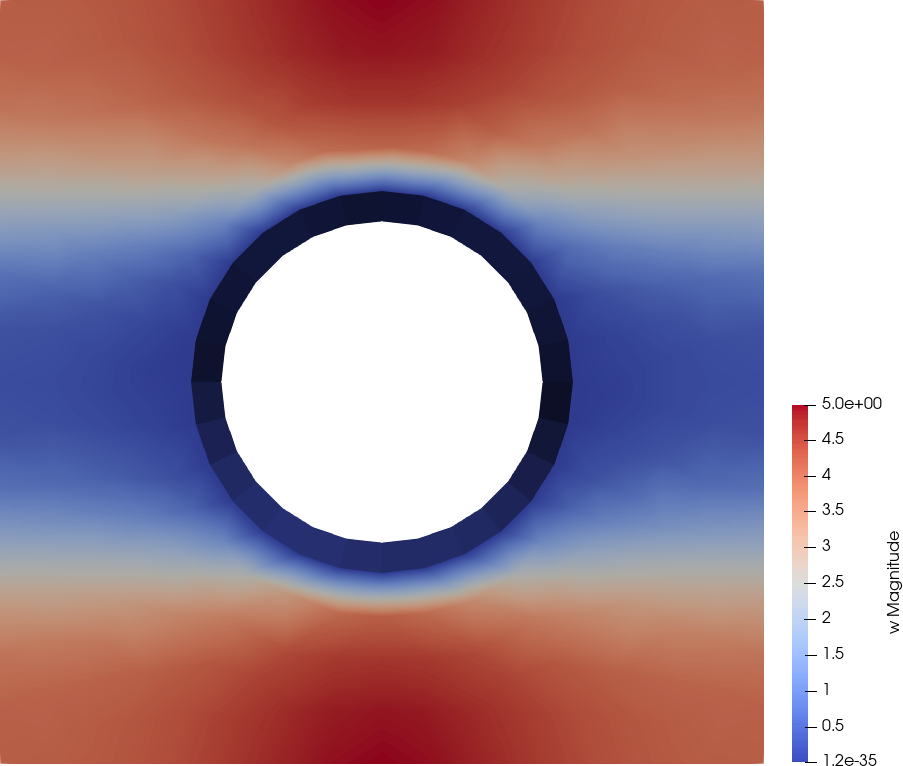}
		\end{subfigure} &
		\begin{subfigure}{0.4\textwidth}
			\centering
			\includegraphics[width=\textwidth]{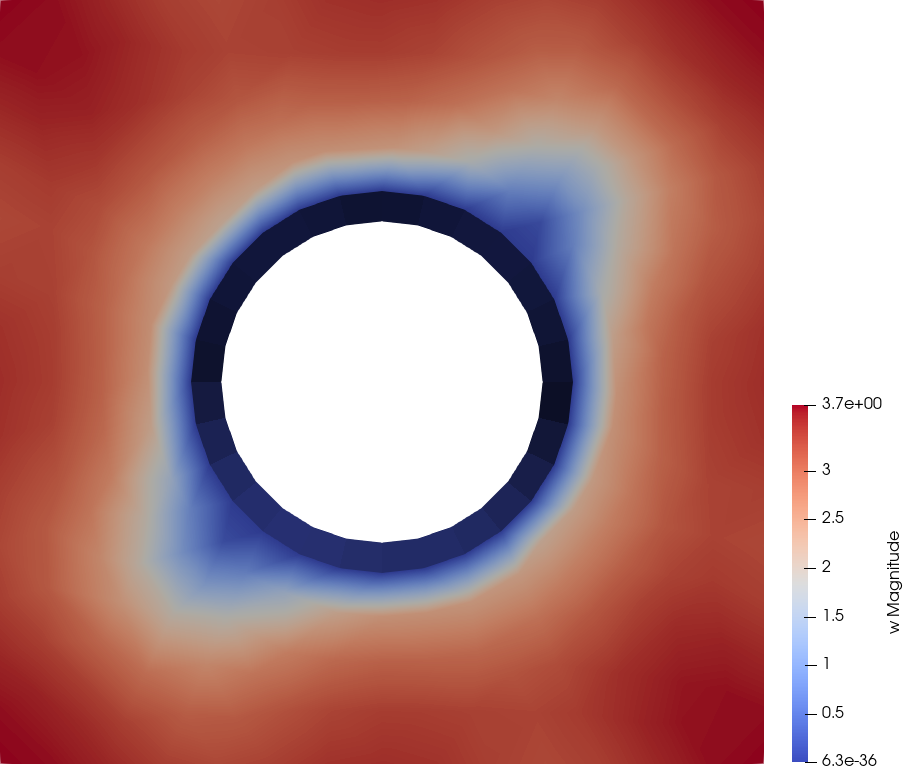}
		\end{subfigure}
		\\
		\begin{subfigure}{0.4\textwidth}
			\centering
			\includegraphics[width=\textwidth]{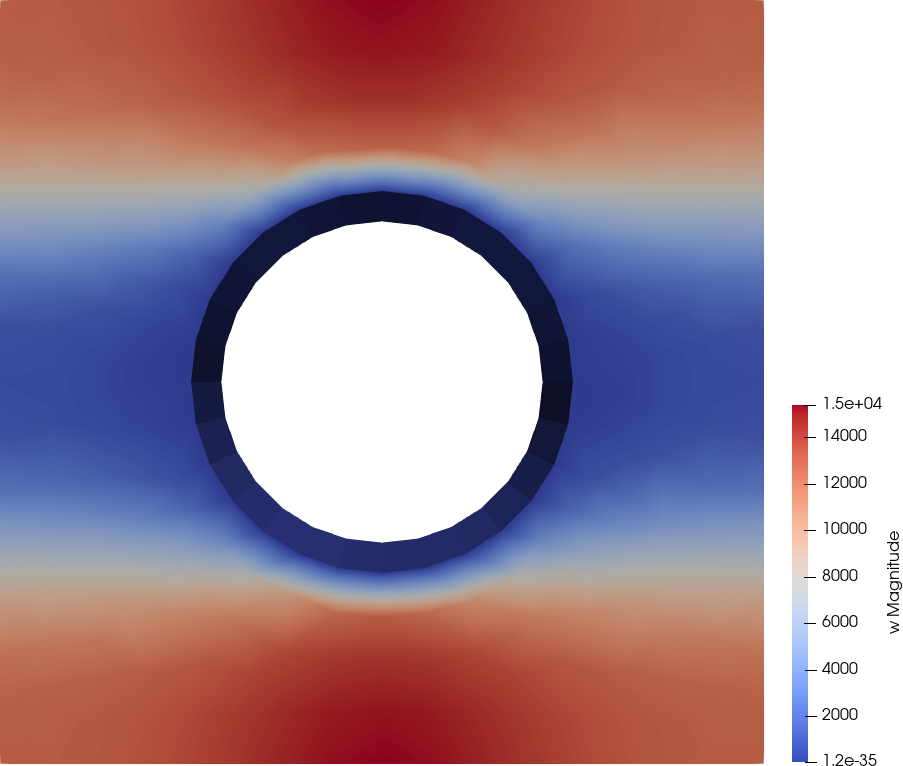}
			\caption*{$\xi'=(1,0)$}
		\end{subfigure} &
		\begin{subfigure}{0.4\textwidth}
			\centering
			\includegraphics[width=\textwidth]{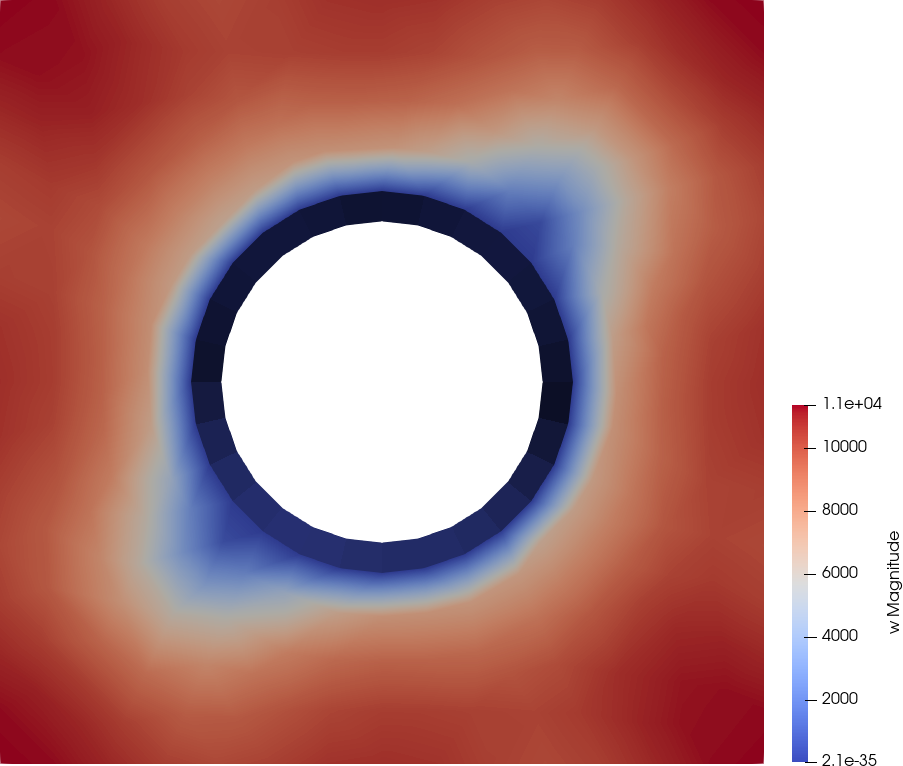}
			\caption*{$\xi'=(\sqrt{2}/2,\sqrt{2}/2)$}
		\end{subfigure}
	\end{tabular}
	
	\caption{Magnitude of the solution $w_{\xi'}$ to the cell problem~\eqref{LocalProblemNonNewtonian}, with $\mu=0.1$, in the horizontal cross section $Y_f'\times\{1/2\}$ of the cell $Y_f$, in the case of circular-based cylindrical inclusions. The parameters of the Carreau law~\eqref{Carreau} are $r=1.3$, $\eta_0=1$, $\eta_\infty=0$. First line : $\lambda=1$, second line : $\lambda=1000$. Left column : $\xi'=(1,0)$, right column : $\xi'=(\sqrt{2}/2,\sqrt{2}/2)$.}
\end{figure}


\begin{table}
	\centering
\begin{tabular}{|c|c| c||c|c|c|}
	\hline
	$\xi'$ &$\lambda$&$\mu$ & $V_1$ & $V_2$ & $\|V'\| $ \\ 
	\hline\hline
	 \multirow{6}{*}{$(1,0)$}& & 10  & 0.00251667 & -1.19492e-08 & 0.00251667\\
	  &1& 1  & 0.0258962 & -2.03842e-07 & 0.0258962 \\
	 && 0.1   &1.66673 & -0.000146684 & 1.66673 \\
	\cline{2-6}
			& & 10  & 0.00329309 & -2.07601e-07 & 0.00329309 \\
	 &1000& 1  & 2.41035 & -9.17069e-05 & 2.41035   \\
	&&  0.1   &  5192.89 & -0.196614 & 5192.89\\
	\hline\hline
	 \multirow{6}{*}{$(\sqrt{2}/2,\sqrt{2}/2)$} &  &10  & 0.00177941 & 0.00177946 & 0.0025165 \\
	   &1& 1  & 0.0181773 & 0.0181779 & 0.025707 \\
	   &  & 0.1 & 1.06628&  1.06742 & 1.50875 \\
	\cline{2-6}		& &  10  & 0.00222964& 0.00223&  0.00315344 \\
	   &1000& 1  &  1.5389 & 1.54094 & 2.17777 \\
	& &  0.1 &  3315.42 & 3319.82 & 4691.83 \\
	\hline
\end{tabular}
\caption{Examples of outputs of the cell problem solver, in the case of the disk, with $\xi'=(1,0)$ and $\xi'=(\sqrt{2}/2,\sqrt{2}/2)$ and different values of the parameters $\lambda, \mu$, the other parameters being set to $r=1.3, \eta_{0}=1$ and $\eta_{\infty}=0$.}\label{Table:DiskCell}
\end{table}


\begin{table}
	\centering
	\begin{tabular}{|c|c| c||c|c|c|}
		\hline
		$\xi'$ &$\lambda$&$\mu$ & $V_1$ & $V_2$ & $\|V'\| $ \\ 
		\hline\hline
		\multirow{6}{*}{$(1,0)$}&& 10  & 0.00355439 & -2.04336e-08 & 0.00355439 \\
		&1& 1   & 0.0365926 & -2.17653e-07 & 0.0365926 \\
		&&0.1 & 3.19015 & -0.000131535& 3.19015\\
		\cline{2-6}
		& & 10  & 0.00484256 & -6.00639e-08 & 0.00484256 \\
		&1000& 1  & 4.65322 & -0.000188363&  4.65322   \\
		&&  0.1   &  10025 & -0.405677 & 10025\\
		\hline\hline
		\multirow{6}{*}{$(\sqrt{2}/2,\sqrt{2}/2)$}&  &10  & 0.00251303 & 0.000763869 & 0.00262656 \\
		   &1& 1  & 0.0255712 & 0.0078233 & 0.0267412 \\
		&  & 0.1 & 1.34546& 0.453007& 1.41967\\
		\cline{2-6}		& &  10  & 0.00303631& 0.000973364& 0.00318851 \\
		&1000& 1  &  1.93222& 0.650889& 2.0389 \\
		& &  0.1 &  4162.78& 1402.28& 4392.63 \\
		\hline
	\end{tabular}
	\caption{Examples of outputs of the cell problem solver, in the case of the ellipse, with $\xi'=(1,0)$ and $\xi'=(\sqrt{2}/2,\sqrt{2}/2)$ and different values of the parameters $\lambda, \mu$, the other parameters being set to $r=1.3, \eta_{0}=1$ and $\eta_{\infty}=0$.}\label{Table:EllipseCell}
\end{table}

\begin{figure}[h!]
	\centering
	\begin{minipage}{0.4\textwidth}
		\centering
		\includegraphics[width=\textwidth]{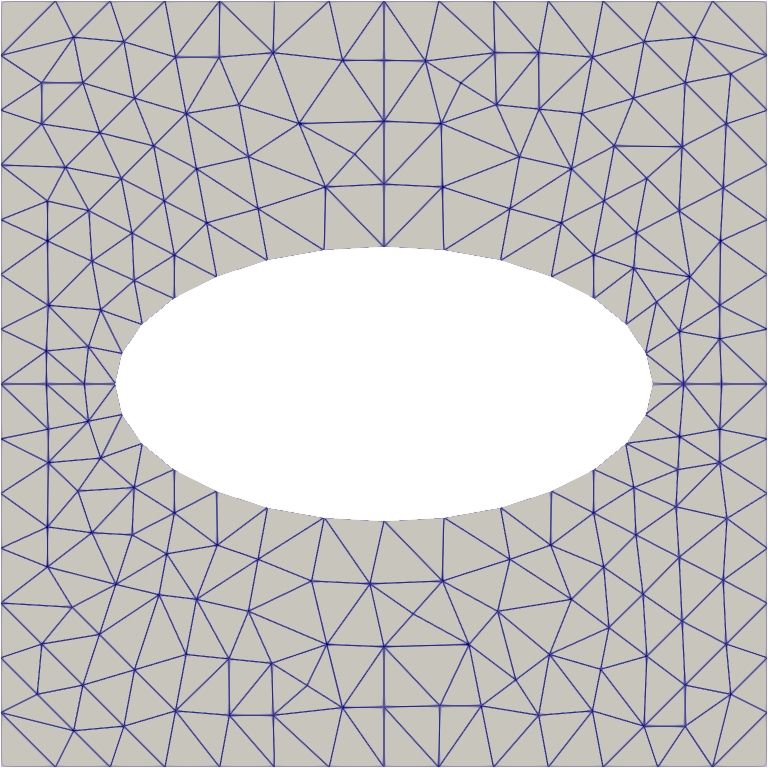}
	\end{minipage}
	\hspace{0.05\textwidth}
	\begin{minipage}{0.4\textwidth}
		\centering
		\includegraphics[width=\textwidth]{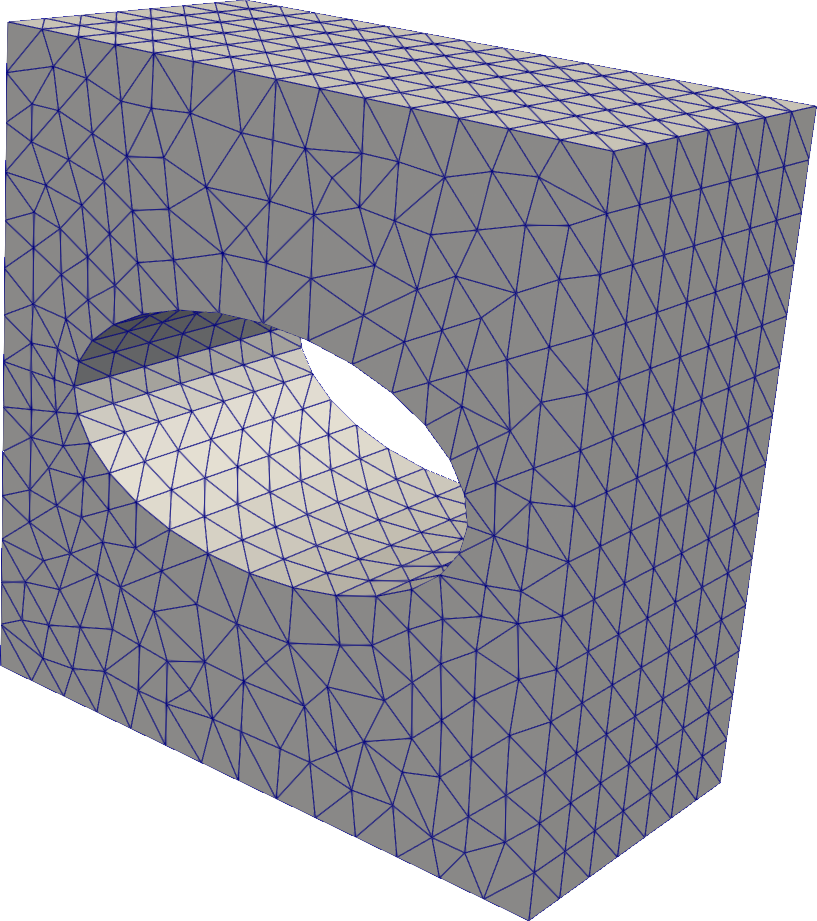}
	\end{minipage}
	\caption{Left: mesh of the 2D domain $Y_f'$ in the case of an elliptic-based cylindrical inclusion. Right: mesh of the lower half $Y_f'\times (0,1/2)$ of the 3D cell $Y_f$, obtained by extending the 2D mesh of $Y_f'$ in the $z$ direction using the FreeFem++ command \texttt{buildlayers}. The cell problem~\eqref{LocalProblemNonNewtonian} is solved in this half cell, using symmetric boundary condition on $y_3=1/2$. In this example, the 3D mesh contains $8106$ tetrahedra.}
	\label{Fig:MeshCellEllipse}
\end{figure}

\begin{figure}[h!]
	\centering
	
	\begin{tabular}{c c}
		\begin{subfigure}{0.4\textwidth}
			\centering
			\includegraphics[width=\textwidth]{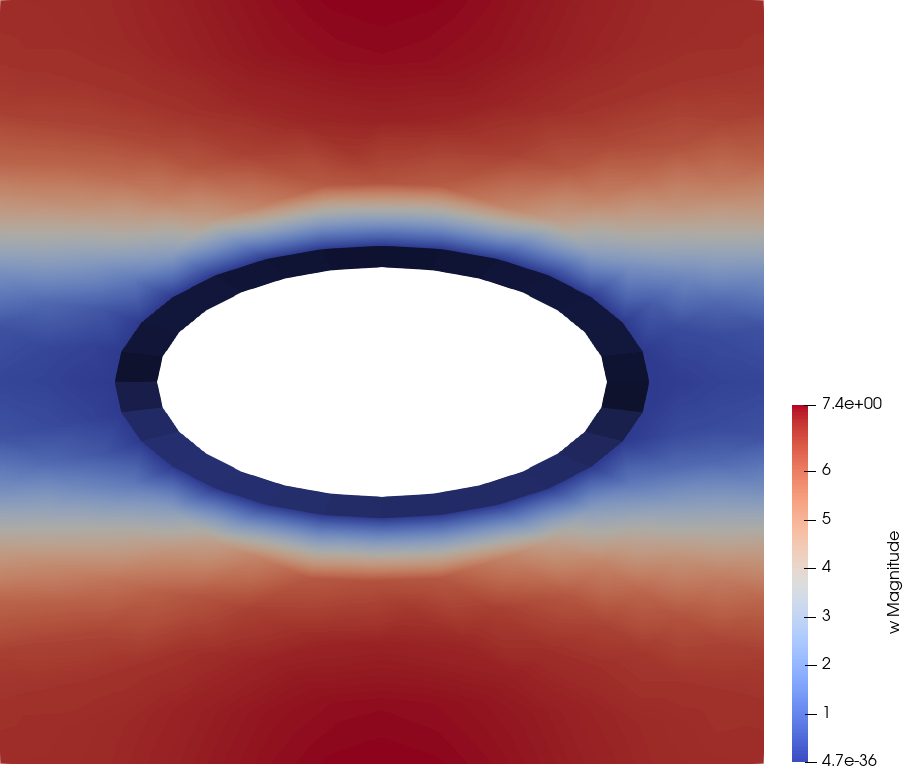}
		\end{subfigure} &
		\begin{subfigure}{0.4\textwidth}
			\centering
			\includegraphics[width=\textwidth]{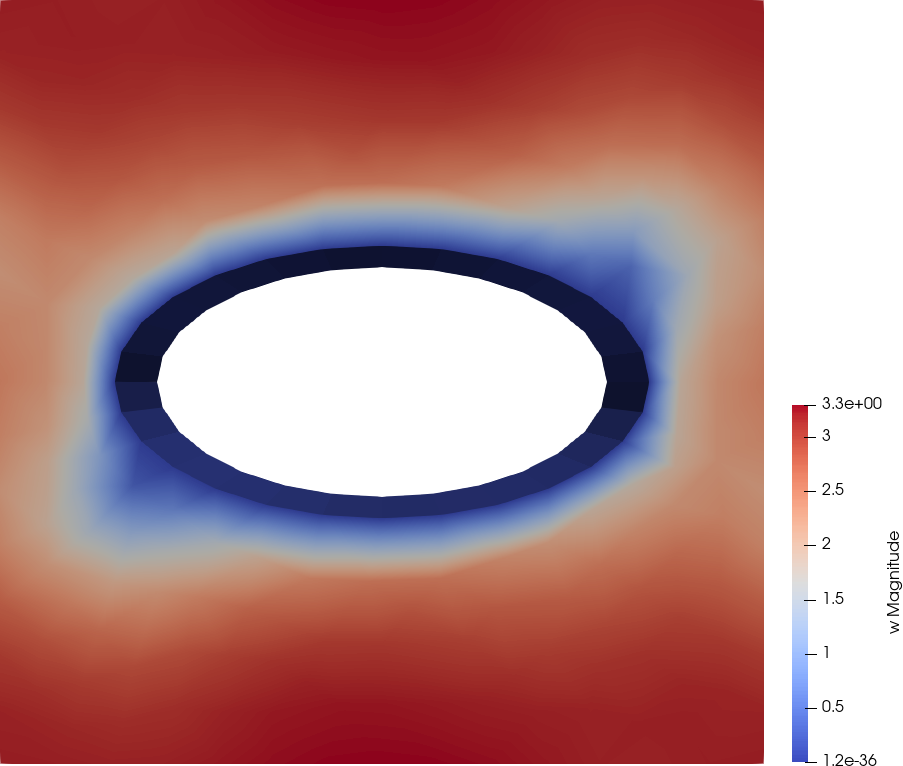}
		\end{subfigure}
		\\
		\begin{subfigure}{0.4\textwidth}
			\centering
			\includegraphics[width=\textwidth]{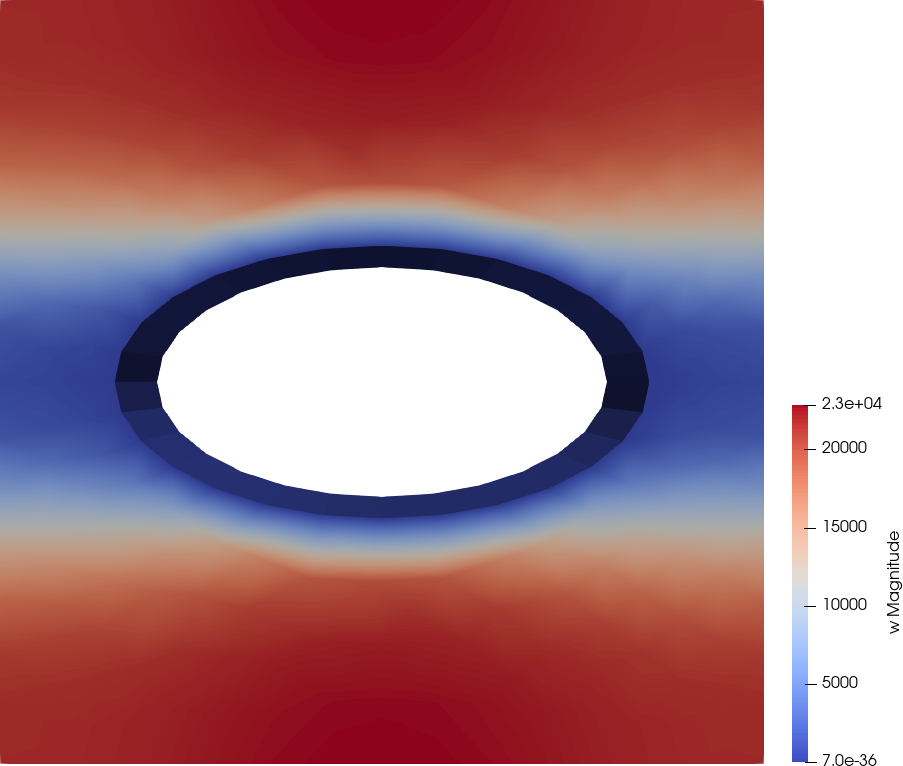}
			\caption*{$\xi'=(1,0)$}
		\end{subfigure} &
		\begin{subfigure}{0.4\textwidth}
			\centering
			\includegraphics[width=\textwidth]{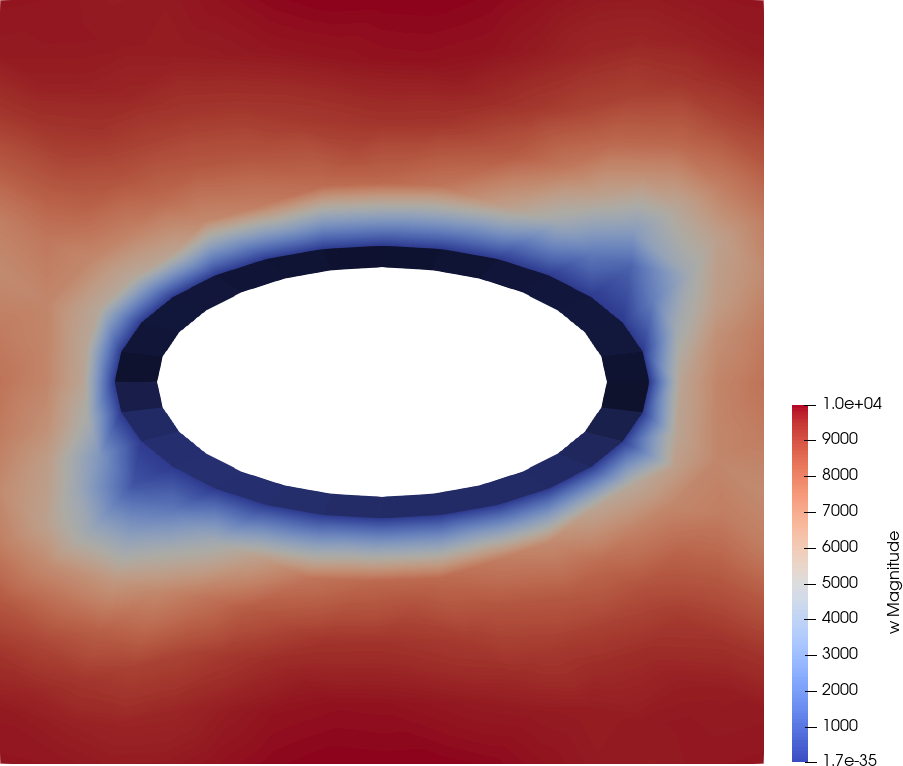}
			\caption*{$\xi'=(\sqrt{2}/2,\sqrt{2}/2)$}
		\end{subfigure}
	\end{tabular}
	
	\caption{Magnitude of the solution $w_{\xi'}$ to the cell problem~\eqref{LocalProblemNonNewtonian}, with $\mu=0.1$, in the horizontal cross section $Y_f'\times\{1/2\}$ of the cell $Y_f$, in the case of elliptic-based cylindrical inclusions. The parameters of the Carreau law~\eqref{Carreau} are $r=1.3$, $\eta_0=1$, $\eta_\infty=0$. First line : $\lambda=1$, second line : $\lambda=1000$. Left column : $\xi'=(1,0)$, right column : $\xi'=(\sqrt{2}/2,\sqrt{2}/2)$.}
\end{figure}

\subsubsection{Simulation of Darcy's law~\eqref{thm:system_gamma1}}

In this section, we simulate the limit system~\eqref{thm:system_gamma1} associated with the physical model~\eqref{2}, with a geometry inspired by the experimental setting from~\cite{Larsson}. We take $\omega=(0,1)\times(0,1/2)$ and a grid size $\vps=0.1$. We consider the case of circular-based and elliptic-based cylindrical inclusions and represent the associated domains $\omega_\vps$ and $\Omega_\vps$ in Fig.~\ref{Fig:ThinLayerDisk}. We consider the critical case $\gamma=1$, that is, that the Reynolds number $Re$ is $1/(\eps\mu)$. We take the forcing term $f'(x_1,x_2) = (x_2(0.5-x_2),0)$. The fixed parameters are: $r=1.3$, $\lambda=1$, $\eta_0=1$, $\etainf=0$.

We have represented in Figs.~\ref{Fig:DarcyNumDiskRe=1} and \ref{Fig:DarcyNumDiskRe=100} the numerical approximation of the solution of Darcy's law~\eqref{thm:system_gamma1}, with respectively $Re=1$ (which corresponds to $\mu=10$) and $Re=100$ ($\mu=0.1$). These simulations show that the pressure field is not affected by changing the Reynolds number, or more precisely, the value of parameter $\mu$ appearing in the cell problem~\eqref{LocalProblemNonNewtonian}. On the opposite, the norm of the filtration velocity $V'$ is multiplied by a factor $100$ when $\mu$ is divided by the same factor. A possible explanation is that for $f'-\nabla_{x'}P$ of order $10^{-2}$, the nonlinearity of the cell problem is not apparent for these values of $\mu$. As mentioned in Section~\ref{Sect:NumCellProblem}, in the case of circular-based inclusions, the vector $\mathcal U(\xi')$ is colinear to $\xi'$. This explains why $V'$ and $f'-\nabla_{x'}P$ seem proportional in Figs.~\ref{Fig:DarcyNumDiskRe=1} and~\ref{Fig:DarcyNumDiskRe=100}.

However, this is no longer the case if one considers an anisotropic situation, such as elliptic-based inclusions. In this case, Fig.~\ref{Fig:DarcyNumEllipseRe=1} shows noticeable differences in the distribution of the norms of $V'$ and $f'-\nabla_{x'}P$ on $\omega$. Since both vector fields are related by the relation $V'(x')=\mathcal U(f'(x')-\nabla_{x'}P(x'))$, this is linked with the fact that, as mentioned in Section~\ref{Sect:NumCellProblem}, for this shape of inclusion, the vector $\mathcal U(\xi')$ is not aligned with $\xi'$ in general.

\begin{figure}[h!]
	\centering
	\begin{minipage}{0.4\textwidth}
		\centering
		\includegraphics[width=\textwidth]{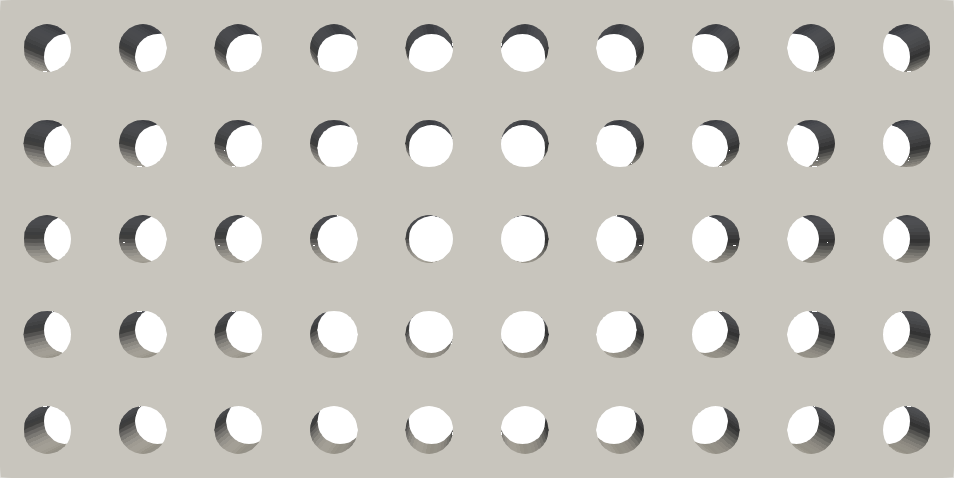}
	\end{minipage}
	\hspace{0.05\textwidth}
	\begin{minipage}{0.4\textwidth}
		\centering
		\includegraphics[width=\textwidth]{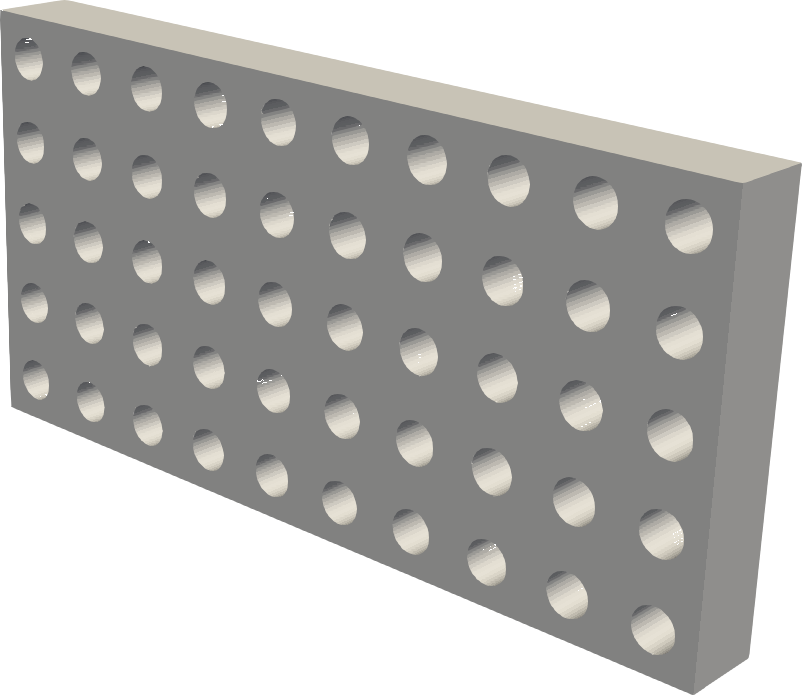}
	\end{minipage}
	\caption{Domain $\Omega_\eps$ corresponding to $\omega=(0,1)\times (0,1/2)$ and $\eps=0.1$, in the case of circular-based cylindrical inclusions.}
	\label{Fig:ThinLayerDisk}
\end{figure}

\begin{figure}[h!]
	\centering
	\begin{minipage}{0.4\textwidth}
		\centering
		\includegraphics[width=\textwidth]{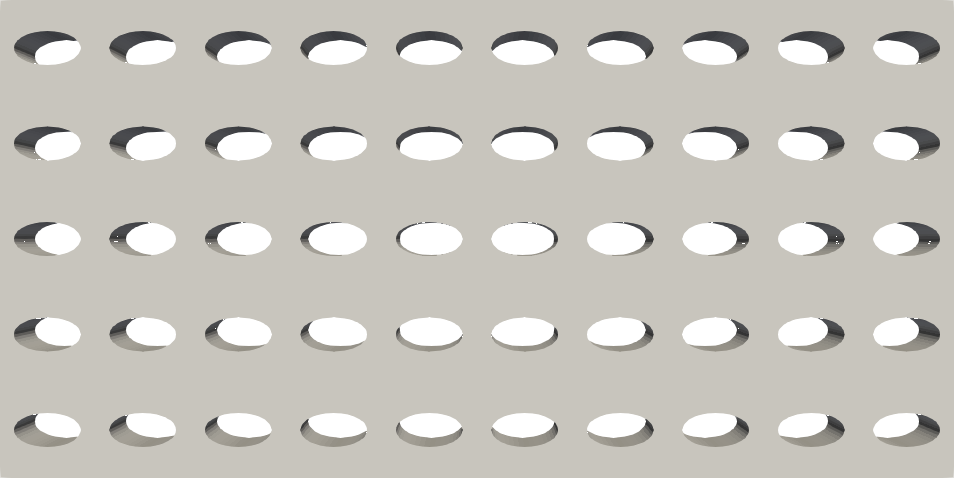}
	\end{minipage}
	\hspace{0.05\textwidth}
	\begin{minipage}{0.4\textwidth}
		\centering
		\includegraphics[width=\textwidth]{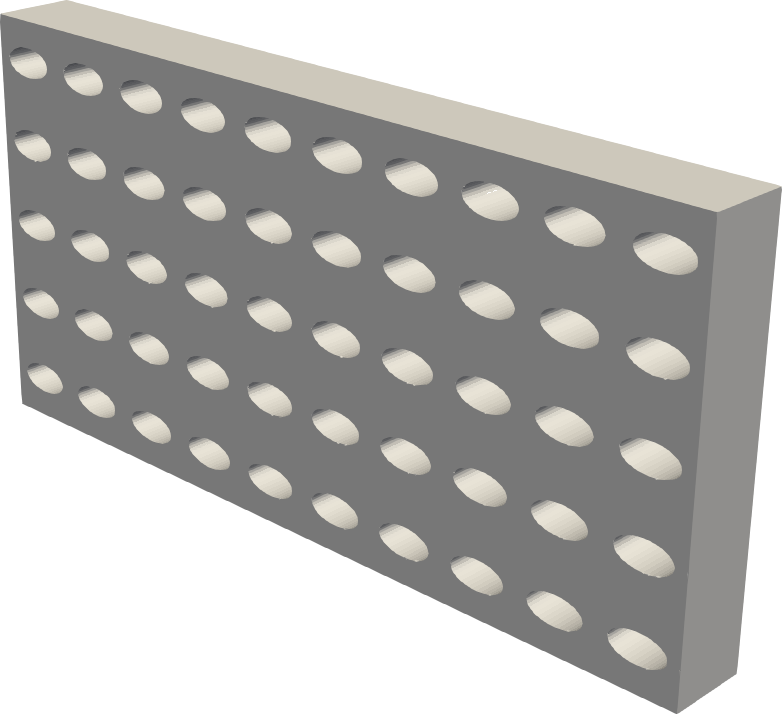}
	\end{minipage}
	\caption{Domain $\Omega_\eps$ corresponding to $\omega=(0,1)\times (0,1/2)$ and $\eps=0.1$, in the case of elliptic-based cylindrical inclusions.}
	\label{Fig:ThinLayerEllipse}
\end{figure}


\begin{figure}[htbp]
	\centering
	
	\begin{subfigure}{0.45\textwidth}
		\centering
		\includegraphics[width=\linewidth]{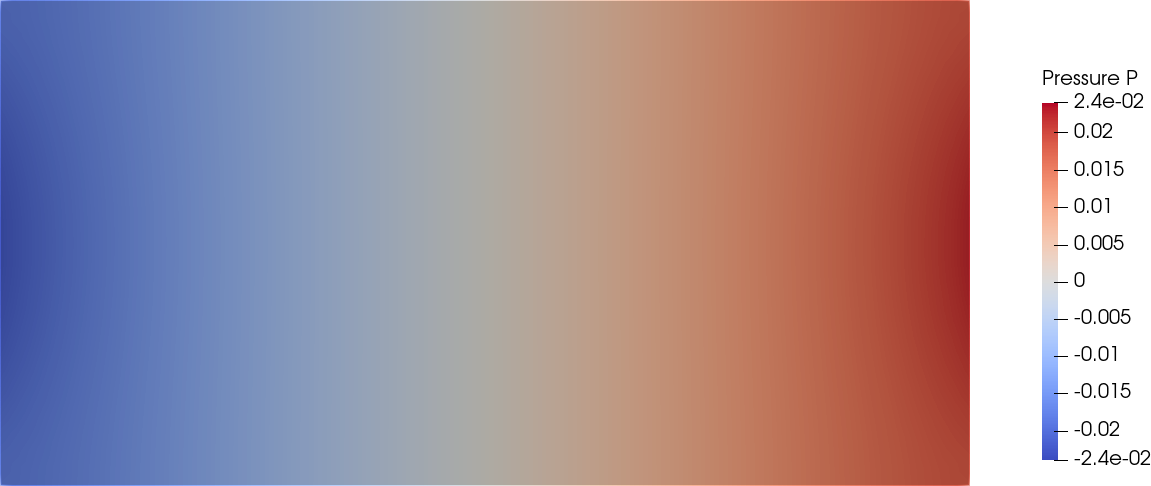}
		\caption*{$P$}
	\end{subfigure}\hfill	
		\begin{subfigure}{0.45\textwidth}
		\centering
		\includegraphics[width=\linewidth]{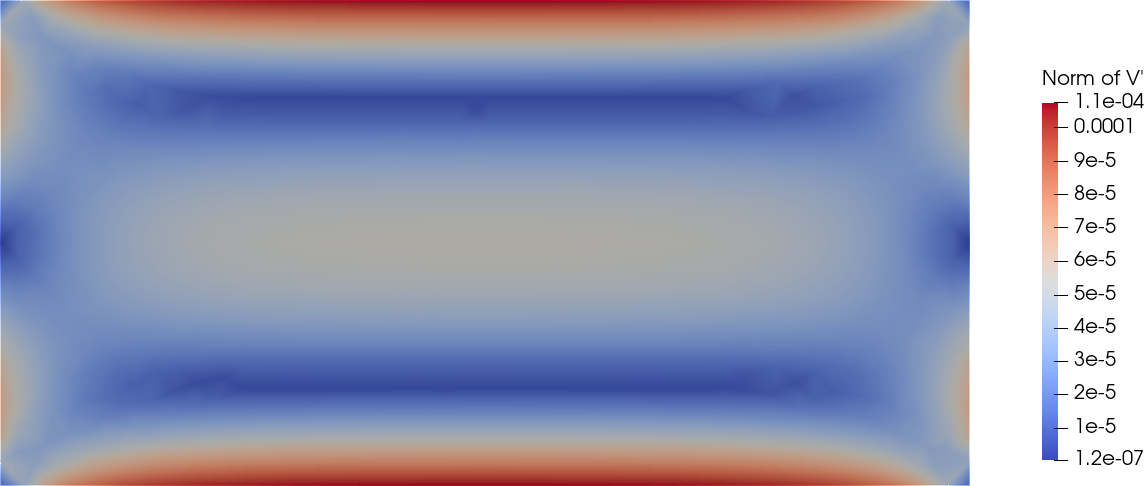}
		\caption*{$\|V'\|$}
	\end{subfigure}\hfill

	\vspace{0.5cm}

	\begin{subfigure}{0.45\textwidth}
	\centering
	\includegraphics[width=\linewidth]{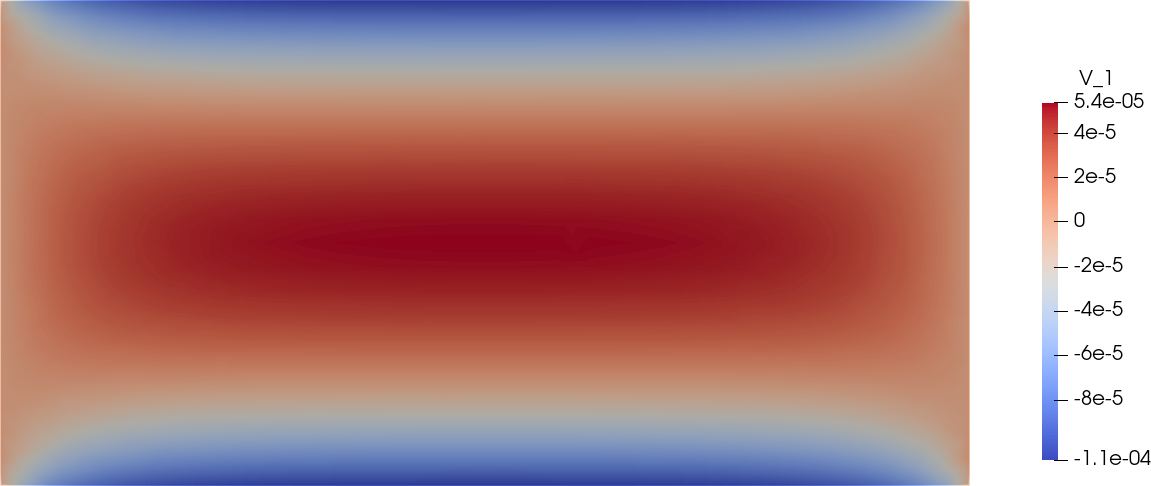}
	\caption*{$V_1$}
\end{subfigure}\hfill
\begin{subfigure}{0.45\textwidth}
	\centering
	\includegraphics[width=\linewidth]{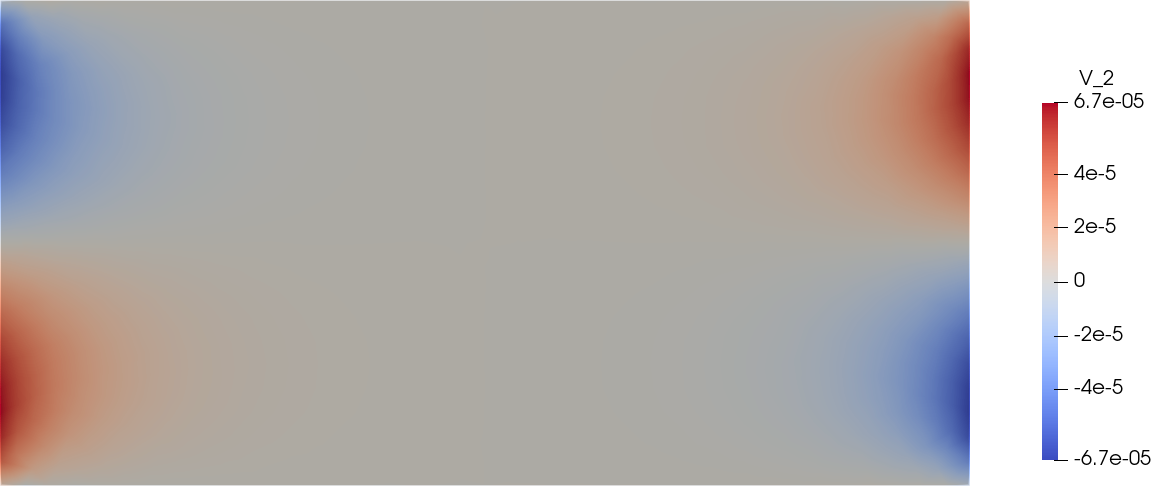}
	\caption*{$V_2$}
\end{subfigure}

		\vspace{0.5cm}
	
	\begin{subfigure}{0.45\textwidth}
		\centering
		\includegraphics[width=\linewidth]{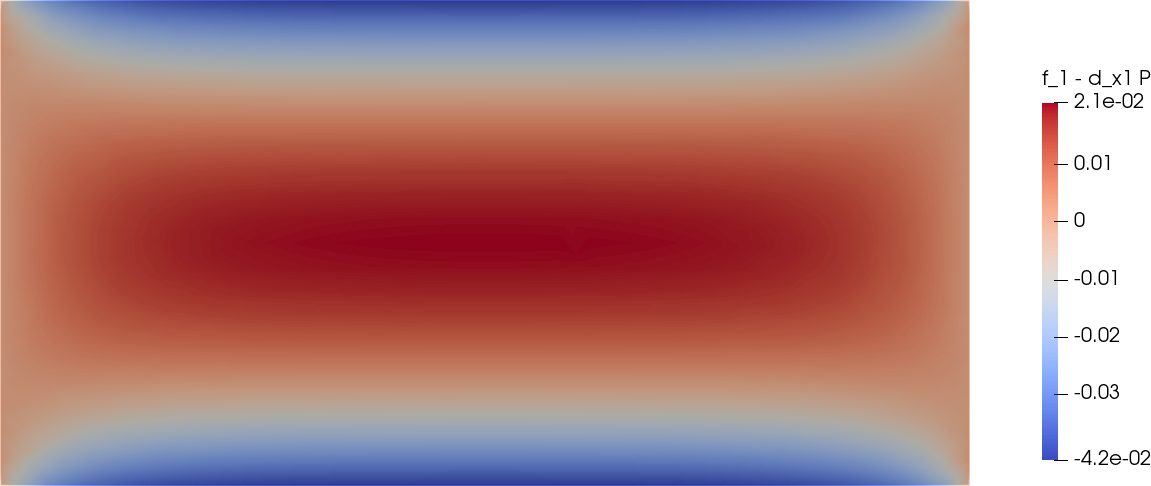}
		\caption*{$f_1-\partial_{x_1}P$}
	\end{subfigure}\hfill
	\begin{subfigure}{0.45\textwidth}
		\centering
		\includegraphics[width=\linewidth]{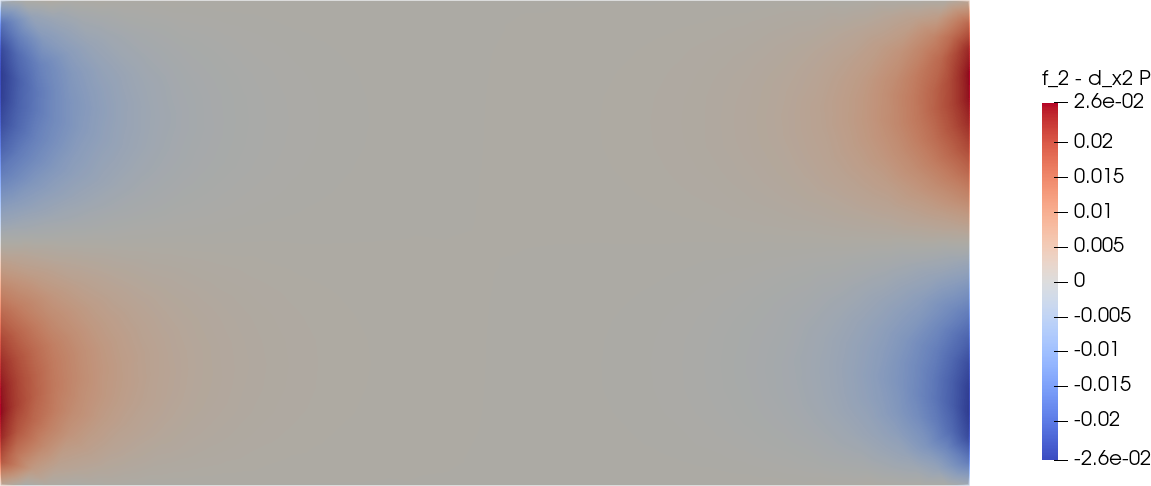}
		\caption*{$f_2-\partial_{x_2}P$}
	\end{subfigure}
	
		\vspace{0.5cm}
	
	\begin{subfigure}{0.45\textwidth}
		\centering
		\includegraphics[width=\linewidth]{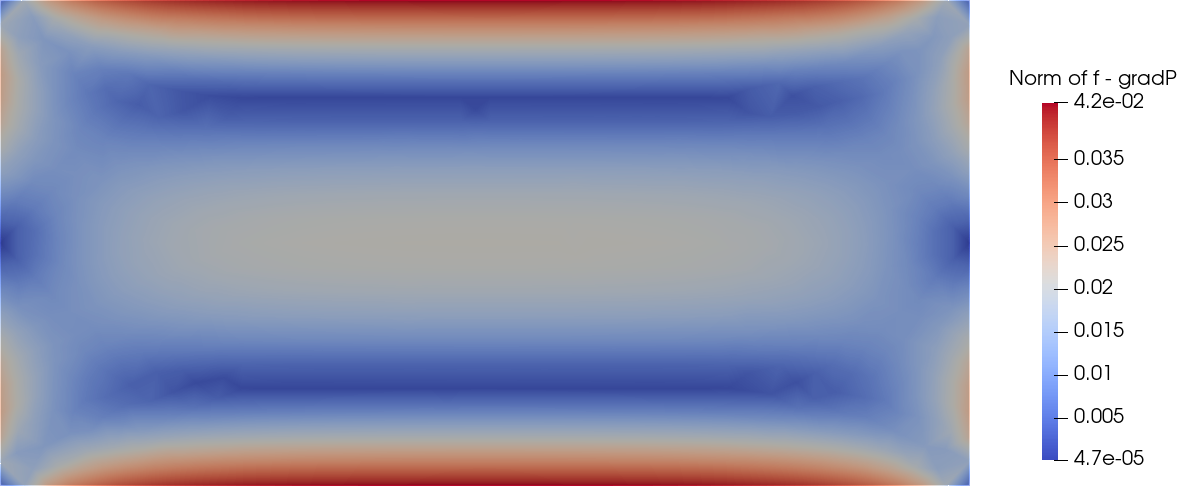}
		\caption*{$\|f'-\nabla_{x'}P\|$}
	\end{subfigure}

	\caption{ Numerical solution of Darcy's law~\eqref{thm:system_gamma1}: pressure $P$, norm, horizontal and vertical components of the filtration velocity $V'$, and horizontal, vertical components, and norm of $f'-\nabla P$. Case of circular-based inclusions and Reynolds number equal to $1$.}
	\label{Fig:DarcyNumDiskRe=1}
\end{figure}


\begin{figure}[htbp]
	\centering
	
	\begin{subfigure}{0.45\textwidth}
		\centering
		\includegraphics[width=\linewidth]{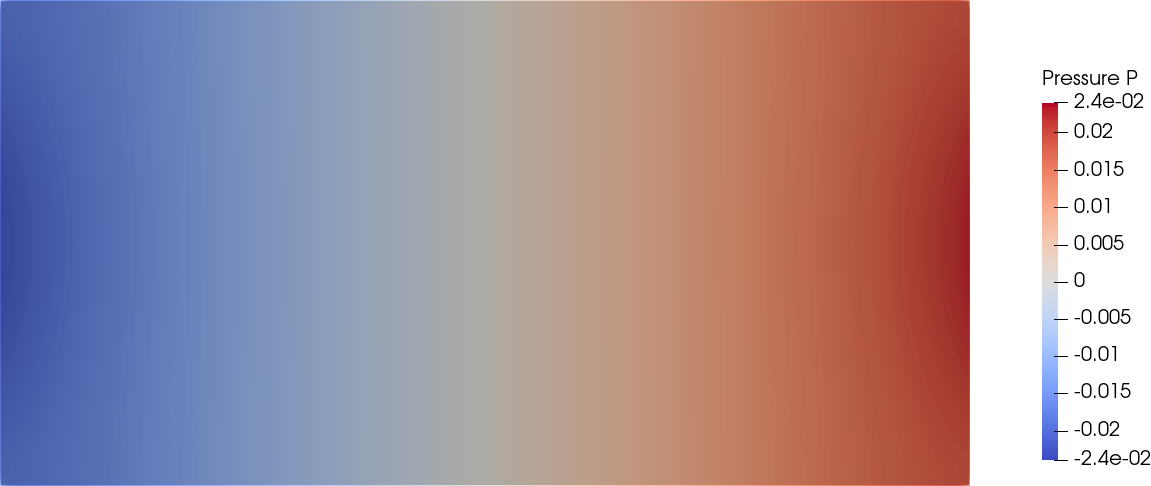}
		\caption*{$P$}
	\end{subfigure}\hfill	
	\begin{subfigure}{0.45\textwidth}
		\centering
		\includegraphics[width=\linewidth]{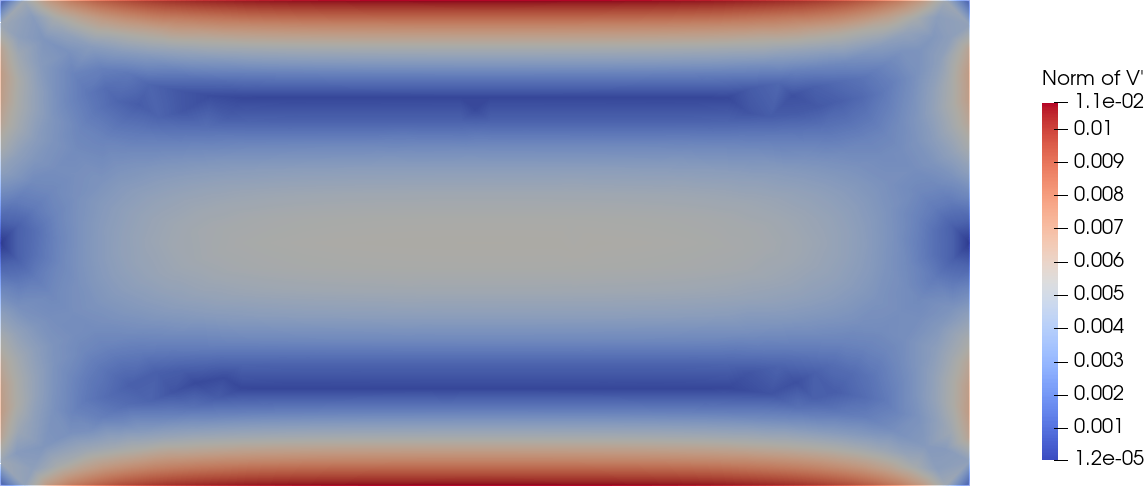}
		\caption*{$\|V'\|$}
	\end{subfigure}\hfill

	\vspace{0.5cm}

	\begin{subfigure}{0.45\textwidth}
		\centering
		\includegraphics[width=\linewidth]{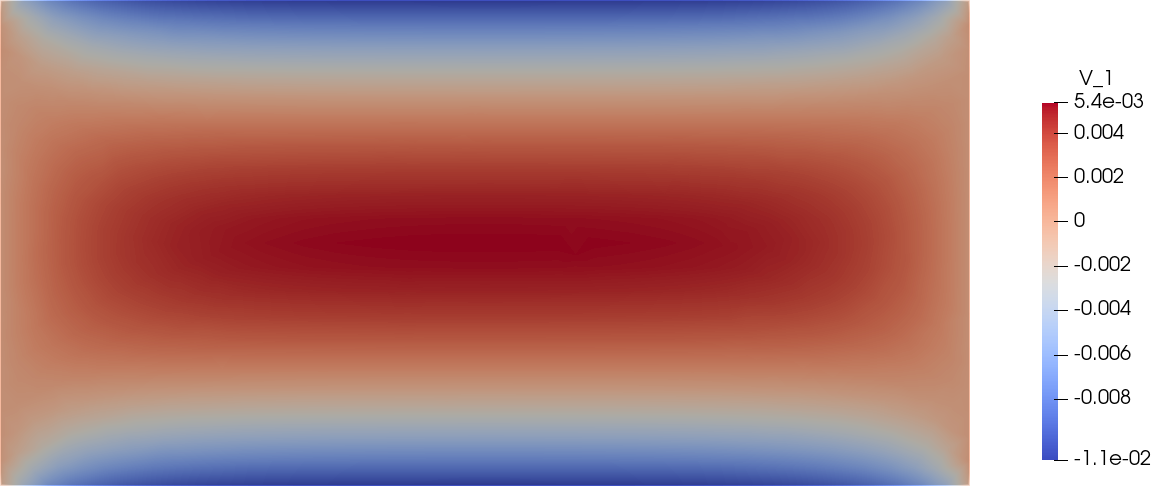}
		\caption*{$V_1$}
	\end{subfigure}\hfill
	\begin{subfigure}{0.45\textwidth}
		\centering
		\includegraphics[width=\linewidth]{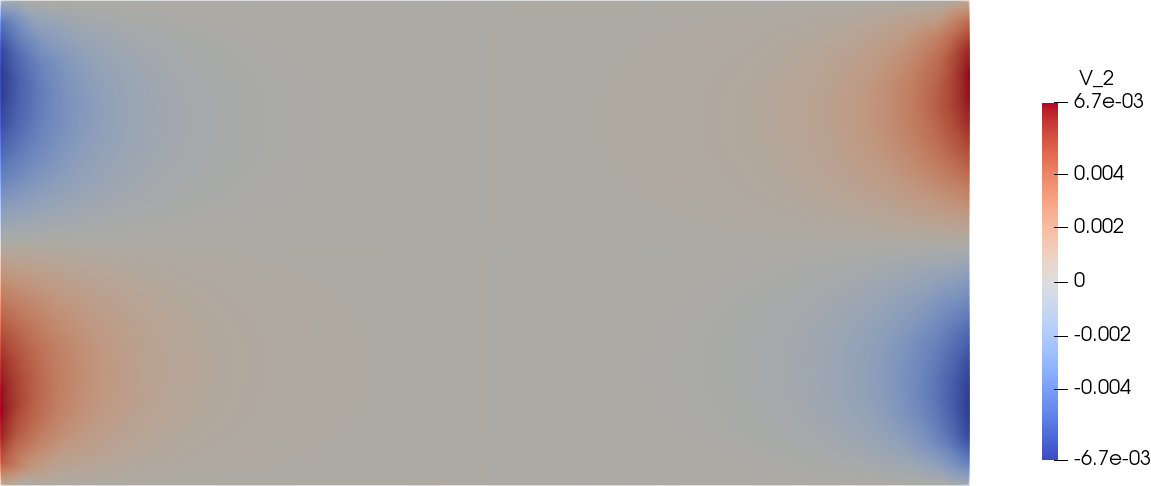}
		\caption*{$V_2$}
	\end{subfigure}
	
	\vspace{0.5cm}
	
	\begin{subfigure}{0.45\textwidth}
		\centering
		\includegraphics[width=\linewidth]{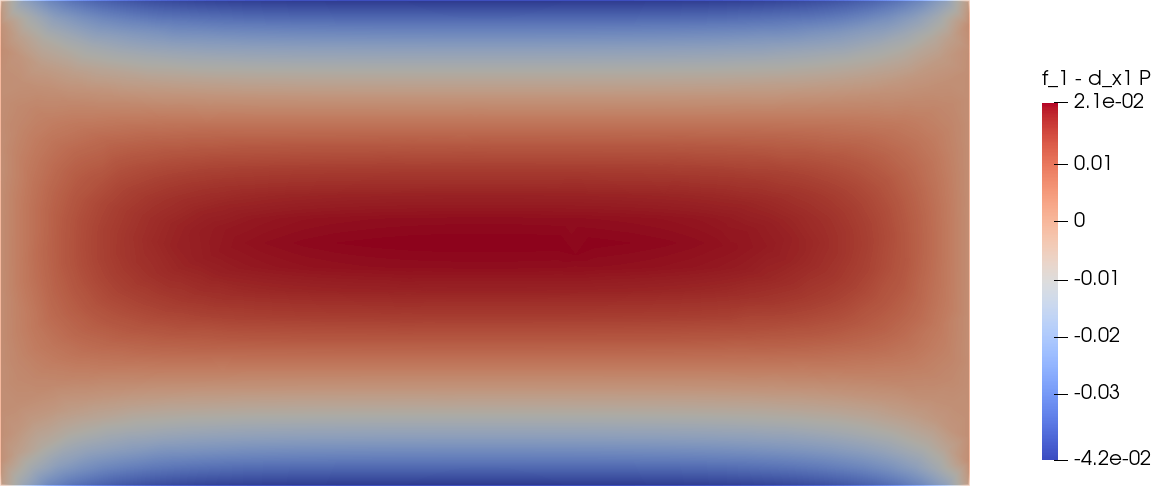}
		\caption*{$f_1-\partial_{x_1}P$}
	\end{subfigure}\hfill
	\begin{subfigure}{0.45\textwidth}
		\centering
		\includegraphics[width=\linewidth]{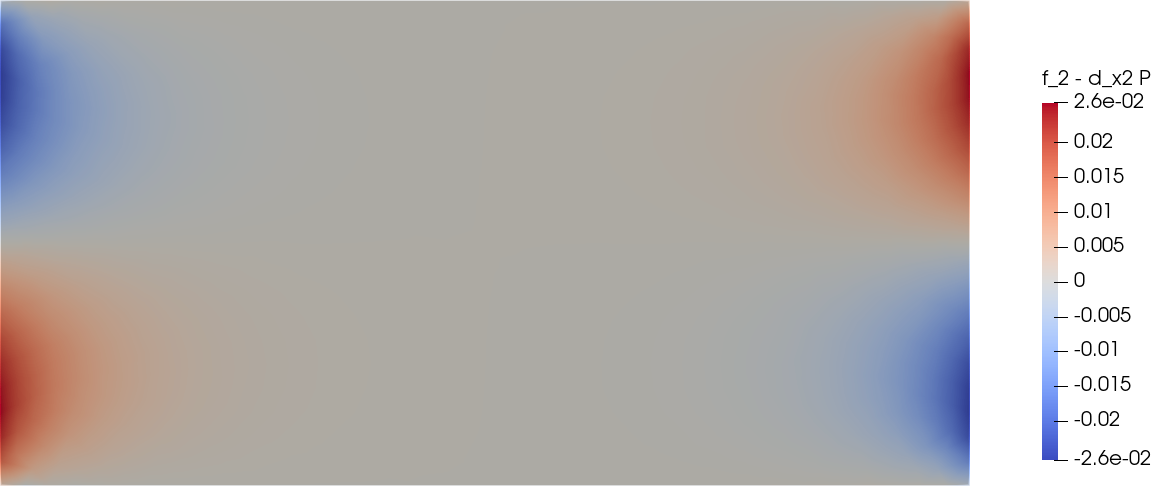}
		\caption*{$f_2-\partial_{x_2}P$}
	\end{subfigure}
	
	\vspace{0.5cm}
	
	\begin{subfigure}{0.45\textwidth}
		\centering
		\includegraphics[width=\linewidth]{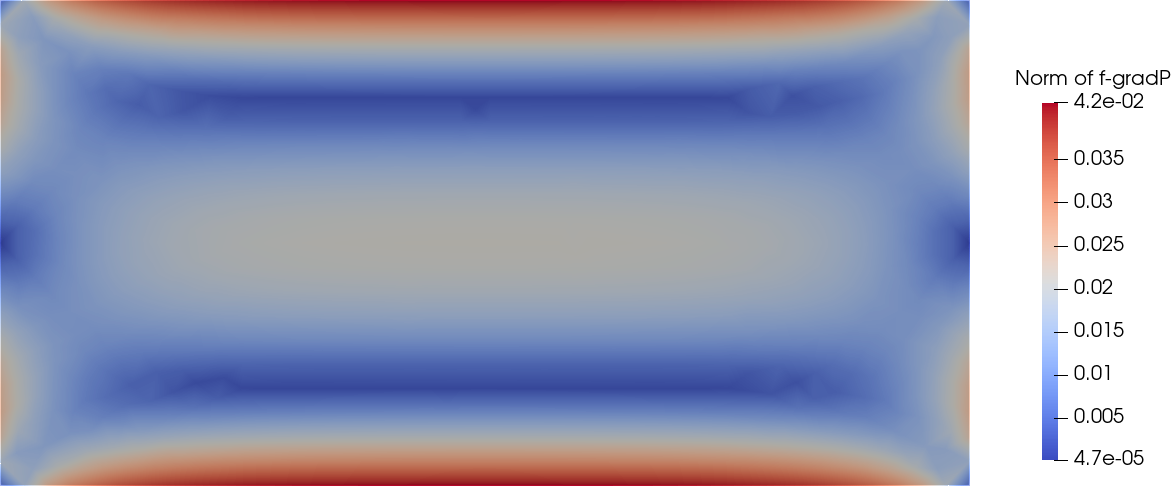}
		\caption*{$\|f'-\nabla_{x'}P\|$}
	\end{subfigure}

	\caption{Numerical solution of Darcy's law~\eqref{thm:system_gamma1}: pressure $P$, norm, horizontal and vertical components of the filtration velocity $V'$, and horizontal, vertical components, and norm of $f'-\nabla P$. Case of circular-based inclusions and Reynolds number equal to $100$.}
	\label{Fig:DarcyNumDiskRe=100}
\end{figure}


\begin{figure}[htbp]
	\centering
	
	\begin{subfigure}{0.45\textwidth}
		\centering
		\includegraphics[width=\linewidth]{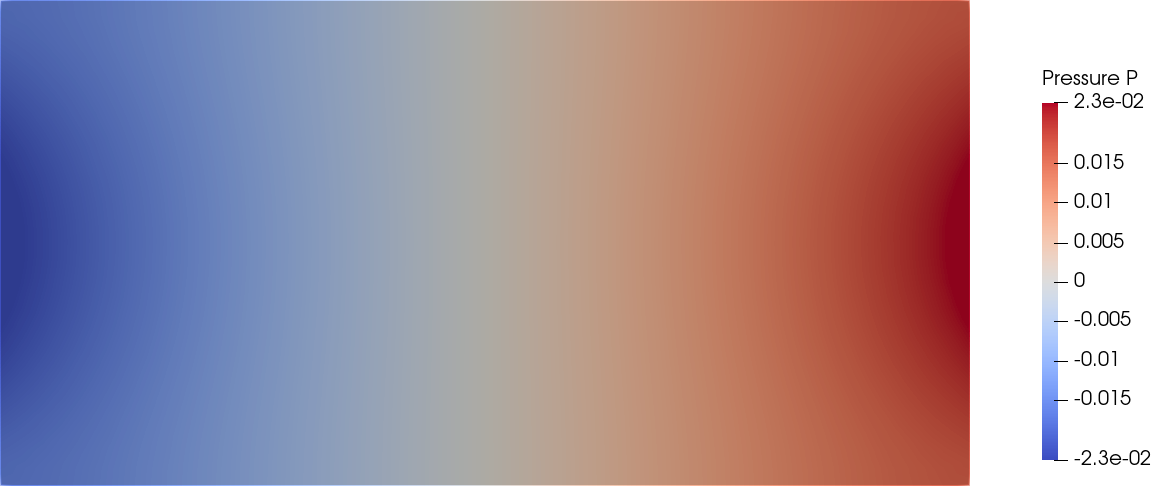}
		\caption*{$P$}
	\end{subfigure}\hfill	
	\begin{subfigure}{0.45\textwidth}
		\centering
		\includegraphics[width=\linewidth]{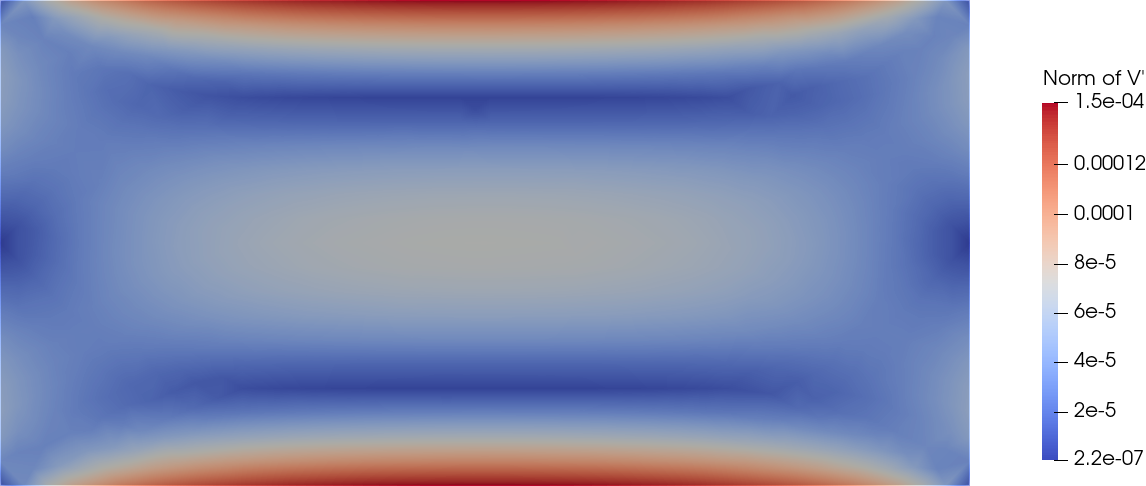}
		\caption*{$\|V'\|$}
	\end{subfigure}\hfill

	\vspace{0.5cm}

	\begin{subfigure}{0.45\textwidth}
		\centering
		\includegraphics[width=\linewidth]{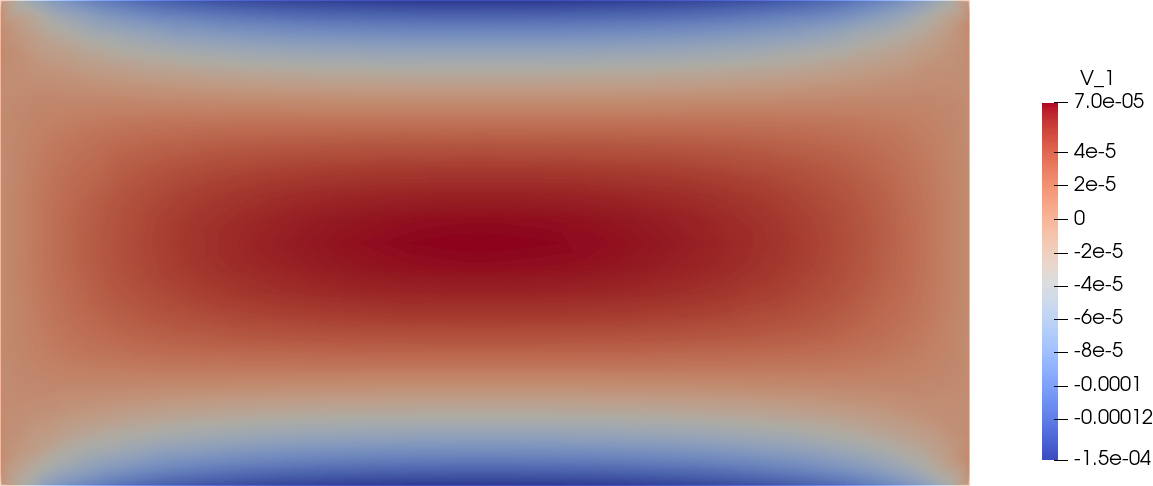}
		\caption*{$V_1$}
	\end{subfigure}\hfill
	\begin{subfigure}{0.45\textwidth}
		\centering
		\includegraphics[width=\linewidth]{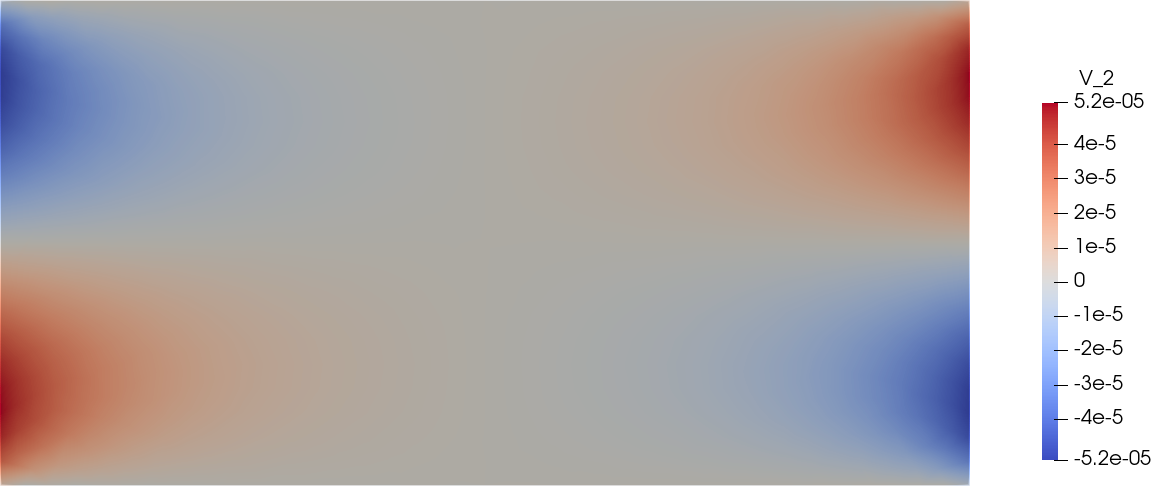}
		\caption*{$V_2$}
	\end{subfigure}
	
	\vspace{0.5cm}
	
	\begin{subfigure}{0.45\textwidth}
		\centering
		\includegraphics[width=\linewidth]{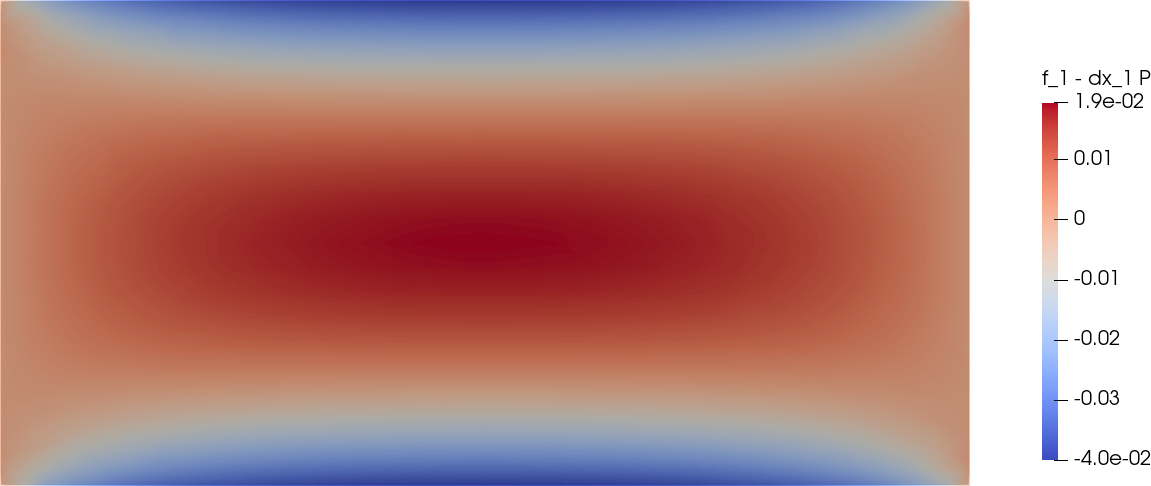}
		\caption*{$f_1-\partial_{x_1}P$}
	\end{subfigure}\hfill
	\begin{subfigure}{0.45\textwidth}
		\centering
		\includegraphics[width=\linewidth]{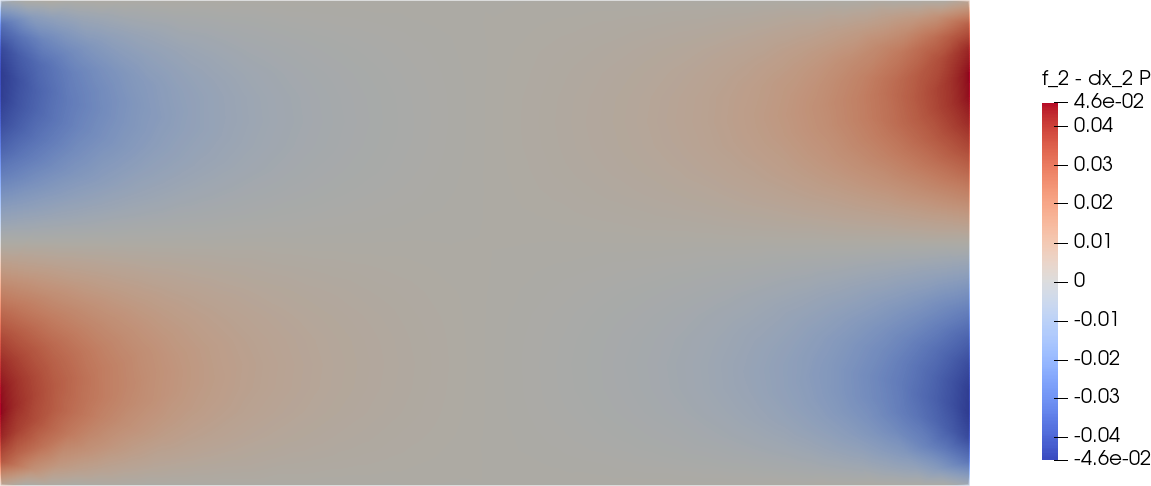}
		\caption*{$f_2-\partial_{x_2}P$}
	\end{subfigure}
	
	\vspace{0.5cm}
	
	\begin{subfigure}{0.45\textwidth}
		\centering
		\includegraphics[width=\linewidth]{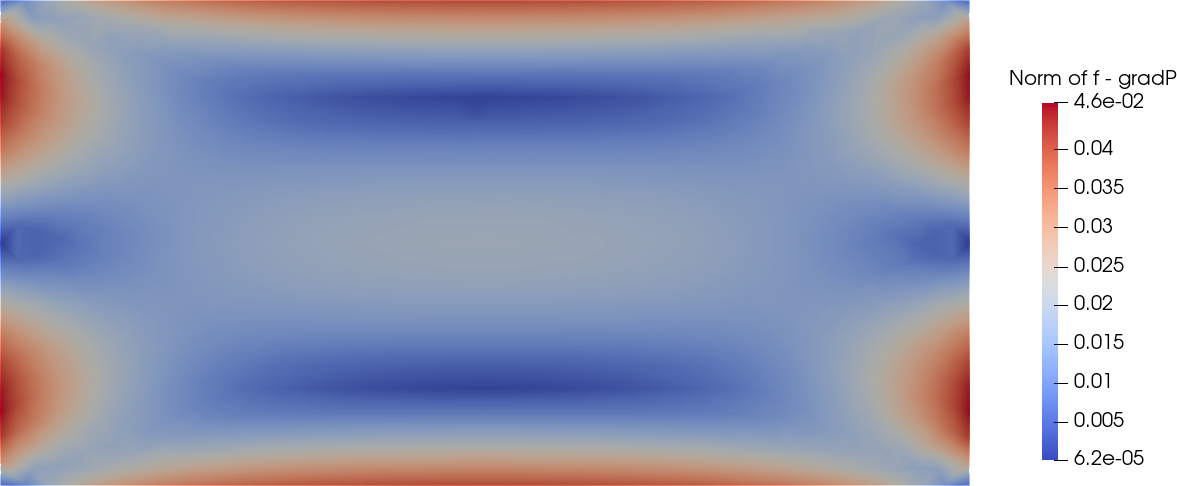}
		\caption*{$\|f'-\nabla_{x'}P\|$}
	\end{subfigure}

	\caption{Numerical solution of Darcy's law~\eqref{thm:system_gamma1}: pressure $P$, norm, horizontal and vertical components of the filtration velocity $V'$, and horizontal, vertical components, and norm of $f'-\nabla P$. Case of elliptic-based inclusions and Reynolds number equal to $1$.}
	\label{Fig:DarcyNumEllipseRe=1}
\end{figure}

\section*{Acknowledgments}

M.~Bonnivard was partially supported by the ANR Project Stoiques (ANR-24-CE40-2216).

\end{document}